\documentclass[english,a4paper,cleveref,nolineno]{socg-lipics-v2021}

\title{Computing Projective Implicit Representations from Poset Towers} 

\titlerunning{Computing PiReps from Poset Towers}

\author{Tamal K. Dey}{Department of Computer Science, Purdue University, West Lafayette, USA}{tamaldey@purdue.edu}{https://orcid.org/0000-0001-5160-9738}{Partially supported by NSF grants CCF-2437030 and DMS-2301360} 

\author{Florian Russold}{Institute of Geometry, Graz University of Technology, Graz, Austria}{russold@tugraz.at}{https://orcid.org/0009-0003-2978-0477}{Partially supported by Austrian Science Fund (FWF) grants P 33765-N and W1230}

\authorrunning{T. Dey, F. Russold} 

\Copyright{Tamal K. Dey and Florian Russold} 



\relatedversion{} 

\acknowledgements{} 

\hideLIPIcs

\usepackage{mathtools}
\usepackage{tikz-cd}
\usepackage{blkarray}
\usepackage{mathdots}

\DeclareMathSymbol{\minus}{\mathbin}{AMSa}{"39}

\newcommand{\cancel}[1]{}

\newcommand{\field}{\Bbbk}
\newcommand{\G}{\mathcal G}
\newcommand{\F}{\mathcal F}
\newcommand{\Scal}{\mathcal S}
\newcommand{\Ccal}{\mathcal C}

\def\rsf{\mathsf{r}}
\def\rrsf{\mathsf{rr}}

\DeclareMathOperator{\im}{im}

\DeclareMathOperator{\coker}{coker}
\DeclareMathOperator{\Hom}{Hom}
\DeclareMathOperator{\proj}{Proj}

\begin{document}

\maketitle

\begin{abstract}
A family of simplicial complexes, connected by simplicial maps and
indexed by a finite poset $P$, is called a \emph{poset tower}. The concept
of poset towers subsumes classical objects of study in the persistence
literature, such as multi-parameter filtrations, zigzag filtrations, and one-parameter
simplicial towers, while also allowing arbitrary finite posets and arbitrary
simplicial maps. The homology of a poset tower gives rise to a
$P$-persistence module. To compute this homology globally over $P$, in the
spirit of the persistence algorithm, we consider the homology of the chain
complex segment of $P$-persistence modules,
$C_{\ell\minus 1}\xleftarrow{\partial_{\ell}}C_\ell
\xleftarrow{\partial_{\ell+1}}C_{\ell+1}$, induced by the simplices of the
poset tower. Contrary to the case of one-critical multi-filtrations, the
chain modules $C_\ell$ of a poset tower are not necessarily projective and can have a complicated structure. In this work, we address
the problem of computing a replacement of such a chain complex segment by
projective modules and $P$-graded matrices, resulting in what we call a
projective implicit representation (PiRep), which preserves the homology.
Such a representation plays the role of the graded boundary-matrix
representation in the classical persistence algorithm: it converts the
simplicial data of a poset tower into algebraic input on which persistent
homology can be computed globally over $P$. In particular, a PiRep can be
used as input to existing algorithms for computing minimal presentations of
persistent homology. We give an efficient algorithm to compute a PiRep from a poset tower. Our
algorithm constructs degreewise minimal presentations and asymptotically
minimal second terms of projective resolutions of the chain modules
$C_\ell$, lifts the boundary maps $\partial_\ell$ to these resolutions, and
assembles this data into a PiRep using an additional correction term. The
method is tailored to the chain complexes induced by poset towers and
computes the required algebraic data combinatorially by exploiting
their special structure, avoiding general-purpose algebraic reduction. In
the context of poset towers, it is fully general and can serve as a
foundation for developing efficient algorithms on specific posets.    
\end{abstract}

\section{Introduction} \label{sec:introduction}

Persistent homology is a tool to quantify topological information of a space. After its introduction in \cite{edelsbrunner2002topological}, persistent homology has become a broad and very active field of research, with numerous applications in data science \cite{dey2022computational,edelsbrunner2010computational,DONUT}. Persistent homology tracks topological features of a space evolving over some indexing poset $P$. The spaces most commonly considered in applications are simplicial complexes. We call a family of simplicial complexes and simplicial maps, subject to commutativity relations, indexed by a finite poset $P$, a \emph{poset tower}. A poset tower can be succinctly described as a functor $K\colon P\rightarrow \mathbf{SCpx}$ from the indexing poset $P$ to the category of simplicial complexes $\mathbf{SCpx}$. Figure \ref{fig:indexing_posets} shows examples of common indexing posets and Figures \ref{fig:persistence_computation_filt}-\ref{fig:two_join_poset_indec} show examples of poset towers. The $\ell$-th homology of a poset tower, $H_\ell(K)\coloneqq H_\ell\circ K$, gives rise to a $P$-persistence module, i.e., a functor $P\rightarrow \mathbf{Vec}$ from $P$ to the category of vector spaces. It can be defined as the homology of a chain complex segment $C_{\ell+1}(K)\xrightarrow{\partial_{\ell+1}}C_\ell(K)\xrightarrow{\partial_\ell}C_{\ell\minus 1}(K)$ of $P$-persistence modules. But for a general poset tower $K$, the persistence modules $C_\ell(K)\coloneqq C_\ell\circ K\colon P\rightarrow\mathbf{Vec}$, describing the evolution of simplicial $\ell$-chains in $K$, can have a complicated structure, and the boundary maps $\partial_\ell$ cannot be globally represented by matrices. The goal of this work is to compute a chain complex segment $D_{\ell+1}\xrightarrow{d_{\ell+1}}D_{\ell}\xrightarrow{d_\ell}D_{\ell\minus 1}$ such that each $D_\ell$ is a \emph{projective module}, the maps $d_\ell$ are represented by $P$-graded matrices, and $H_\ell(K)\cong\ker(d_\ell) /\im(d_{\ell+1})$. We call such a representation a \emph{projective implicit representation}, PiRep for short, of $H_\ell(K)$. We adopt this notion from the \emph{free implicit representations} (FiRep) used in multi-parameter persistence \cite{Kerber2020FastMP,lensickwright}, but use projective modules because we work in the setting of finite indexing posets. Given a PiRep of $H_\ell(K)$, we can compute a minimal presentation of the homology module $H_\ell(K)$ using algorithms discussed in \cite{brown2024discretemicrolocalmorsetheory}.

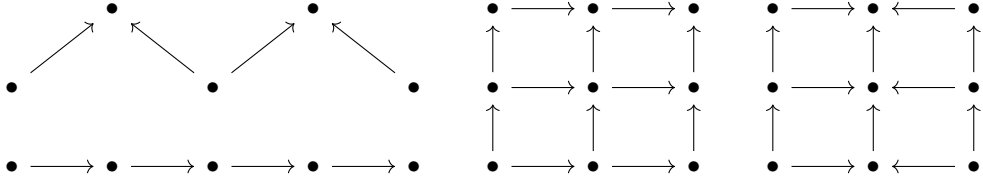
\begin{figure}[h]
\centering
\begin{tikzcd}
& \bullet & & \bullet & &[-8pt] \bullet \arrow[r] & \bullet \arrow[r] & \bullet &[-8pt] \bullet \arrow[r] & \bullet & \bullet \arrow[l]  \\
\bullet \arrow[ur] & & \bullet \arrow[ur] \arrow[ul] & & \bullet \arrow[ul] & \bullet \arrow[r] \arrow[u] & \bullet \arrow[r] \arrow[u] & \bullet \arrow[u] & \bullet \arrow[r] \arrow[u] & \bullet \arrow[u] & \bullet \arrow[u] \arrow[l]  \\
\bullet \arrow[r] & \bullet \arrow[r] & \bullet \arrow[r] & \bullet \arrow[r] & \bullet & \bullet \arrow[r] \arrow[u] & \bullet \arrow[r] \arrow[u] & \bullet \arrow[u] & \bullet \arrow[r] \arrow[u] & \bullet \arrow[u] & \bullet \arrow[u] \arrow[l] 
\end{tikzcd}    
\caption{Common indexing posets: (lower left) one-parameter persistence, (upper left) zigzag persistence, (center) two-parameter persistence, (right) two-parameter zigzag persistence.}
\label{fig:indexing_posets}
\end{figure}

\begin{figure}[h]
    \centering
    \includegraphics[width=1\linewidth,trim={5pt 5pt 5pt 5pt},clip]{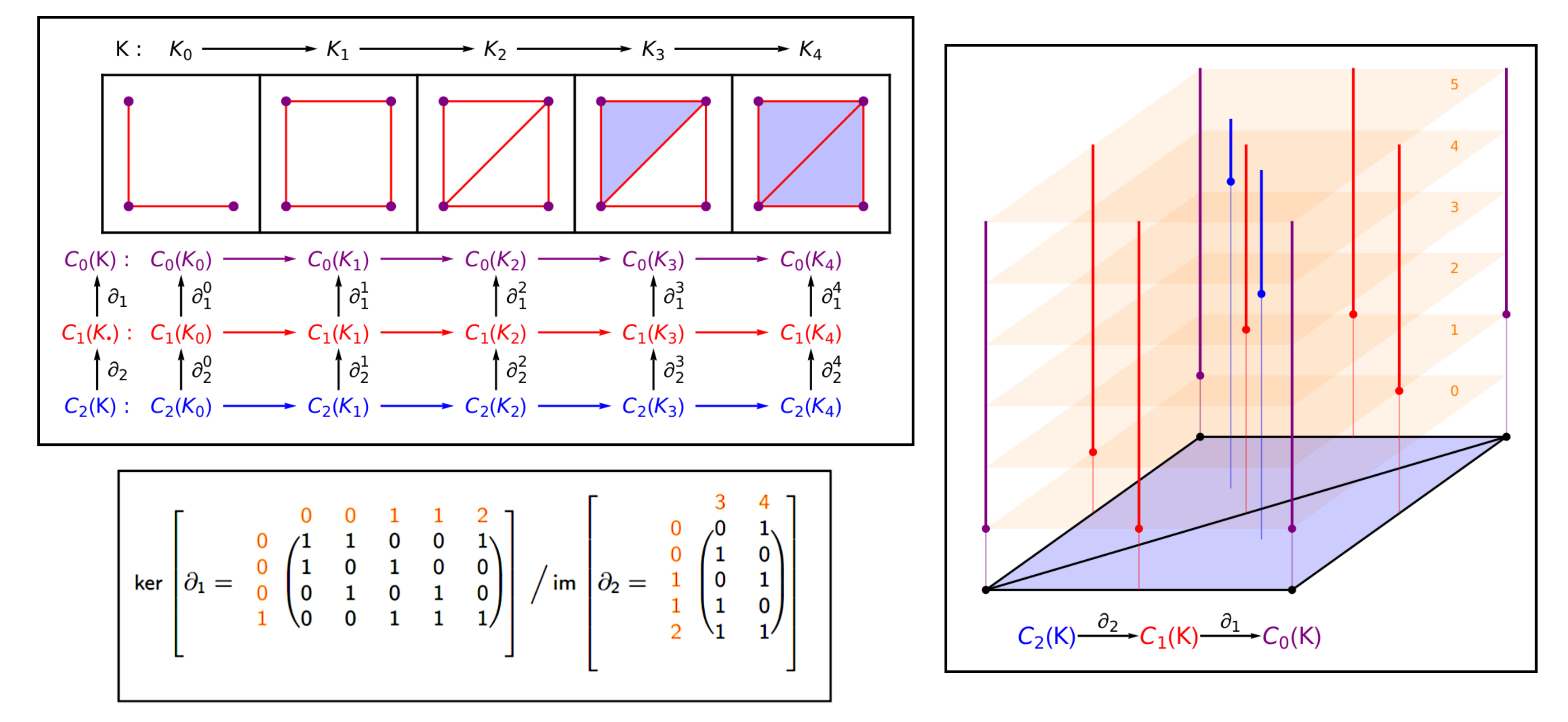}
    \caption{One-parameter persistent homology pipeline. Upper left: A filtered simplicial complex together with its algebraic representation by the simplicial chain complex. Right: The simplicial chain spaces as projective interval modules, represented by colored vertical lines, corresponding to the simplices in the filtration. Lower left: The representation of the boundary maps by graded matrices. The orange numbers on the right and in the lower left represent the grades in the filtration.}
    \label{fig:persistence_computation_filt}
\end{figure}
\textbf{From filtered simplicial complexes to poset towers.} Originally, persistent homology was introduced for filtered simplicial complexes (Figure \ref{fig:persistence_computation_filt} upper left), which are described by functors $K\colon P\rightarrow \mathbf{SCpx}$, where $P=\{1,\ldots,n\}$ with the standard order, and $K(x\leq y)$ is an inclusion for all $x\leq y\in P$. A cornerstone of the field of persistent homology is the so-called persistence algorithm. It efficiently computes an invariant, called the barcode, completely capturing the persistent homology of a filtered simplicial complex \cite{edelsbrunner2002topological,zomorodian2004computing}. It is based on the fact that each chain module $C_\ell(K)$ (Figure \ref{fig:persistence_computation_filt} upper left) is projective and can be described as a sum of upset interval modules\footnote{An interval persistence module is one-dimensional (point-wise) on a
convex connected subposet of the indexing poset $P$ and zero elsewhere,
with identity maps between comparable elements in the support.}, corresponding to the $\ell$-simplices that start at the point where the simplex appears first in the filtration (Figure \ref{fig:persistence_computation_filt} right). The significance of $C_\ell(K)$ being projective is that, in this case, we can represent the boundary maps $\partial_\ell$ by graded matrices (Figure \ref{fig:persistence_computation_filt} lower left). This matrix representation can be read off immediately from the input and allows us to compute the persistent homology as $\ker(\partial_\ell) /\im(\partial_{\ell+1})$ via matrix operations on these graded matrices.  

\begin{figure}[h]
    \centering
    \includegraphics[width=1\linewidth]{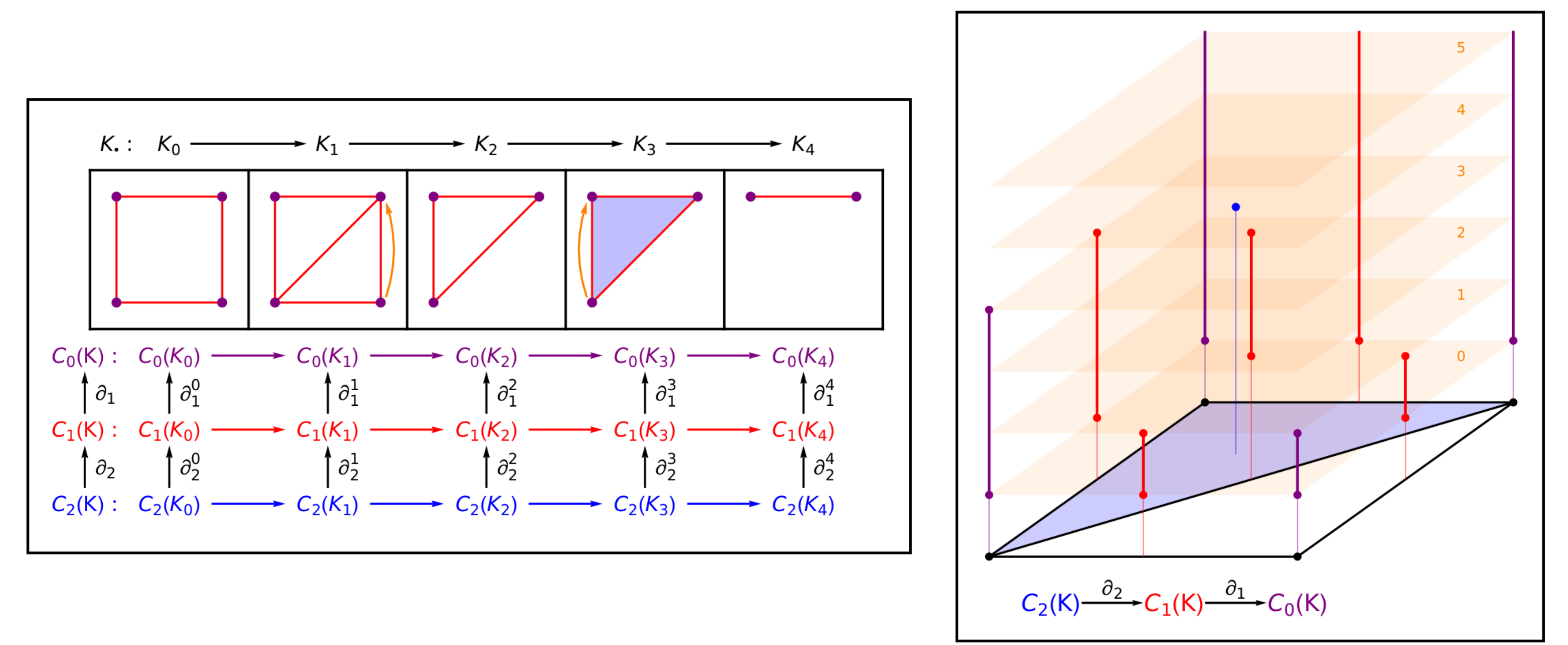}
    \caption{Left: A simplicial tower with its algebraic representation by a simplicial chain complex. The orange arrows represent simplicial collapses. Right: The chain spaces represented by interval modules corresponding to the simplices. Contrary to the case of filtrations, an interval can end if the corresponding simplex dies in the tower.}
    \label{fig:persistence_computation_tower}
\end{figure}

Since its introduction, the classical one-parameter persistent homology of filtered spaces has been generalized to various other settings and indexing posets (see Figure \ref{fig:indexing_posets}). Simplicial towers \cite{dey2014computing} generalize filtrations by allowing the maps $K(x\leq y)$ to be arbitrary simplicial maps (see Figure \ref{fig:persistence_computation_tower}). In this situation, $C_\ell(K)$ is also a sum of interval modules, where the intervals correspond to the simplices in the tower (Figure \ref{fig:persistence_computation_tower} right). The difference from the case of filtrations is that intervals can end if a simplex gets collapsed. In other words, the chain modules $C_\ell(K)$ can have relations (killing simplices) and are no longer projective. To compute the persistent homology via graded matrices, we first have to construct such a matrix representation from $K$. This can be achieved by either converting it to a filtration via geometric methods \cite{dey2014computing,kerber2019barcodes}, or via algebraic methods \cite{dey_et_al:LIPIcs.SoCG.2024.51}. We can then compute the barcode using the standard persistence algorithm.
Similar statements hold for zigzag persistence \cite{carlsson2010zigzag,dey_et_al:LIPIcs.ESA.2022.43}, which generalizes the indexing poset by allowing the arrows to point in both directions.

\begin{figure}[h]
    \centering
    \includegraphics[width=0.8\linewidth,trim={2pt 2pt 2pt 2pt},clip]{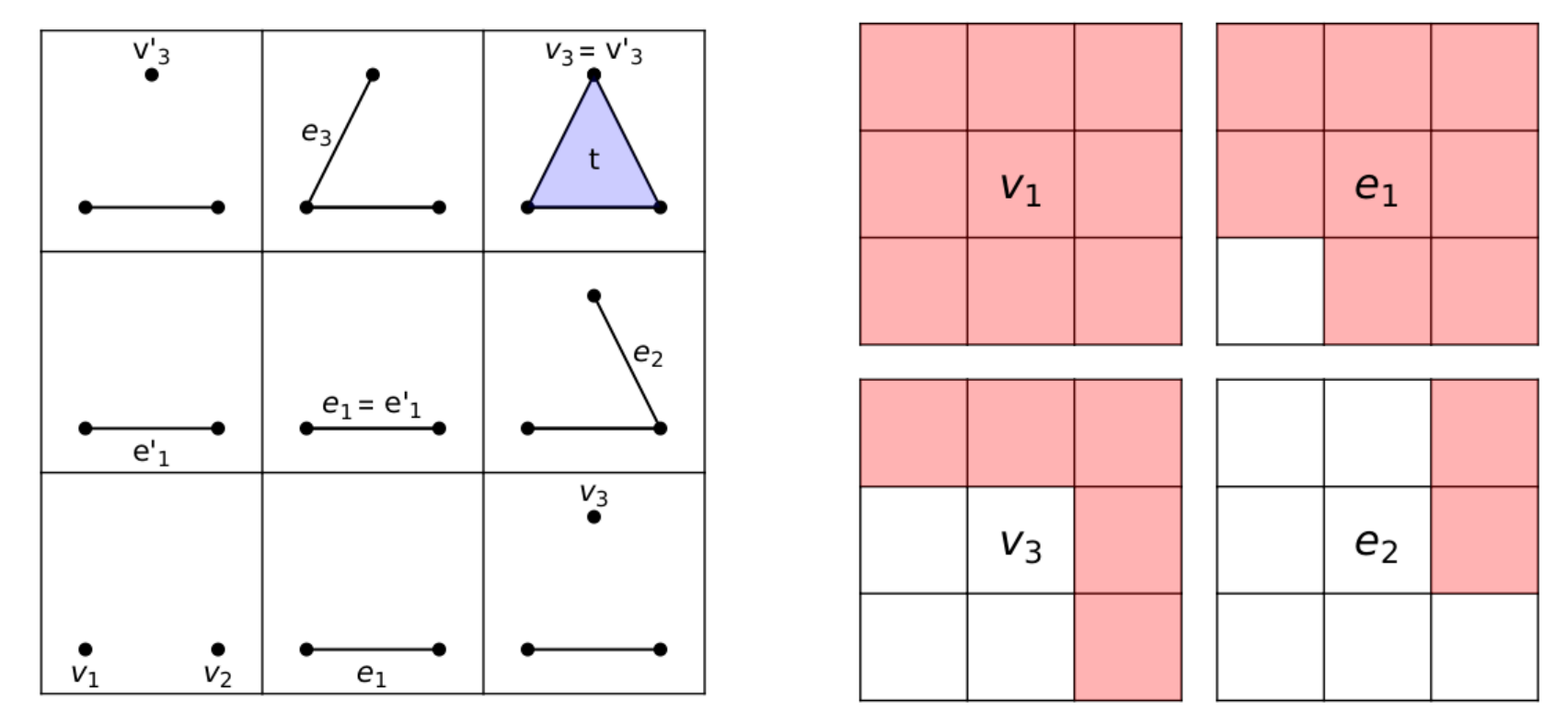}
    \caption{Left: a bifiltered simplicial complex with simplex labels shown at their generating grades. Here $v_3$ and $v'_3$ represent two copies of the same vertex that get identified in the upper right. Right: The interval modules corresponding to some of the simplices.}
    \label{fig:grid_filtration_indec}
\end{figure}

In multi-parameter persistence \cite{DBLP:conf/compgeom/BauerLL23,BL23,carlsson2007theory,miller2020}, we increase the dimension of the indexing poset to obtain $d$-dimensional grid-posets. For $d>1$, the persistent homology can no longer be described by a complete discrete invariant, like a barcode. In this case, a minimal presentation of $H_\ell(K)$ has been proposed as an alternative in the topological data analysis community~\cite{lensickwright}. For two parameters, efficient algorithms for computing this minimal presentation have been proposed
in~\cite{FK19,Kerber2020FastMP,lensickwright}. Beyond two parameters, 
an algorithm has been suggested in~\cite{multi_pres}, though full details are missing at the moment. 

These algorithms require, as input, a free/projective implicit representation of the chain complex segment $C_{\ell+1}(K)\xrightarrow{\partial_{\ell+1}}C_\ell(K)\xrightarrow{\partial_\ell}C_{\ell\minus 1}(K)$.
It is conceivable that understanding structural and computational aspects of the chain modules
$C_\ell(K)$ may become useful for computing minimal presentations of $H_\ell(K)$ more efficiently.
The chain modules $C_\ell(K)$ of multi-parameter filtrations are still decomposable into interval modules, with intervals representing the simplices (see Figure \ref{fig:grid_filtration_indec}). But simplices can be generated at different incomparable grades of the filtration and get identified at their joins. Hence, contrary to the case of one-parameter filtrations, $C_\ell(K)$ can have relations and is not necessarily projective. To obtain a valid input for the algorithm of \cite{lensickwright}, we first have to convert the simplicial input into a free/projective implicit representation by graded matrices. In \cite{chacholski2017combinatorial}, a method to convert a general simplicial filtration over a $d$-dimensional grid to a FiRep is discussed. If we consider arbitrary simplicial maps over a grid poset, as in the example in Figure \ref{fig:collapse_grid_indec}, the chain modules $C_\ell(K)$ will not even be interval decomposable anymore. Different simplices can be glued together by collapses, yielding more complicated indecomposables.  

\begin{figure}[h]
    \centering
    \includegraphics[width=0.85\linewidth,trim={2pt 2pt 2pt 2pt},clip]{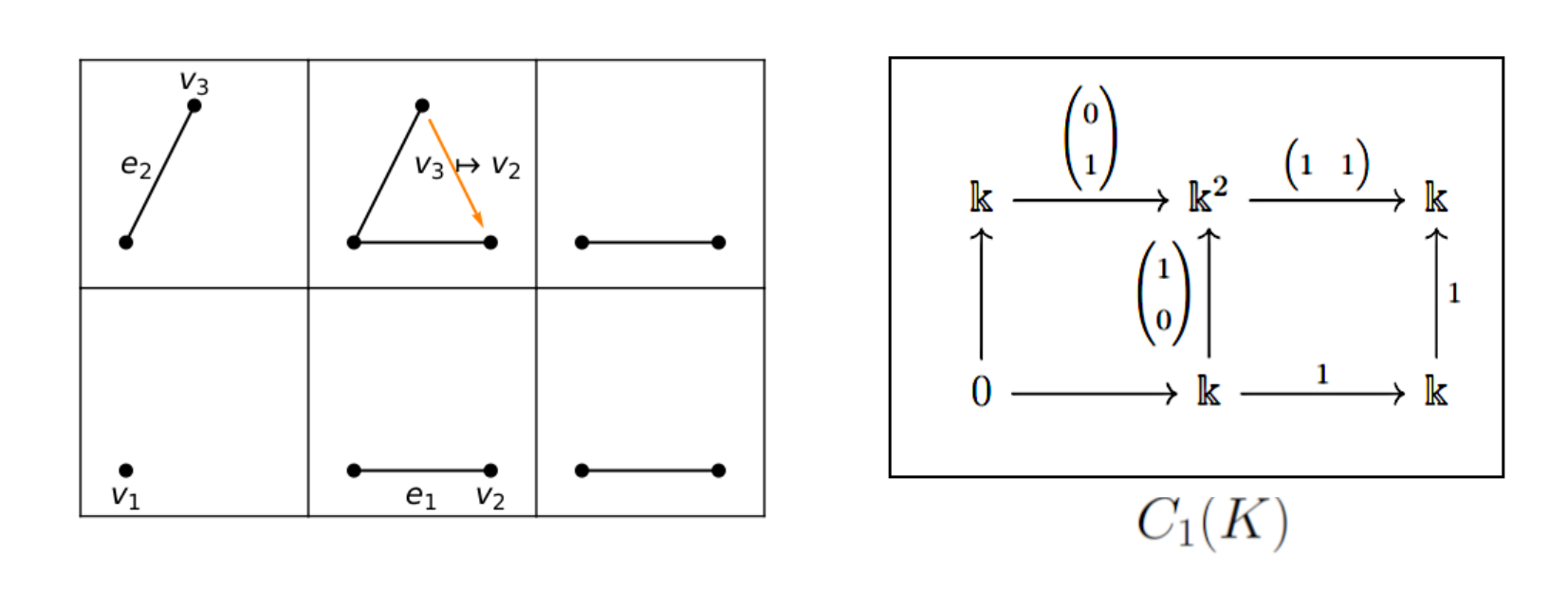}
    \caption{Left: a poset tower over a $2\times 3$-grid. The orange arrow denotes the collapse of the vertex $v_3$ into the vertex $v_2$ when moving to the right. Right: The indecomposable persistence module representing $C_1(K)$.}
    \label{fig:collapse_grid_indec}
\end{figure}

What about going beyond $d$-parameter filtrations? That leads us to poset towers. In topological data analysis (TDA), these general-purpose
frameworks have already started to appear. 
In~\cite{DS25,flammer}, zigzag-grids, that is, grids with arrows pointing in both directions along one axis (which gives quasi zigzag persistent homology as described in~\cite{DS25}), are considered. The authors in~\cite{DS25} argue that the resulting quasi zigzag persistence module
is useful to study time-varying data.
Also, persistent homology over general posets has become a subject of study \cite{kim2023persistence} in TDA. If we consider arbitrary indexing posets, the chain modules $C_\ell(K)$ will not even be interval decomposable in the case of inclusions. The example in Figure \ref{fig:two_join_poset_indec} shows a poset tower over a poset that does not have a maximum. In this case, simplices can be identified in some branches and can be different in other branches. In other words, we can no longer speak about ``the simplices'' in absolute terms, because two simplices can be the same from the perspective of one maximal element but can be different from the perspective of another maximal element. Therefore, except for the special case of one-parameter filtrations, the chain modules $C_\ell(K)$ are, in general,  not projective and can have a complicated structure. In this work, we tackle the problem of computing a projective implicit representation of the persistent homology of a general poset tower.  

\begin{figure}[h]
    \centering
    \includegraphics[width=0.95\linewidth,trim={2pt 2pt 2pt 2pt},clip]{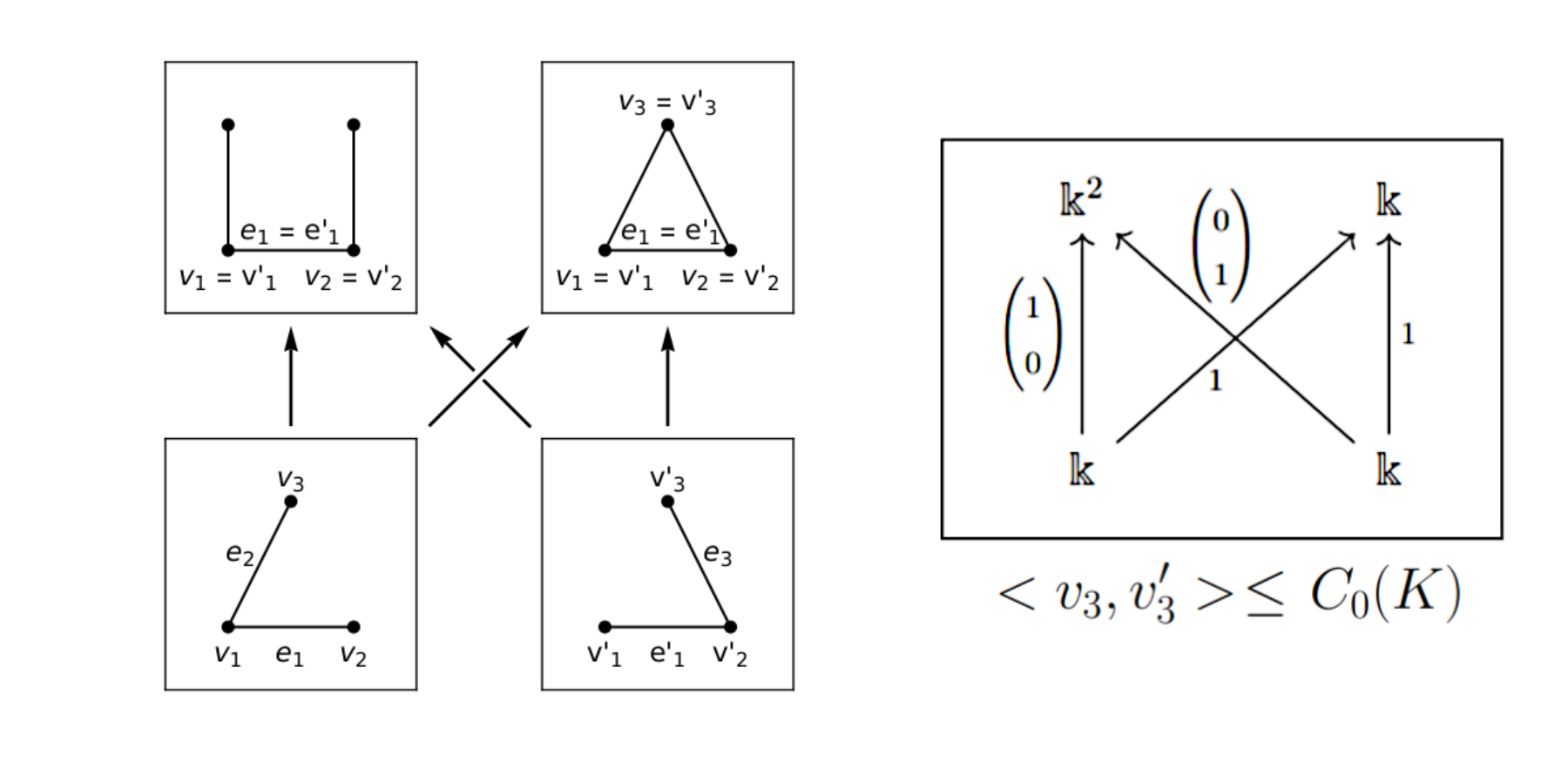}
    \caption{Left: a poset tower over a poset with two maximal elements, where all maps are inclusions. At the left maximal element the two vertices $v_3$ and $v'_3$ are different. In the right maximal element they are identified. Right: The indecomposable summand $<v_3,v'_3>$ of $C_0(K)$, generated by the two vertices $v_3$ and $v'_3$.}
    \label{fig:two_join_poset_indec}
\end{figure}

\textbf{Contributions.} We develop a general-purpose algorithm ({\sc PiRep} in Section~\ref{sec:algorithm}) to compute projective implicit representations of the homology of poset towers $K\colon P\rightarrow \mathbf{SCpx}$, where we allow arbitrary simplicial maps and arbitrary (finite) indexing posets. The algorithm {\sc PiRep} computes degreewise (projective) minimal presentations of the (persistent) simplicial chain complex $C_\bullet(K)$. This means that, for every $\ell$, we compute a minimal presentation $p_\ell^1$ of the chain module $C_\ell(K)$, as depicted in \eqref{eq:intro_presentation}, and a lift $f^0_\ell$ of the boundary map $\partial_\ell$ to the generators. 
\begin{equation} \label{eq:intro_presentation}
\begin{tikzcd}
C_{\ell+1}(K) \arrow[d,swap,"\partial_{\ell+1}"] & G_{\ell+1} \arrow[l] \arrow[d,swap,"f^0_{\ell+1}"] \arrow[ddr,dashed,"\vartheta_{\ell+1}"{xshift=-12pt,yshift=12pt}] & R_{\ell+1} \arrow[l,swap,"p_{\ell+1}^1"] \arrow[d,"f^1_{\ell+1}"] & RR_{\ell+1} \arrow[l,swap,"p_{\ell+1}^2"] \\
C_{\ell}(K) \arrow[d,swap,"\partial_{\ell}"] & G_\ell \arrow[l] \arrow[d,swap,"f^0_\ell"] & R_\ell \arrow[l,swap,"p_\ell^1"{xshift=3pt}] \arrow[d,"f^1_\ell"] & RR_\ell \arrow[l,swap,"p_\ell^2"] \\
C_{\ell\minus 1}(K) & G_{\ell\minus 1}\arrow[l] & R_{\ell\minus 1} \arrow[l,swap,"p_{\ell\minus 1}^1"] & RR_{\ell\minus 1} \arrow[l,swap,"p_{\ell\minus 1}^2"] 
\end{tikzcd}
\end{equation}
Our algorithm takes advantage of the special structure of the chain complex $C_\bullet(K)$. We observe that the presentation matrix $p_\ell^1$ (almost) has the structure of the boundary matrix of a $P$-filtered  graph (Definition~\ref{def:boundary_p}). This allows us to construct a minimal presentation by constructing a graph and avoiding closing (superfluous) cycles. We represent the maps $p_\ell^1$ and $f^0_\ell$ by $P$-graded matrices. Together, they faithfully represent the chain complex $C_\bullet(K)$. Given this representation, one could now compute a projective implicit representation of $H_\ell(K)$ by purely algebraic methods. However, we observe that the structure of $C_\bullet(K)$ can be further exploited to efficiently compute an asymptotically minimal second term $p_\ell^2$ in a projective resolution of $C_\ell(K)$, as depicted in \eqref{eq:intro_presentation}. Here asymptotic minimality means that the number of columns of $p_\ell^2$ is asymptotically worst-case optimal. Individual outputs need not be exactly minimal. We compute $p_\ell^2$, which can be interpreted as capturing relations among relations, by detecting cycles in a graph associated with $p_\ell^1$. We take advantage of the graph structure of $p_\ell^1$ a third time to compute the maps $f_\ell^1$, lifting the boundary maps $\partial_\ell$ to the relations, and $\vartheta_{\ell+1}$, which corrects for the failure of the maps $f^0_\ell$ to form a chain complex. These maps can be determined by solving linear systems, but the graph structure of the coefficient matrix allows us to solve these systems much more efficiently. Finally, we construct a projective implicit representation of $H_\ell(K)$ by assembling the maps in \eqref{eq:intro_presentation} (see Section \ref{sec:pirep_theory}).

At this point, it is worth mentioning that, in \cite{brown2024discretemicrolocalmorsetheory}, the authors provide an algorithm to compute minimal injective resolutions of sheaves on finite posets. A sheaf of vector spaces on a finite poset is equivalent to a $P$-persistence module, and an injective resolution of a $P$-persistence module is equivalent to a projective resolution of a $P^\text{op}$-persistence module obtained by dualization. Their algorithm does not address computing lifts $f_\ell$, as in \eqref{eq:intro_presentation}, but could be applied in our context to compute the degreewise resolutions. For a resolution of $C_\ell(K)$, as depicted in the rows of \eqref{eq:intro_presentation}, the complexity of the algorithm in \cite{brown2024discretemicrolocalmorsetheory} is $O(t_0c^3)$, where $t_0$ denotes the number of points in the poset, and $c$ is the maximal size of the graded matrices in the output resolution. This cubic behavior, with an additional factor from the size of the poset, is of course expected by a general algorithm based on matrix reduction. A matrix representing a minimal $p_\ell^2$ can have $\Theta(t_1n)$ columns if $n$ denotes the number of simplex generators and $t_1$ the number of edges in the Hasse diagram of the poset (see Proposition \ref{prop:asymptotic_optimality}). This would result in a worst-case complexity of $O(t_0t_1^3n^3)$, while our specialized algorithm achieves a complexity of $O\big((t_0+t_1)n^2\big)$. 

It is also worth noting that, for the special case of $K$ being a graph, efficient algorithms for computing minimal presentations are known. The authors of \cite{morozov2024computingbettitablesminimal} give an algorithm to compute a minimal presentation of $H_0(G)$ of a $P$-filtered graph $G$ in $O(|G|^2)$ time for a general poset $P$ and in $O(|G| \log |G|)$ time for $P=\mathbb{Z}^2$.
The algorithm is not applicable to the general case of simplicial filtrations or towers indexed by a (finite) poset.

\textbf{Outline.} In Section \ref{sec:background}, we discuss the basic notions of poset towers, persistence modules, projective modules, projective resolutions, and projective implicit representations. In Section \ref{sec:pirep_theory}, we discuss how we can construct a PiRep from partial projective resolutions of $C_\ell(K)$ and lifts of the boundary maps $\partial_\ell$, as depicted in \eqref{eq:intro_presentation}. In Section \ref{sec:algorithm}, we present an algorithm {\sc PiRep} to compute a projective implicit representation from a poset tower. In Section \ref{sec:correctness}, we prove the correctness of the algorithm. We conclude in Section \ref{sec:conclusion}. We provide proofs omitted from the main text in Appendix \ref{app:proofs_sec_background} and \ref{app:proofs}. In Appendix \ref{app:tree_solve}, we discuss an efficient algorithm for solving systems of linear equations whose coefficient matrix has the structure of the boundary matrix of a tree, which is used in the algorithm {\sc PiRep}. For completeness, we discuss in Appendix \ref{app:homology_presentation} how to compute a minimal presentation of $H_\ell(K)$ from a PiRep using an algorithm discussed in \cite{brown2024discretemicrolocalmorsetheory}. 

\section{Poset towers, PiReps, and algebraic background} \label{sec:background}

In this section, we introduce the basic notions of poset towers, persistence modules, projective modules, projective resolutions, and projective implicit representations.

\textbf{Notation.} In the following, we denote by $P$ a finite partially ordered set (poset) with its points referred to as \emph{grades}. For $x,y\in P$, we use the notation $x\prec y$ if $x<y$ and there is no grade $z\in P$ such that $x<z<y$. We call such an $x$ a predecessor of $y$. We assume that $P$ is represented by a directed graph (called the \emph{Hasse diagram}), whose nodes are the grades and whose edges are relations $x\prec y$ in $P$. We denote by $\overline{P}\coloneqq P\cup\{-\infty\}$ the poset obtained from $P$ by adding a grade $-\infty$ such that $-\infty<x$ for all $x\in P$.

We also view $P$ as a category with objects the elements of $P$ and a unique morphism $x\rightarrow y$ if $x\leq y$. 
For a functor $F\colon P\rightarrow\mathbf{C}$, we denote by $F(x)$ the image of $x\in P$ under $F$ and by $F(x\leq y)$ the image of $x\rightarrow y$ under $F$ for $x\leq y\in P$. Let the category of functors $P\rightarrow\mathbf{C}$ and their natural transformations be denoted by $\mathbf{C}^P$. We use the notation $\mathbf{Vec}$ for the category of finite-dimensional vector spaces over $\mathbb{F}_2$, $\mathbf{SCpx}$ for the category of finite simplicial complexes and simplicial maps, and $\Delta\mathbf{Cpx}_{\leq 1}$ for the category of finite one-dimensional $\Delta$-complexes (see Chapter 2.1 in \cite{hatcher2002algebraic}). For simplices $\sigma$ and $\tau$ in a simplicial or $\Delta$-complex, we write $\sigma\leq\tau$ if $\sigma$ is a face of $\tau$, and $\sigma<\tau$ if $\sigma$ is a proper face of $\tau$.

\textbf{Coefficients.} Throughout this work, we use coefficients in $\mathbb{F}_2$, as is common in topological data analysis. This choice simplifies both the exposition and the algorithmic treatment. Since we allow general simplicial maps over general posets, working over an arbitrary field would require additional orientation data and sign bookkeeping at the chain level. In particular, one would need to keep track of the signs induced by simplicial maps on oriented simplices, which would substantially complicate both the input representation and the resulting algorithms. 
We do not see an immediate conceptual obstruction to extending our methods to general field coefficients, but the present algorithm and its proof of correctness are developed specifically for $\mathbb{F}_2$.
Finally, we note that although we formulate the algebraic results in Sections~\ref{sec:background} and~\ref{sec:pirep_theory} over $\mathbb{F}_2$, they remain true over arbitrary fields after the appropriate signs are taken into account.

\subsection{Poset towers and persistence modules}

\begin{figure}
\centerline{\includegraphics[width=0.8\textwidth]{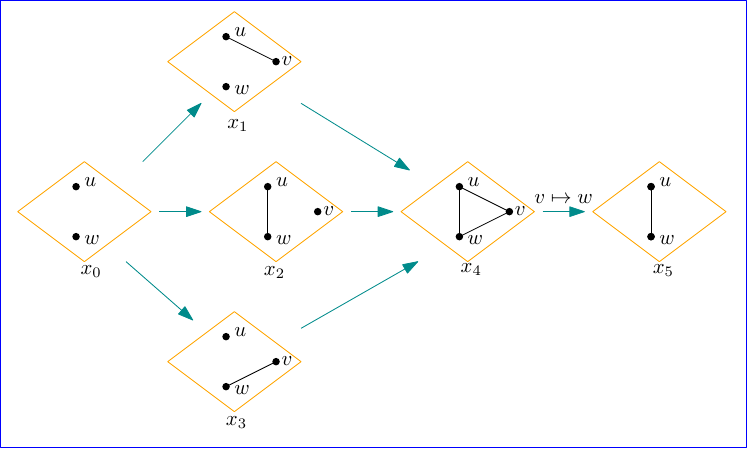}}
\caption{A poset tower on a poset $P=\{x_0,\ldots,x_5\}$ with simplicial complexes comprising
vertices and edges at the nodes ($\diamond$) and inclusions/collapses on poset edges ($\rightarrow$); e.g. vertex $v$ is generated at $x_1,x_2,x_3$, i.e., it has generators $g_v^{x_1},g_v^{x_2},g_v^{x_3}$, which generate two relations $r^{x_4}_1\mapsto g_v^{x_1}+g_v^{x_2}$ and $r^{x_4}_2\mapsto g_v^{x_2}+g_v^{x_3}$ at $x_4$;
there is a vertex collapse $v\mapsto w$ on the poset edge $x_4x_5$;  
this collapse generates the relation $r^{x_5}_1\mapsto g_{vw}^{x_3}$,
it also
generates two copies of $uw$ at $x_5$ with generators $g_{uv}^{x_1}$
and $g_{uw}^{x_2}$ which get identified with a relation
$r^{x_5}_2\mapsto g_{uw}^{x_2}+g_{uv}^{x_1}$ at $x_5$.}
\label{fig:posettower}
\end{figure}

Generalizing the notion of simplicial tower~\cite{dey2014computing,kerber2019barcodes}, we introduce poset towers: 

\begin{definition}[Poset tower] \label{def:poset_tower}
A poset tower is a functor $K\colon P\rightarrow \mathbf{SCpx}$; see, e.g., Figure~\ref{fig:posettower}.
\end{definition}

We denote by $K_\ell(x)$ the set of all $\ell$-simplices in $K(x)$. In particular, $K_0(x)$ is the set of vertices of $K(x)$. We define $\operatorname{Vert}(K)\coloneqq \bigcup_{x\in P}K_0(x)$ as the set of all vertices in $K$ where the union is the ordinary set-theoretic union and vertex sets at different grades may overlap. For each $x\in P$, every simplex of $K(x)$ is uniquely determined by its set of vertices in $\operatorname{Vert}(K)$. The simplicial maps need not preserve these labels. 

In most applications, it is not economical to store a poset tower as a simplicial complex for every grade of the poset $P$. For example, in the case of a multi-filtered simplicial complex, which, in our setting, is a poset tower $K\colon P\rightarrow \mathbf{SCpx}$ over a finite grid $P=\{1,\ldots,m\}^d$ where each map $K(x\leq y)$ is an inclusion, we only need a list of birth grades for every simplex in the complex $K(\text{max})$ at the maximum grade $\text{max}=(m,\ldots,m)\in P$. We generalize this compact representation to general poset towers by introducing simplex generators and edge events.

\vspace{20pt}

\begin{definition}[Simplex generators] \label{def:simplex_generators}
Let $K\colon P\rightarrow \mathbf{SCpx}$ be a poset tower. We call the tuple $g_\sigma^x\coloneqq (x,\sigma)$ a generator of the simplex $\sigma$ at $x\in P$ if $\sigma \in K(x)$ and $\sigma\notin \im \!\big(K(y\prec x)\big)$ for all $y\prec x\in P$. We denote by $\mathcal{S}$ the set of all simplex generators of $K$. Formally:
\begin{equation*}
\mathcal{S}\coloneqq\Big\{g_\sigma^x\,\vert\, x\in P\text{,}\sigma\in K(x)\setminus\bigcup_{y\prec x}\im \!\big(K(y\prec x)\big)\Big\}.
\end{equation*}
We denote by ${\mathcal S}_\ell$ the set of generators of $\ell$-simplices. 
\end{definition}

Note that, for the special case of a multi-filtered simplicial complex $K$, the simplex generators $g^x_{\sigma}$ simply recover the birth grades of a simplex $\sigma\in K(\text{max})$ in the complex at the maximum. For general poset towers, the definition of simplex generators provides the flexibility needed to handle situations in which simplices may be collapsed or no maximum grade exists; see Figures~\ref{fig:collapse_grid_indec} and~\ref{fig:two_join_poset_indec}.

Next, we introduce edge events as a compact way to represent the simplicial maps in a poset tower. An edge event over an edge $x\prec y$ is a vertex in $K(x)$ that is mapped to a vertex different from itself along the edge. These edge events account for possible collapses and permutations that can happen with arbitrary simplicial maps, which are simulated by vertex identifications. The collapse $v\mapsto w$ on the poset edge $x_4x_5$ in Figure~\ref{fig:posettower} is an edge event.

\begin{definition}[Edge events] \label{def:edge_events}
Let $K\colon P\rightarrow \mathbf{SCpx}$ be a poset tower. We call the tuple $c_{v\mapsto w}^{x\prec y}\coloneqq(x,y,v,w)$ an edge event if $x\prec y\in P$, $v\in K_0(x)$, $w=K(x\prec y)(v)\in K_0(y)$, and $w\neq v$ as elements in $\operatorname{Vert}(K)$. We denote by $\mathcal{C}$ the set of all edge events of $K$. Formally:
\begin{equation*}
\mathcal{C}\coloneqq \Big\{c_{v\mapsto w}^{x\prec y}\,\vert\, x\prec y\in P, v\in K_0(x), w=K(x\prec y)(v)\in K_0(y), v\neq w\Big\} .
\end{equation*}
\end{definition}

Instead of recording all simplicial maps, we just record when they change a simplex. A multi-filtered simplicial complex, for example, has no edge events. The simplex generators and edge events completely determine the poset tower.

\begin{proposition} \label{prop:tower_from_generators_events}
A poset tower $K\colon P\rightarrow \mathbf{SCpx}$ is completely described by the list of simplex generators $\mathcal{S}$ and edge events $\mathcal{C}$ induced by $K$.   
\end{proposition}

\begin{proof}
Let $\mathcal{S}$ and $\mathcal{C}$ be as in Definitions \ref{def:simplex_generators} and \ref{def:edge_events}. First, we show that 
\begin{equation} \label{eq:vert_claim}
\operatorname{Vert}(K)=\{v\vert g_v^x\in \mathcal{S}_0\}\cup \{w\vert c_{v\mapsto w}^{x\prec y}\in\mathcal{C} \}.
\end{equation}
Let $u\in \operatorname{Vert}(K)$. Then there exists $x\in P$ such that $u\in K_0(x)$. If $u\notin \im \!\big(K(y\prec x)\big)$ for all $y\prec x\in P$, then $g_u^x\in \mathcal{S}_0$. Otherwise, there exist $y\prec x$ and $v\in K_0(y)$ such that $K(y\prec x)(v)=u$. If $v\neq u$ in $\operatorname{Vert}(K)$, then $c_{v\mapsto u}^{y\prec x}\in\mathcal{C}$. Else $v=u\in K_0(y)$ and we can repeat the same argument for $y<x$. Since $P$ is finite and every vertex at a minimal grade has a generator in $\mathcal{S}_0$, this process terminates, yielding either a generator $g_u^z\in\mathcal S_0$ or an edge event whose target is $u$. Conversely, by definition, $g_v^x\in \mathcal{S}_0$ implies $v\in K_0(x)\subseteq \operatorname{Vert}(K)$ and $c_{v\mapsto w}^{x\prec y}\in\mathcal{C}$ implies $w\in K_0(y)\subseteq \operatorname{Vert}(K)$. This finishes the proof of \eqref{eq:vert_claim}.

We now extend $P$ to a linear order $P_{lin}=\{x_1<_{lin} \cdots<_{lin} x_n\}$ and construct $K'$ recursively. For the construction we use $\operatorname{Vert}(K)$ which can be obtained directly from $\mathcal{S}$ and $\mathcal{C}$ by \eqref{eq:vert_claim} and denote by $\Delta_{\operatorname{Vert}(K)}$ the complete simplicial complex on the vertex set $\operatorname{Vert}(K)$. For $i=1,\ldots, n$, define 
\begin{equation*}
K'(x_i)\coloneqq \{\sigma\vert \exists g_\sigma^{x_i}\in\mathcal{S}\}\cup\bigcup_{y\prec x_i}\im\!\big( \overline{K}'(y\prec x_i)\big).
\end{equation*}
For all $x_i\prec y$, define $\overline{K}'(x_i\prec y)\colon K'(x_i)\rightarrow \Delta_{\operatorname{Vert}(K)}$ as the simplicial map defined on vertices $v\in K'_0(x_i)$ by $\overline{K}'(x_i\prec y)(v)\coloneqq w$ if there exists a $c_{v\mapsto w}^{x_i\prec y}\in \mathcal{C}$ and $\overline{K}'(x_i\prec y)(v)\coloneqq v$ otherwise. After processing all $x_i$ in $P_{lin}$, we define $K'(x\prec y)$ as the restriction of $\overline{K}'(x\prec y)$ to a map $K'(x\prec y)\colon K'(x)\rightarrow K'(y)$, for all $x\prec y\in P$, where we use that $\im \big( \overline{K}'(x\prec y)\big)\subseteq K'(y)$.

We now show $K=K'$: Suppose $K(x_i)=K'(x_i)$, $v\in K'_0(x_i)=K_0(x_i)$ and $x_i\prec y$. If $K(x_i\prec y)(v)=v$, there is no edge event $c_{v\mapsto w}^{x_i\prec y}$  in $\mathcal{C}$ and, thus, $K'(x_i\prec y)(v)=v$. If $K(x_i\prec y)(v)=w\neq v$, there is an edge event $c_{v\mapsto w}^{x_i\prec y}\in \mathcal{C}$ and, thus, $K'(x_i\prec y)(v)=w$. Since the images of the vertices completely determine a simplicial map, we get $K(x_i\prec y)=K'(x_i\prec y)$.

We show that $K(x_i)=K'(x_i)$ by induction on the grades in the linear order. The first element $x_1$ must be minimal in $P$. Hence, by definition, for all $\sigma\in K(x_1)$ we have $g^{x_1}_\sigma\in\mathcal{S}$. Therefore, $K(x_1)=K'(x_1)$. Suppose $K(x_j)=K'(x_j)$ for all $j<i$. Every predecessor $y\prec x_i$ satisfies $y=x_j$ for some $j<i$ in the linear extension. Hence, by the argument in the previous paragraph, we get $K(x_i)=\{\sigma\vert \exists g_\sigma^{x_i}\in\mathcal{S}\}\cup\bigcup_{y\prec x_i}\im\!\big( K(y\prec x_i)\big)=\{\sigma\vert \exists g_\sigma^{x_i}\in\mathcal{S}\}\cup\bigcup_{y\prec x_i}\im \!\big(K'(y\prec x_i)\big)=K'(x_i)$. 

Therefore, $K(x)=K'(x)$ and $K(x\prec y)=K'(x\prec y)$ for all $x\prec y\in P$, which implies that $K=K'$.
\end{proof}

We now define the main algebraic objects, the persistence modules, and discuss how we can work with them in a computational framework.

\begin{definition}[$P$-persistence module] \label{def:persistence_module}
 A $P$-persistence module is a functor $M\colon P\rightarrow \mathbf{Vec}$.
\end{definition}

The $P$-persistence modules are precisely representations of finite posets. Calling these objects modules is justified by the fact that the category of representations of finite posets is equivalent to the category of modules over the so-called path or incidence algebra of the poset \cite{assem2006elements,schiffler2014quiver}.
Persistence modules in our context naturally arise as the homology of poset towers. 
The $\ell$-th homology, with coefficients in $\mathbb{F}_2$, can be viewed 
as a functor $H_\ell\colon\mathbf{SCpx}\rightarrow \mathbf{Vec}$ for every $\ell\geq 0$. If we take the composition of the $\ell$-th homology functor with a poset tower $K\colon P\rightarrow\mathbf{SCpx}$, we obtain the persistence module $H_\ell(K)\coloneqq H_\ell\circ K\colon P\rightarrow\mathbf{Vec}$. 

The homology of a simplicial complex is computed from the simplicial chain complex, which can be viewed as an algebraic representation of the complex. For $x\leq y\in P$, we denote by $C_\bullet\big(K(x)\big)$ the simplicial chain complex and by $C_\bullet\big(K(x\leq y)\big)$ the morphism induced by the simplicial map $K(x\leq y)$; see \eqref{eq:simplicial_chain complex}. The vector space $C_\ell\big(K(x)\big)$ has a basis consisting of the $\ell$-simplices in $K(x)$ and $C_\ell\big(K(x\leq y)\big)$ sends $\sigma\in K_\ell(x)$ to $K(x\leq y)(\sigma)$ if $K(x\leq y)(\sigma)\in K_\ell(y)$ and to zero otherwise.
\begin{equation} \label{eq:simplicial_chain complex}
\begin{tikzcd}[row sep=large, column sep=large]
\cdots & C_{\ell\minus 1}\big(K(x)\big) \arrow[l] \arrow[d,"C_{\ell\minus 1}\big(K(x\leq y)\big)"] & C_\ell\big(K(x)\big) \arrow[l,swap,"\partial_{\ell}"] \arrow[d,"C_{\ell}\big(K(x\leq y)\big)"] & C_{\ell+1}\big(K(x)\big) \arrow[l,swap,"\partial_{\ell+1}"] \arrow[d,"C_{\ell+1}\big(K(x\leq y)\big)"] & \cdots \arrow[l] \\
\cdots & C_{\ell\minus 1}\big(K(y)\big) \arrow[l] & C_\ell\big(K(y)\big) \arrow[l,swap,"\partial_{\ell}"] & C_{\ell+1}\big(K(y)\big) \arrow[l,swap,"\partial_{\ell+1}"] & \cdots \arrow[l] 
\end{tikzcd}
\end{equation}

Let $\mathbf{Ch}\big(\mathbf{Vec}\big)$ and $\mathbf{Ch}\big(\mathbf{Vec}^P\big)$ denote the categories of chain complexes over vector spaces and over $P$-persistence modules, respectively. We can view $C_\bullet(K)\colon P\rightarrow \mathbf{Ch}(\mathbf{Vec})$ as a persistent chain complex or, equivalently, as a chain complex of persistence modules in $\mathbf{Ch}\big(\mathbf{Vec}^P\big)$:
\begin{equation*}
\begin{tikzcd}
\cdots & C_{\ell\minus 1}(K) \arrow[l] & C_\ell(K) \arrow[l,swap,"\partial_{\ell}"] & C_{\ell+1}(K) \arrow[l,swap,"\partial_{\ell+1}"] & \cdots \arrow[l]
\end{tikzcd}
\end{equation*}
where each $C_\ell(K)\colon P\rightarrow\mathbf{Vec}$ is a persistence module. We will adopt this global perspective in the following.

\subsection{Projective modules} \label{subsec:projective_modules}

To carry out computations with persistence modules, we need to represent them by matrices. As with poset towers, it is not economical to represent a persistence module by a collection of matrices $M(x\leq y)$ for all $x\leq y\in P$. Instead, we represent persistence modules by elementary building blocks called projective modules. Since we work with a specific class of modules, we will directly define the projective modules as in Definition \ref{def:elementary_module}. See \cite[Theorem III.1.6, Lemma III.2.4]{assem2006elements} for a reference showing that these are exactly the projective modules according to the general definition \cite[Definition I.5.2]{assem2006elements}. Note that in persistence theory, one often works with graded free modules. For $\mathbb{Z}^d$-graded modules over graded polynomial rings, these coincide with graded projective modules. In the setting of modules over finite posets, however, projective modules are generally different from free modules; see, for example, \cite[Chapters I.5, III.2]{assem2006elements} or \cite[Definition 2.6]{schiffler2014quiver}.

\begin{definition}[Projective modules] \label{def:elementary_module}
We denote by $\proj[x]\colon P\rightarrow\mathbf{Vec}$ the elementary projective module at $x\in P$ defined for all $y\leq z\in P$ by
\begin{equation*}
\proj[x](y)=\begin{cases} \mathbb{F}_2 &\text{if}~ x\leq y \\ 0 & \text{otherwise} \end{cases} \quad , \qquad \proj[x](y\leq z)=\begin{cases} \text{id} &\text{if}~x\leq y\leq z \\ 0 & \text{otherwise} \end{cases} .
\end{equation*}   
A direct sum $\bigoplus_{i=1}^n\proj[x_i]$ is called a projective module; see Figure \ref{fig:projective} for an example.
\end{definition}

\begin{figure}[htbp]
\centerline{\includegraphics[width=0.5\textwidth]{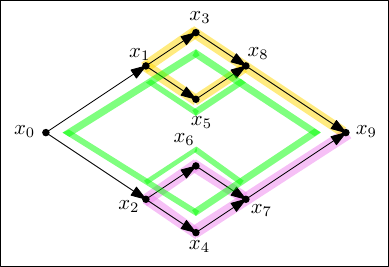}}
\caption{An illustration of a projective module with three elementary projective summands; $\proj[x_0]$ shaded green, $\proj[x_1]$ shaded yellow, and  $\proj[x_2]$ shaded violet.}
\label{fig:projective}
\end{figure}

Projective modules are well-suited for computations because morphisms between them are simple and can be described by $P$-graded matrices.

\vspace{40pt}

\begin{proposition}[Properties of morphisms between projective modules]\label{prop:elementary_projective}
$\phantom{a}$
\begin{enumerate}
    \item For $x,y\in P$:
    \begin{equation*}
    \Hom\big(\proj[x],\proj[y]\big)\cong\begin{cases}\mathbb{F}_2 & \text{ if } y\leq x \\ 0 & \text{ otherwise} \end{cases} 
    \end{equation*}
    as $\mathbb{F}_2$-vector spaces.
    \item For $Q=\bigoplus_{i=1}^m\proj[x_i]$ and $W=\bigoplus_{j=1}^n\proj[y_j]$ projective modules:
    \begin{equation*}
    \begin{aligned}
    \Hom\big(Q,W\big)&\cong \Hom\Big(\bigoplus_{i=1}^m\proj[x_i],\bigoplus_{j=1}^n\proj[y_j]\Big)\\&\cong \bigoplus_{j=1}^n \bigoplus_{i=1}^m \Hom\big(\proj[x_i],\proj[y_j]\big)\\&\cong \big\{(\lambda_{ji})_{j,i}\in\mathbb{F}_2^{n\times m}\vert \lambda_{ji}=0\text{ if } y_j\nleq x_i\big\} .
    \end{aligned}
    \end{equation*}
\end{enumerate}
\end{proposition}

\begin{proof}
See Appendix \ref{app:proofs_sec_background}.
\end{proof}

\begin{definition}[$P$-graded matrix] \label{def:P-graded_matrix}
Let $f\colon Q\rightarrow W$ be a morphism of projective modules as in Proposition \ref{prop:elementary_projective}(2). Then $f$ is completely determined by a matrix $(f_{ji})_{i=1,j=1}^{m,n}\in\mathbb{F}_2^{n\times m}$ whose columns and rows are identified with the elementary projective modules in $Q$ and $W$, respectively. We call a matrix representation $(f_{ji})_{i=1,j=1}^{m,n}$ whose columns are labeled with the grades $x_i$ of the summands $\proj[x_i]$ of $Q$ and whose rows are labeled with the grades $y_j$ of the summands $\proj[y_j]$ of $W$ a $P$-graded matrix representation of $f$.
\end{definition}

\textbf{Notation for projective generators.} To simplify the exposition, we introduce the following notation: We view an elementary projective module $\proj[x]\colon P\rightarrow\mathbf{Vec}$ as a formal span of a projective generator $\langle b^{x}\rangle$ and a general projective module $\bigoplus_{i=1}^n\proj[x_i]$ as $\langle b^{x_1},\ldots, b^{x_n}\rangle$. If we pick an element $x\in P$, then the vector space $\big(\bigoplus_{i=1}^n\proj[x_i]\big)(x)$ has a basis $\{b^{x_j}\vert x_j\leq x\}$ consisting of all formal projective generators born below $x$ in $P$. Moreover, for $x\leq y\in P$, the linear map $\big(\bigoplus_{i=1}^n\proj[x_i]\big)(x\leq y)$ is induced by the inclusion of basis elements  $\{b^{x_j}\vert x_j\leq x\}\xhookrightarrow{} \{b^{x_k}\vert x_k\leq y\}$. This allows us to simplify the notation for maps between projective modules $\proj[y]\xrightarrow{1} \proj[x]$ by $b^{y}\xmapsto{} b^{x}$ for $x\leq y$ and, more generally, 
\begin{equation*}
\proj[y]\xrightarrow{\begin{pmatrix} 1 \\ \vdots \\ 1\end{pmatrix}} \bigoplus_{i=1}^n\proj[x_i]
\end{equation*}
by $b^y\mapsto b^{x_1}+\cdots + b^{x_n}$. For fixed $x_i\leq y\leq z\in P$ this can again be interpreted as the linear map sending the basis element $b^y$ of $\big(\proj[y]\big)(z)$ to the sum of basis elements $b^{x_1}+\cdots + b^{x_n}$ of $\big(\bigoplus_{i=1}^n\proj[x_i]\big)(z)$. Given this notion of projective generators, we can also define $P$-graded matrices to have columns and rows labeled by the projective generators of the corresponding summands. \\

\textbf{Equations of morphisms of projective modules.} In the following, we will encounter the problem of solving equations $A\circ X=B$ of morphisms of projective modules, as depicted in \eqref{eq:projective_object_7}. The special properties of projective modules allow us to solve such equations, which we will call \emph{$P$-graded linear systems}, via standard linear algebra.

\begin{definition}[$P$-graded linear system] \label{def:P_graded_system}
For a poset $P$, projective modules $Q=\bigoplus_{i=1}^n \proj[x_i]$, $U=\bigoplus_{i=1}^m \proj[y_i]$, $R=\bigoplus_{i=1}^s \proj[z_i]$, and morphisms $A$ and $B$ as in \eqref{eq:projective_object_7}, 
\begin{equation} \label{eq:projective_object_7}
\begin{tikzcd}
U \arrow[d,swap,"B"] \arrow[dr,dashed,"X"] \\
Q  & R \arrow[l,"A"]
\end{tikzcd}
\end{equation}
we call the equation $A\circ X=B$ a $P$-graded linear system for the unknown morphism $X$. 

By Proposition \ref{prop:elementary_projective}, the vector space $\Hom(U,R)$ can be described as $\{(X_{ij})_{i,j}\in \mathbb{F}_2^{s\times m}\vert X_{ij}=0\text{ if }z_i\nleq y_j\}$. Using analogous matrix representations $(A_{ij})_{i,j}$ and $(B_{ij})_{i,j}$ of $A$ and $B$, respectively, we obtain the matrix representation of the composition by matrix multiplication: $(A\circ X)_{ij}=\sum_{k=1}^s A_{ik}X_{kj}$. The $P$-graded linear system $A\circ X=B$ can then equivalently be defined as the following constraint linear system with coefficient matrix $A$:
\begin{equation} \label{eq:P_linear_system}
\begin{aligned}
& (A\circ X)_{ij}=\sum_{k=1}^s A_{ik}X_{kj}=B_{ij} \hspace{3pt},\hspace{5pt} \forall\hspace{2pt} 1\leq i \leq n \hspace{1pt}, \hspace{2pt}1\leq j\leq m \\
& X_{kj}=0 \hspace{3pt},\hspace{5pt}\forall\hspace{2pt} 1\leq k \leq s \hspace{1pt}, \hspace{2pt}1\leq j\leq m \text{ such that }  z_k\nleq y_j .
\end{aligned}
\end{equation}
\end{definition}

The following proposition shows that a $P$-graded linear system, as in Definition \ref{def:P_graded_system}, can be reduced to $m$ unconstrained linear systems. 

\begin{proposition} \label{prop:eliminate_constraints}
Let $AX=(AX_1,\ldots,AX_m)=(B_1,\ldots,B_m)=B$ and $X_{kj}=0$ for all $z_k\nleq y_j$ be a $P$-graded linear system as in Definition \ref{def:P_graded_system}. Let $A^{\leq y_j}$, $A^{\nleq y_j}$, $X_j^{\leq y_j}$, and $X_j^{\nleq y_j}$ be the submatrices of columns and rows indexed by $z_k$ such that $z_k\leq y_j$ and $z_k\nleq y_j$, respectively. Then $AX=B$ and $X_{kj}=0$ for all $z_k\nleq y_j$ if and only if $A^{\leq y_j}X^{\leq y_j}_j=B_j$ and $X_j^{\nleq y_j}=0$ for all $1\leq j\leq m$. 
\end{proposition}

\begin{proof}
See Appendix \ref{app:proofs_sec_background}.
\end{proof}

\subsection{Projective resolutions} \label{subsec:projective_resolutions}

Every persistence module, and also every morphism between persistence modules can be expressed by projective modules and morphisms between them. We will use this fact to translate complex algebraic objects into $P$-graded matrices.  

\begin{definition}[Projective resolution and presentation] \label{def:projective_resolution}
$\phantom{a}$
\begin{enumerate}
    \item A projective resolution of a $P$-persistence module $M$ is a chain complex
    \begin{equation}\label{eq:projres_def}
    \begin{tikzcd}
    0 & M \arrow[l,swap,"q_{\minus 1}"] & Q_0 \arrow[l,swap,"q_0"] & Q_1 \arrow[l,swap,"q_1"] & Q_2 \arrow[l,swap,"q_2"] & \cdots \arrow[l,swap,"q_3"]
    \end{tikzcd}
    \end{equation}
    where $Q_i$ is a projective module and $\ker(q_{i\minus 1})=\im(q_i)$ for all $i\geq 0$.

    \item Let $\phi\colon M\rightarrow N$ be a morphism of $P$-persistence modules and $M\leftarrow Q_\bullet$ and $N\leftarrow W_\bullet$ be projective resolutions of $M$ and $N$, respectively. A lift of $\phi$ to these projective resolutions is a family of morphisms $f_\bullet\colon Q_\bullet\rightarrow W_\bullet$:
    \begin{equation}\label{eq:projres_map_def}
    \begin{tikzcd}
    0 & M \arrow[l,swap,"q_{\minus 1}"] \arrow[d,swap,"\phi"] & Q_0 \arrow[l,swap,"q_0"] \arrow[d,"f_0"] & Q_1 \arrow[l,swap,"q_1"] \arrow[d,"f_1"] & Q_2 \arrow[l,swap,"q_2"] \arrow[d,"f_2"] & \cdots \arrow[l,swap,"q_3"] \\
    0 & N \arrow[l,swap,"w_{\minus 1}"] & W_0 \arrow[l,swap,"w_0"] & W_1 \arrow[l,swap,"w_1"] & W_2 \arrow[l,swap,"w_2"] & \cdots \arrow[l,swap,"w_3"]
    \end{tikzcd}
    \end{equation}
    such that all squares in \eqref{eq:projres_map_def} commute.

    \item A (projective) presentation of a $P$-persistence module $M$ is an exact sequence of the form 
    \begin{equation} \label{eq:def_proj_presentation}
        0\leftarrow M\xleftarrow{q_0} Q_0\xleftarrow{q_1} Q_1
    \end{equation}
    where $Q_0$ and $Q_1$ are projective modules, $q_0$ is an epimorphism, and $\ker(q_{0})=\im(q_1)$.
\end{enumerate}
\end{definition}

Note that while $M\xleftarrow{q_0} Q_0$ in Definition \ref{def:projective_resolution} is important for theoretical arguments and proofs, in practice $M$ is fully described by $q_1$ since $M\cong \coker(q_1)$. 

\begin{proposition}[Existence of projective resolutions] \label{prop:existence_resolutions}
$\phantom{a}$
\begin{enumerate}
    \item Let $M$ be a $P$-persistence module. Then there exists a projective resolution, as in \eqref{eq:projres_def}. The resolution is unique up to homotopy equivalence of chain complexes. 
    \item Let $\phi\colon M\rightarrow N$ be a morphism of $P$-persistence modules and $M\leftarrow Q_\bullet$ and $N\leftarrow W_\bullet$ projective resolutions. Then there exists a lift $f_\bullet\colon Q_\bullet\rightarrow W_\bullet$ of $\phi$, as in \eqref{eq:projres_map_def}. The lift is unique up to chain homotopy. This means that if $g_\bullet\colon Q_\bullet\rightarrow W_\bullet$ is another lift of $\phi$:
    \begin{equation*}
    \begin{tikzcd}
    0 & M \arrow[l,swap,"q_{\minus 1}"] \arrow[d,swap,"\phi"] & Q_0 \arrow[l,swap,"q_0"] \arrow[d,xshift=2pt,"f_0"] \arrow[d,swap,xshift=-2pt,"g_0"] \arrow[dr,dashed,"h_0"] & Q_1 \arrow[l,swap,"q_1"] \arrow[d,xshift=2pt,"f_1"] \arrow[d,swap,xshift=-2pt,"g_1"] \arrow[dr,dashed,"h_1"] & Q_2 \arrow[l,swap,"q_2"] \arrow[d,xshift=2pt,"f_2"] \arrow[d,swap,xshift=-2pt,"g_2"] \arrow[dr,dashed,"h_2"] & \cdots \arrow[l,swap,"q_3"] \\
    0 & N \arrow[l,swap,"w_{\minus 1}"] & W_0 \arrow[l,swap,"w_0"] & W_1 \arrow[l,swap,"w_1"] & W_2 \arrow[l,swap,"w_2"] & \cdots \arrow[l,swap,"w_3"]
    \end{tikzcd}
    \end{equation*}
    then there exists a chain homotopy $h_\bullet\colon Q_\bullet\rightarrow W_{\bullet+1}$ such that $f_i+g_i=h_{i\minus 1}\circ q_i+w_{i+1}\circ h_i$ for all $i\geq 1$ and $f_0+g_0=w_1\circ h_0$.
    \item Resolutions and lifts can be constructed iteratively. Any partial resolution or lift satisfying the conditions in Definition \ref{def:projective_resolution}(1) or (2) up to index $i\geq 0$ can be extended to a full resolution or lift, respectively.   
\end{enumerate} 
\end{proposition}

\begin{proof}
See Lemma 2.2.5 and Theorem 2.2.6 in \cite{Weibel_1994}.
\end{proof}

\textbf{Minimal projective resolutions.} Since presentations and resolutions are not unique up to isomorphism, they can contain superfluous information. For computational efficiency, we want resolutions that are as small as possible.

\begin{definition}[Minimal projective resolution] \label{def:minimal_resolution}
A projective resolution $M\leftarrow Q_\bullet$, as in Definition  \ref{def:projective_resolution}, is called minimal if $Q_i$ minimizes the number of elementary projective summands for all $i\geq 0$. A presentation is called minimal if $Q_0$ and $Q_1$ minimize the number of elementary projective summands.
\end{definition}

We will use an equivalent characterization of minimal resolutions based on the notion of the radical and projective cover. This notion of minimality will be crucial for the algorithm {\sc PiRep} and the proof of its correctness. 

\begin{definition}[Radical] \label{def:radical}
Let $M$ be a $P$-persistence module. Then $\mathrm{Rad}(M)\colon P\rightarrow \mathbf{Vec}$ defined for all $x\leq y\in P$ by 
\begin{equation*}
\begin{aligned}
\mathrm{Rad}(M)(x)&\coloneqq \sum_{u<x}\im\!\big( M(u<x)\big) \\
\mathrm{Rad}(M)(x\leq y)&\coloneqq \left[ \sum_{u<x}\im \!\big(M(u<x)\big)\xrightarrow{M(x\leq y)} \sum_{v<y}\im \!\big(M(v<y)\big) \right]
\end{aligned}
\end{equation*}
is called the radical of $M$.
\end{definition}

\begin{definition}[Projective cover] \label{def:projective_cover}
Let $M$ be a $P$-persistence module. Then a morphism $g\colon Q\rightarrow M$ is called a projective cover of $M$ if $Q$ is projective, $g$ is an epimorphism, and $\ker(g)\subseteq \mathrm{Rad}(Q)$.
\end{definition}

\begin{proposition}\label{prop:minimal_eqq_projective_cover}
A projective resolution, as in Definition \ref{def:projective_resolution}, is minimal, according to Definition \ref{def:minimal_resolution}, if and only if $q_i\colon Q_i\rightarrow\ker(q_{i\minus 1})$ is a projective cover for all $i\geq 0$. Similarly, a projective presentation, as in Definition \ref{def:projective_resolution}, is minimal, according to Definition \ref{def:minimal_resolution}, if and only if $q_0:Q_0\rightarrow M$ and $q_1\colon Q_1\rightarrow \ker(q_0)$ are projective covers. 
\end{proposition}

\begin{proof}
See Appendix \ref{app:proofs_sec_background}.
\end{proof}

Finally, in our setting, minimal resolutions always exist.

\begin{proposition}[Existence of minimal projective resolutions] \label{prop:existence_minimal_resolutions}
Let $M$ be a $P$-persistence module over a finite poset valued in finite-dimensional vector spaces. Then $M$ has a minimal projective resolution. The minimal resolution is unique up to isomorphism, consists of finite direct sums of elementary projective modules, and its length is bounded by the longest chain in the poset $P$. Moreover, every minimal presentation is the restriction of a minimal projective resolution. 
\end{proposition}

\begin{proof}
See Theorem 5.8 and Corollary 5.10 in \cite{assem2006elements} as well as Proposition 2.2.6 in \cite{ROGNERUD2021107885}.
\end{proof}

\subsection{Projective implicit representations} \label{subsec:pireps}

To compute the persistent homology of a poset tower $K\colon P\rightarrow \mathbf{SCpx}$, we need to compute the homology of the chain complex of $P$-persistence modules $C_\bullet(K)$. For this task, it would be ideal if we had a projective resolution of the whole chain complex $C_\bullet(K)$, which can be defined as a chain complex of projective modules $Q_\bullet$ together with a morphism of chain complexes $f_\bullet\colon Q_\bullet\rightarrow C_\bullet(K)$:
\begin{equation*}
\begin{tikzcd}
0 & Q_0 \arrow[l] \arrow[d,"f_0"] & Q_1 \arrow[l,swap,"q_1"] \arrow[d,"f_1"] & Q_2 \arrow[l,swap,"q_2"] \arrow[d,"f_2"] & \cdots \arrow[l,swap,"q_3"] \\
0 & C_0(K) \arrow[l] & C_1(K) \arrow[l,swap,"\partial_1"] & C_2(K) \arrow[l,swap,"\partial_2"] & \cdots \arrow[l,swap,"\partial_3"] 
\end{tikzcd}
\end{equation*}
such that $H_\ell(f_\bullet)\colon H_\ell(Q_\bullet)\rightarrow H_\ell\big(C_\bullet(K)\big)$ is an isomorphism for all $\ell\geq 0$. Computing such a resolution for arbitrary posets is a difficult problem. Fortunately, to compute the homology $H_\ell\big(C_\bullet(K)\big)$, we only need to replace the relevant segment of $C_\bullet(K)$ by projective modules. 

\begin{definition}[Projective implicit representation] \label{def:pirep}
Let 
\begin{equation*}
\begin{tikzcd}
C^S\colon &[-20pt] C_{l} & C_c \arrow[l,swap,"\partial_{c}"] & C_{r} \arrow[l,swap,"\partial_r"] \arrow[ll,bend left=15pt, "0"]
\end{tikzcd}
\end{equation*}
be a chain complex segment of $P$-persistence modules. We call a chain complex segment
\begin{equation*}
\begin{tikzcd}
D^S\colon &[-20pt] D_{l} & D_c \arrow[l,swap,"d_{c}"] & D_{r} \arrow[l,swap,"d_{r}"] \arrow[ll,bend left=15pt, "0"]
\end{tikzcd}
\end{equation*}
a projective implicit representation (PiRep) of $H(C^S)\coloneqq \ker(\partial_c)/\im(\partial_{r})$ if $D_x$ is projective for $x\in\{l,c,r\}$ and $H(C^S)\cong \ker(d_c)/\im(d_{r})\eqqcolon H(D^S)$.
\end{definition}

In the following, we abbreviate projective implicit representation as \emph{PiRep}.

\begin{remark}
We note that we could also define a PiRep of $H(C^S)$ as a chain complex segment $D^S$ together with maps: 
\begin{equation*}
\begin{tikzcd}
D^S\colon \arrow[d,swap,"f^S"] &[-20pt] D_{l} \arrow[d,"f_{l}"] & D_c \arrow[l,swap,"d_{c}"] \arrow[d,"f_{c}"] & D_{r} \arrow[l,swap,"d_{r}"] \arrow[d,"f_{r}"] \\ 
C^S\colon &[-20pt] C_{l} & C_c \arrow[l,swap,"\partial_{c}"] & C_{r} \arrow[l,swap,"\partial_{r}"]
\end{tikzcd}
\end{equation*}
such that the squares commute and $H(f^S)\colon H(D^S)\rightarrow H(C^S)$ is an isomorphism. This could be viewed as a ``local'' projective resolution and is more natural from a theoretical perspective. The PiReps we will construct in the following satisfy this stronger property. Nevertheless, we adopt Definition~\ref{def:pirep} to align with the notion of free implicit representations used in the TDA literature. We also note that the most prevalent construction of free implicit representations \cite{chacholski2017combinatorial} does not satisfy this stronger property.
\end{remark}

\section{Constructing PiReps from projective resolutions} \label{sec:pirep_theory}

In this section, we establish the theoretical basis for our method to compute a PiRep of $C_{\ell\minus 1}(K) \xleftarrow{\partial_\ell} C_\ell(K)\xleftarrow{\partial_{\ell+1}} C_{\ell+1}(K)$, as in Definition \ref{def:pirep}, from a poset tower $K$. Our method consists of two steps. In the first step, we compute degreewise projective resolutions of $C_\ell(K)$ up to the second term, together with lifts of the boundary maps $\partial_\ell$ to the resolutions up to the first term, as depicted in \eqref{eq:diagram_homology_from_resolutions} below. In the second step, we assemble the partial resolutions and maps in \eqref{eq:diagram_homology_from_resolutions}, together with an additional correction term $\vartheta_{\ell+1}$, to obtain a PiRep of $H_\ell(K)$. We note that the results presented below hold for arbitrary chain complexes of $P$-persistence modules, but since we focus on poset towers, we present them in this context.

\begin{equation} \label{eq:diagram_homology_from_resolutions}
\begin{tikzcd}
0 & C_{\ell+1}(K) \arrow[d,swap,"\partial_{\ell+1}"] \arrow[l] & G_{\ell+1} \arrow[l,swap,"\alpha_{\ell+1}"] \arrow[d,swap,"f_{\ell+1}^0"] \arrow[ddr,dashed,"\vartheta_{\ell+1}"{xshift=-12pt,yshift=12pt}] & R_{\ell+1} \arrow[l,swap,"p_{\ell+1}^1"] \arrow[d,"f_{\ell+1}^1"] & RR_{\ell+1} \arrow[l,swap,"p_{\ell+1}^2"] \\
0 & C_{\ell}(K) \arrow[d,swap,"\partial_{\ell}"] \arrow[l] & G_\ell \arrow[l,swap,"\alpha_{\ell}"] \arrow[d,swap,"f_\ell^0"] & R_\ell \arrow[l,swap,"p_\ell^1"{xshift=4pt}] \arrow[d,"f_\ell^1"] & RR_\ell \arrow[l,swap,"p_\ell^2"] \\
0 & C_{\ell\minus 1}(K) \arrow[l] & G_{\ell\minus 1} \arrow[l,swap,"\alpha_{\ell\minus 1}"]  & R_{\ell\minus 1} \arrow[l,swap,"p_{\ell\minus 1}^1"]  & RR_{\ell\minus 1} \arrow[l,swap,"p_{\ell\minus 1}^2"] 
\end{tikzcd}
\end{equation}

We use the following notation for the partial resolutions of $C_\ell(K)$, illustrated in the rows of \eqref{eq:diagram_homology_from_resolutions}. The zeroth term of the resolution is denoted by $G_\ell$. It should be thought of as the module of generators of $C_\ell(K)$. We will show in Section \ref{sec:p1correctness} that the set of simplex generators $\mathcal{S}_\ell$ of $K$ gives rise to a minimal set of (projective) generators of $C_\ell(K)$. In the following, we will slightly abuse notation and also denote the projective generators of $G_\ell$ by $g_\sigma^x$. Hence, we obtain $G_\ell=\bigoplus_{g_\sigma^x\in\mathcal{S}_\ell}\proj[x]$. This means that with the right choice of $\alpha_\ell$, as discussed in detail in Section \ref{sec:p1correctness}, $\alpha_\ell\colon G_\ell\rightarrow C_\ell(K)$ defines the first step of a minimal resolution (see Proposition~\ref{prop:alpha_proj_cover}). The first term of the resolution is denoted by $R_\ell$ and should be thought of as the module of relations of the generators $G_\ell$ of $C_\ell(K)$. We will denote its projective generators by $r^x$ and simply call them relations. The second term of the resolution is denoted by $RR_\ell$ and should be thought of as the relations of relations $R_\ell$ of $C_\ell(K)$. We denote its projective generators by $rr^x$ and call them relations of relations. 

Figure \ref{fig:HomologyExample} shows an example of a poset tower and \eqref{eq:poset_tower_presentation} shows the partial resolutions of $C_\ell(K)$ together with the lifts of the boundary maps, as in \eqref{eq:diagram_homology_from_resolutions}. 

\begin{figure}[htbp]
\centerline{\includegraphics[width=0.7\textwidth]{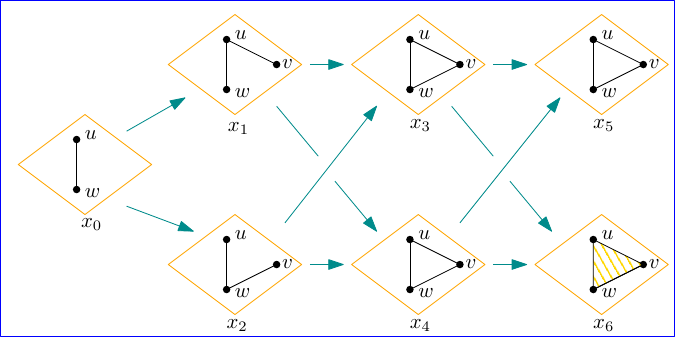}}
\caption{A poset tower over $P=\{x_0,\ldots,x_6\}$ where all the maps are inclusions. Every simplex has a unique generator except $v$, which appears at $x_1$ and $x_2$. The two copies of $v$ get identified at $x_3$ and $x_4$ by relations $r^{x_3},r^{x_4}\mapsto g^{x_1}_v+g^{x_2}_v$. These two relations become the same at $x_5$ and $x_6$, leading to relations of relations $rr^{x_5},rr^{x_6}\mapsto r^{x_3}+r^{x_4}$.}
\label{fig:HomologyExample}
\end{figure}

\begin{equation} \label{eq:poset_tower_presentation}
\begin{tikzcd}[ampersand replacement=\&,every label/.append style = {font = \footnotesize},column sep=large, row sep=large]
G_2 \arrow{d}[swap,xshift=-7pt,yshift=5pt]{\begin{blockarray}{c@{\hspace{7pt}}c}
& g^{x_6}_{uvw} \\
\begin{block}{c@{\hspace{7pt}}(c)}
g^{x_0}_{uw} & 1 \\[2pt]
g^{x_1}_{uv} & 1 \\[2pt]
g^{x_2}_{vw} & 1 \\[2pt]
\end{block}
\end{blockarray}}{f_2^0}  \& 0 \arrow[l,swap,"p_2^1"] \arrow[d,"f_2^1"] \& 0 \arrow[l,swap,"p_2^2"] \\
G_1 \arrow{d}[swap,xshift=-10pt,yshift=-7pt]{\begin{blockarray}{c@{\hspace{7pt}}ccc}
& g^{x_0}_{uw} & g^{x_1}_{uv} & g^{x_2}_{vw} \\
\begin{block}{c@{\hspace{7pt}}(ccc)}
g^{x_0}_{u} & 1 & 1 & 0 \\[2pt]
g^{x_0}_{w} & 1 & 0 & 1 \\[2pt]
g^{x_1}_{v} & 0 & 1 & 0 \\[2pt]
g^{x_2}_{v} & 0 & 0 & 1 \\[2pt]
\end{block}
\end{blockarray}}{f_1^0}   \& 0 \arrow[l,swap,"p_1^1"] \arrow[d,"f_1^1"] \& 0 \arrow[l,swap,"p_1^2"] \\
G_0 \& R_0 \arrow{l}[xshift=-15pt,yshift=-7pt]{\begin{blockarray}{c@{\hspace{7pt}}cc}
& r^{x_3} & r^{x_4} \\
\begin{block}{c@{\hspace{7pt}}(cc)}
g^{x_0}_{u} & 0 & 0 \\[2pt]
g^{x_0}_{w} & 0 & 0 \\[2pt]
g^{x_1}_{v} & 1 & 1 \\[2pt]
g^{x_2}_{v} & 1 & 1 \\[2pt]
\end{block}
\end{blockarray}}[swap]{p_0^1}  \& RR_0 \arrow{l}[xshift=3pt,yshift=-7pt]{\begin{blockarray}{c@{\hspace{7pt}}cc}
& rr^{x_5} & rr^{x_6} \\
\begin{block}{c@{\hspace{7pt}}(cc)}
r^{x_3} & 1 & 1 \\[2pt]
r^{x_4} & 1 & 1 \\[2pt]
\end{block}
\end{blockarray}}[swap]{p_0^2}
\end{tikzcd}
\end{equation}

We now need to assemble the projective modules and matrices in \eqref{eq:diagram_homology_from_resolutions} in such a way that we obtain a PiRep of $H_\ell(K)=\ker(\partial_\ell)/\im(\partial_{\ell+1})$.
We note that this is more or less standard homological algebra, but we are not aware of a good reference explaining this process. We provide proofs of the results in this section, omitted from the main text, in Appendix~\ref{app:proofs}.

We start by discussing how to obtain $\ker(\partial_\ell)$ from the resolutions. The running example in Figure \ref{fig:HomologyExample} and \eqref{eq:poset_tower_presentation} illustrates the intricacies of this process. As one can see, the matrix $f_1^0$ representing $\partial_1$ on the generators has no kernel. The image of the generators representing the triangle $uvw$ is $f_1^0\big(g^{x_0}_{uw}+g^{x_1}_{uv}+g^{x_2}_{vw}\big)=g^{x_1}_v+g^{x_2}_v$. This is because at $x_1$ the only available generator for the boundary vertex $v$ of the edge $uv$ is $g^{x_1}_v$ whereas at $x_2$ only $g^{x_2}_v$ is available for the boundary of the edge $vw$. Hence, the kernel of $f_1^0$ is insufficient to represent the kernel of $\partial_1$, which is the $P$-persistence module $\ker (\partial_1)\colon P\rightarrow \mathbf{Vec}$ that is grade-wise generated by the cycles $\ker \big(\partial_1(x)\big)$ in Figure \ref{fig:HomologyExample}. To get the correct kernel, we have to recognize that at $x_3$ and $x_4$, the two copies of $v$ get identified, which is witnessed by the relations $r^{x_3}$ and $r^{x_4}$. The boundary of the triangle is zero modulo these relations. To account for this, we define the morphism: 
\begin{equation}\label{eq:kernel_generators}
\begin{pmatrix}f_\ell^0 & p_{\ell\minus 1}^1\end{pmatrix}\colon G_\ell\oplus R_{\ell\minus 1}\rightarrow G_{\ell\minus 1}
\end{equation}
whose kernel serves as the generators of $\ker(\partial_\ell)$. In the example, we obtain the matrix:
\begin{equation} \label{eq:kernel_example}
\begin{pmatrix}f_1^0 & p_0^1\end{pmatrix}=
\begin{blockarray}{cccccc}
& g^{x_0}_{uw} & g^{x_1}_{uv} & g^{x_2}_{vw} & r^{x_3} & r^{x_4} \\
\begin{block}{c(ccccc)}
g^{x_0}_{u} & 1 & 1 & 0 & 0 & 0 \\
g^{x_0}_{w} & 1 & 0 & 1 & 0 & 0 \\
g^{x_1}_{v} & 0 & 1 & 0 & 1 & 1 \\
g^{x_2}_{v} & 0 & 0 & 1 & 1 & 1 \\    
\end{block}
\end{blockarray}
\end{equation}
whose kernel is generated by $g^{x_0}_{uw}+g^{x_1}_{uv}+g^{x_2}_{vw}+r^{x_3}$ and $g^{x_0}_{uw}+g^{x_1}_{uv}+g^{x_2}_{vw}+r^{x_4}$. This kernel represents the triangles at $x_3$ and $x_4$ in Figure \ref{fig:HomologyExample}. A problem ensues because these two triangles become the same triangle at $x_5$ and $x_6$. Thus, the kernel of $\partial_1$ has relations at $x_5$ and $x_6$. Since there are no relations between the edges, these relations have to come from relations on $r^{x_3}$ and $r^{x_4}$. In general, relations on the kernel can, of course, also come from relations on the edges. We account for the relations in $G_\ell$ and $R_{\ell\minus 1}$ by defining the following morphism: 
\begin{equation} \label{eq:kernel_relations}
\begin{pmatrix}p_\ell^1 & 0 \\ f_{\ell}^1 & p_{\ell\minus 1}^2\end{pmatrix}\colon R_\ell\oplus RR_{\ell\minus 1}\rightarrow\ker\begin{pmatrix}f_\ell^0 & p_{\ell\minus 1}^1\end{pmatrix}
\end{equation}
from the relations of the domain of \eqref{eq:kernel_generators} to the kernel of $\begin{pmatrix}f_\ell^0 & p_{\ell\minus 1}^1\end{pmatrix}$ where we additionally include the morphism $f^1_{\ell}\colon R_\ell\rightarrow R_{\ell\minus 1}$ from \eqref{eq:diagram_homology_from_resolutions}. The morphism $f^1_{\ell}$ is needed to make \eqref{eq:kernel_relations} well-defined. Since
\begin{equation*}
\begin{pmatrix}f_\ell^0 & p_{\ell\minus 1}^1\end{pmatrix}\circ \begin{pmatrix}p_\ell^1 & 0 \\ f^1_{\ell} & p_{\ell\minus 1}^2\end{pmatrix}=\begin{pmatrix}f_\ell^0\circ p_\ell^1+p^1_{\ell\minus 1}\circ f^1_{\ell} & p^1_{\ell\minus 1}\circ p_{\ell\minus 1}^2\end{pmatrix}=0,
\end{equation*}
we obtain a well-defined map to the kernel. In our running example, the matrix of relations on the kernel is just $p_0^2$, which leads to the identification of the two generators of the kernel at $x_5$ and $x_6$.  

\begin{proposition} \label{prop:kernel_quotient}
The following sequence is exact:
\begin{equation}
\begin{tikzcd}[ampersand replacement=\&,every label/.append style = {font = \small}]
0 \& \ker(\partial_\ell) \arrow[l] \&[15pt] \ker\begin{pmatrix}f_\ell^0 & p_{\ell\minus 1}^1\end{pmatrix} \arrow{l}[swap]{\begin{pmatrix}\alpha_\ell & 0\end{pmatrix}} \&[25pt] R_\ell\oplus RR_{\ell\minus 1} \arrow{l}[swap]{\begin{pmatrix}p_\ell^1 & 0 \\ f^1_{\ell} & p_{\ell\minus 1}^2\end{pmatrix}} .
\end{tikzcd}
\end{equation}
\end{proposition}

\begin{proof}
See Appendix \ref{app:proofs}.
\end{proof}

Proposition \ref{prop:kernel_quotient} shows that $\ker(\partial_\ell)$ is the cokernel of \eqref{eq:kernel_relations}, but this is not a presentation since $\ker\begin{pmatrix}f_\ell^0 & p_{\ell\minus 1}^1\end{pmatrix}$ is not necessarily projective for general $P$.

To obtain a representation of the homology $\ker(\partial_\ell)/\im(\partial_{\ell+1})$, we have to construct a morphism representing $\partial_{\ell+1}\colon C_{\ell+1}\rightarrow \ker(\partial_\ell)$. The morphism $\partial_{\ell+1}$ is represented by $f_{\ell+1}^0$ on the generators but, as one can see in \eqref{eq:poset_tower_presentation} and \eqref{eq:kernel_example}, for the poset tower in Figure \ref{fig:HomologyExample}, $f_2^0$ does not map to the kernel of $\begin{pmatrix}f_1^0 & p_0^1\end{pmatrix}$. The problem is that $f_1^0\circ f_2^0\neq 0$ because the boundary of the triangle is only zero modulo relations. In other words, the morphisms $(f_\ell^0)_\ell$ on the generators do not form a chain complex. However, because they lift the differentials $(\partial_\ell)_\ell$, which compose to zero, they compose to zero up to homotopy. By Definition \ref{def:projective_resolution}(2), $f^0_{\ell}\circ f^0_{\ell+1}$ is a lift of $\partial_{\ell}\circ \partial_{\ell+1}=0$ to the generators. Proposition \ref{prop:existence_resolutions}(3) guarantees that we can extend these partial resolutions and lifts to full resolutions and lifts $f_\ell\circ f_{\ell+1}$. Obviously, the zero map is also a lift of $\partial_{\ell}\circ \partial_{\ell+1}=0$ to these resolutions. Proposition \ref{prop:existence_resolutions}(2) then implies that $f_\ell\circ f_{\ell+1}$ and $0$ are homotopic, which means that there exists a morphism $\vartheta_{\ell+1}\colon G_{\ell+1}\rightarrow R_{\ell\minus 1}$, illustrated by the dashed arrow in \eqref{eq:diagram_homology_from_resolutions}, such that $p_{\ell\minus 1}^1\circ \vartheta_{\ell+1}=f_\ell^0\circ f_{\ell+1}^0+0=f_\ell^0\circ f_{\ell+1}^0$. This motivates the definition of the following morphism:
\begin{equation}\label{eq:image_morphism}
\begin{pmatrix}f_{\ell+1}^0\\ \vartheta_{\ell+1}\end{pmatrix}\colon G_{\ell+1}\rightarrow \ker\begin{pmatrix}f_\ell^0 & p_{\ell\minus 1}^1\end{pmatrix}
\end{equation}
representing $\partial_{\ell+1}$. Again, the additional morphism $\vartheta_{\ell+1}$ is needed to make the map well-defined. Since 
\begin{equation*}
\begin{pmatrix}f_\ell^0 & p_{\ell\minus 1}^1\end{pmatrix}\circ\begin{pmatrix}f_{\ell+1}^0 \\ \vartheta_{\ell+1}\end{pmatrix}=f_\ell^0\circ f_{\ell+1}^0+p_{\ell\minus 1}^1\circ \vartheta_{\ell+1}=0 ,
\end{equation*}
we obtain a well-defined map to the kernel. In the running example (Figure~\ref{fig:HomologyExample}), this homotopy for $f_1^0\circ f_2^0$ can be chosen as
\begin{equation*}
\vartheta_{2}=
\begin{blockarray}{cc}
& g^{x_6}_{uvw} \\
\begin{block}{c(c)}
r^{x_3} & 1 \\
r^{x_4} & 0 \\
\end{block}
\end{blockarray}
\end{equation*}
sending the triangle to a relation annihilating its boundary. 
We now combine \eqref{eq:kernel_relations} and \eqref{eq:image_morphism} to represent $\ker(\partial_\ell)/\im(\partial_{\ell+1})$. Let $\pi\colon \ker(\partial_\ell)\rightarrow \ker(\partial_\ell)/\im(\partial_{\ell+1})$ be the projection to the cokernel.

\begin{proposition}\label{prop:homology_quotient}
The following sequence is exact:
\begin{equation*}
\begin{tikzcd}[ampersand replacement=\&,every label/.append style = {font = \small}]
0 \&[-5pt] \ker(\partial_\ell)/\im(\partial_{\ell+1}) \arrow[l] \&[25pt] \ker\begin{pmatrix}f_\ell^0 & p_{\ell\minus 1}^1\end{pmatrix} \arrow{l}[swap]{\pi\circ\begin{pmatrix}\alpha_\ell & 0\end{pmatrix}} \&[50pt] G_{\ell+1}\oplus R_\ell\oplus RR_{\ell\minus 1}  \arrow{l}[swap]{\begin{pmatrix}f_{\ell+1}^0 & p^1_{\ell} & 0 \\  \vartheta_{\ell+1} & f_{\ell}^1 & p^2_{\ell\minus 1}\end{pmatrix}}   .
\end{tikzcd}
\end{equation*}
\end{proposition}

\begin{proof}
See Appendix \ref{app:proofs}.
\end{proof}

Proposition \ref{prop:homology_quotient} implies that $\ker(\partial_\ell)/\im(\partial_{\ell+1})$ is isomorphic to the cokernel of the matrix built from \eqref{eq:kernel_relations} and \eqref{eq:image_morphism}. The exact sequence in Proposition \ref{prop:homology_quotient} is already a presentation of the homology $\ker(\partial_\ell)/\im(\partial_{\ell+1})$ if $\ker\begin{pmatrix}f_\ell^0 & p_{\ell\minus 1}^1\end{pmatrix}$ is projective. Again, note that this is not necessarily the case for general posets. In any case, we obtain a PiRep of the homology as stated in the following theorem.

\begin{theorem} \label{thm:pirep_main_result}
The following diagram
\begin{equation}\label{eq:projective_complex_homology}
\begin{tikzcd}[ampersand replacement=\&,every label/.append style = {font = \small}]
G_{\ell\minus 1} \&[25pt] G_{\ell}\oplus R_{\ell\minus 1} \arrow{l}[swap]{\begin{pmatrix}f_\ell^0 & p^1_{\ell\minus 1}\end{pmatrix}} \&[50pt] G_{\ell+1}\oplus R_\ell\oplus RR_{\ell\minus 1} \arrow{l}[swap] {\begin{pmatrix}f_{\ell+1}^0 & p^1_{\ell} & 0 \\  \vartheta_{\ell+1} & f^1_{\ell} & p^2_{\ell\minus 1} \end{pmatrix}}
\end{tikzcd}
\end{equation}
is a projective implicit representation of $\ker(\partial_\ell)/\im(\partial_{\ell+1})$.    
\end{theorem}

\begin{proof}
This follows directly from Proposition \ref{prop:homology_quotient}.
\end{proof}

In the case where $\ker\begin{pmatrix}f_\ell^0 & p_{\ell\minus 1}^1\end{pmatrix}$ is not projective, computing a presentation of the homology $\ker(\partial_\ell)/\im(\partial_{\ell+1})$ from a PiRep, as in \eqref{eq:projective_complex_homology}, requires the computation of a projective resolution of $\ker\begin{pmatrix}f_\ell^0 & p_{\ell\minus 1}^1\end{pmatrix}$. Over a general poset, this is not trivial but there are available algorithms for this task. For completeness, we will discuss this part in Appendix \ref{app:homology_presentation}.

In our example, we obtain the following PiRep of $H_1(K)$:
\begin{equation*}
\begin{tikzcd}[column sep=large,ampersand replacement=\&,every label/.append style = {font = \small}]
G_0 \&[110pt] G_1\oplus R_0 \arrow{l}[swap]{\begin{blockarray}{cccccc}
& g^{x_0}_{uw} & g^{x_1}_{uv} & g^{x_2}_{vw} & r^{x_3} & r^{x_4} \\
\begin{block}{c(ccccc)}
g^{x_0}_{u} & 1 & 1 & 0 & 0 & 0 \\
g^{x_0}_{w} & 1 & 0 & 1 & 0 & 0 \\
g^{x_1}_{v} & 0 & 1 & 0 & 1 & 1 \\
g^{x_2}_{v} & 0 & 0 & 1 & 1 & 1 \\    
\end{block}
\end{blockarray}} \&[80pt] G_2\oplus RR_0 \arrow{l}[swap]{\begin{blockarray}{cccccc}
& g^{x_6}_{uvw} & rr^{x_5} & rr^{x_6} \\
\begin{block}{c(ccccc)}
g^{x_0}_{uw} & 1 & 0 & 0 \\
g^{x_1}_{uv} & 1 & 0 & 0 \\
g^{x_2}_{vw} & 1 & 0 & 0 \\
r^{x_3} & 1 & 1 & 1 \\    
r^{x_4} & 0 & 1 & 1 \\   
\end{block}
\end{blockarray}}
\end{tikzcd}
\end{equation*}
In this case, the kernel of $\begin{pmatrix} f_1^0 & p^1_{0}\end{pmatrix}$ is already projective and, after a change of basis, denoted by $\cong$, we obtain the following presentation of the homology: 
\begin{equation*}\small
\begin{pmatrix}f_{2}^0 & p^1_{1} & 0 \\ \vartheta_2 & f_1^1 & p^2_{0}\end{pmatrix}\vert_{\ker\begin{pmatrix} f_1^0 & p^1_{0}\end{pmatrix}}=
\begin{blockarray}{cccc}
& g^{x_6}_{uvw} & rr^{x_5} & rr^{x_6} \\
\begin{block}{c(ccc)}
r^{x_3} & 1 & 1 & 1 \\
r^{x_4} & 0 & 1 & 1 \\
\end{block}
\end{blockarray}
\cong
\begin{blockarray}{cccc}
& rr^{x_5} & g^{x_6}_{uvw} & rr^{x_6} \\
\begin{block}{c(ccc)}
r^{x_3} & 1 & 1 & 0 \\
r^{x_4} & 1 & 0 & 1 \\
\end{block}
\end{blockarray}.
\end{equation*}
It is generated by the two triangular cycles appearing at $x_3$ and $x_4$ in Figure \ref{fig:HomologyExample} which get identified at $x_5$ and killed at $x_6$ by the triangle.

\begin{remark}
On a high level, the construction of this section is a truncated version of a homotopical method that computes a projective resolution of an entire chain complex, as discussed in Section \ref{subsec:pireps}, from projective resolutions of its individual terms. Let $C_\bullet$ be a bounded-below chain complex, $C_\ell\leftarrow Q^\ell_\bullet$ be projective resolutions, and $f^\ell_\bullet\colon Q^\ell_\bullet\rightarrow Q^{\ell\minus 1}_\bullet$ be lifts of the boundary maps $\partial_\ell\colon C_\ell\rightarrow C_{\ell\minus 1}$. If consecutive lifts would compose to zero, $Q_\bullet^\bullet$ would form a double complex whose total complex is quasi-isomorphic to $C_\bullet$; see \cite{Weibel_1994} Section 2.7. But for fixed resolutions $Q_\bullet^\bullet$, it may be impossible to choose lifts with this property; see for example \eqref{eq:poset_tower_presentation}. Proposition \ref{prop:existence_resolutions} guarantees that compositions of these lifts are homotopic to zero. Thus, the resolutions and lifts form a chain complex in the homotopy category. 

However, to construct a total complex, these homotopies must be retained explicitly as additional data. Moreover, these homotopies are also only compatible up to a choice of higher homotopy. The appropriate structure is therefore a homotopy-coherent analogue of a double complex that records all the higher homotopical structure. Its homotopy-coherent total complex includes the resolutions and the lifted boundary maps as well as all contributions from the chosen homotopies and higher homotopies and is quasi-isomorphic to $C_\bullet$. This homotopy-coherent total complex construction is discussed in \cite{Rickard1989Morita} in a different setting. The existence of this homotopy coherent structure in our setting follows from the fact that the homotopies guaranteed by Proposition \ref{prop:existence_resolutions} are again unique up to higher homotopy and the analogous uniqueness holds at every higher level.

This homotopy coherent approach is different from the more well-known Cartan-Eilenberg method; see \cite{Weibel_1994} Section 5.7. That construction uses projective resolutions of the boundary and homology modules of $C_\bullet$ to obtain compatible resolutions that directly form a double complex. But this is not suitable for our purpose because it requires resolving the homology modules first. Our aim is to compute a PiRep to facilitate an efficient global homology computation.

Because we only focus on a segment of $C_\bullet$, we only need to compute degreewise resolutions up to the second term and we only have to incorporate one component of the degree one homotopies. Truncating the chain complex allows us to ignore the higher order structure which would account for longer range interactions.
\end{remark}

\section{PiRep algorithm}\label{sec:algorithm}

\subsection{Input, output, and data structures}

In this section, we describe an efficient algorithm to compute a PiRep of $H_\ell(K)$, as in Theorem \ref{thm:pirep_main_result}, from a poset tower $K\colon P\rightarrow\mathbf{SCpx}$. The PiRep consists of the two block matrices in \eqref{eq:projective_complex_homology}, whose entries come from projective resolutions of $C_\ell(K)$ up to the second term, lifts of the boundary maps $\partial_\ell$ up to the first term, and a correction map $\vartheta_{\ell+1}$, as depicted in \eqref{eq:diagram_homology_from_resolutions}.  Therefore, our algorithm has to compute matrix representations of the maps $p_\ell^1$, $p_\ell^2$ representing the resolutions of $C_\ell(K)$, $f_\ell^0$, $f_\ell^1$ representing the lifts of $\partial_\ell$, and $\vartheta_{\ell+1}$. Instead of performing an algebraic brute-force computation, we make essential use of the special structure of the modules $C_\ell(K)$ and morphisms $\partial_\ell$ to compute all these maps in a combinatorial fashion from the following representation of the poset tower $K$. \\ 

\noindent
\textbf{Input:} A poset tower $K\colon P\rightarrow \mathbf{SCpx}$ represented by:
\begin{itemize}
    \item The Hasse diagram of $P$ as a directed (transitively reduced) graph. 
    \item A list 
    $\mathcal{S}=\{g_{\sigma_1}^{x_1},\ldots,g_{\sigma_{n'}}^{x_{n'}}\}$ of simplex generators. 
    \item A list $\mathcal{C}=\{c^{y_1\prec z_1}_{v_{i_1}\mapsto v_{j_1}},\ldots, c^{y_{n''}\prec z_{n''}}_{v_{i_{n''}}\mapsto v_{j_{n''}}}\}$ of edge events. 
\end{itemize}
The complexity of the proposed algorithm is measured with respect to the input size, which includes (i) $t=t_0+t_1$, the number of vertices ($t_0$) and edges ($t_1$) of the graph representing $P$, (ii) $n=n'+n''$, the number of simplex generators in $\mathcal S$ plus the number of edge events in $\mathcal C$. Additionally, we assume that the maximal simplex dimension is constant. \\

\noindent
\textbf{Output:} $P$-graded sparse matrix representations (containing only the non-zero entries) of 
\begin{equation*}
\begin{pmatrix}f_\ell^0 & p^1_{\ell\minus 1}\end{pmatrix}  \hspace{5pt} \text{ and } \hspace{10pt} \begin{pmatrix}f_{\ell+1}^0 & p^1_{\ell} & 0 \\  \vartheta_{\ell+1} & f^1_{\ell} & p^2_{\ell\minus 1} \end{pmatrix}.
\end{equation*}

As already discussed in Section \ref{sec:pirep_theory}, the set of simplex generators $\mathcal{S}_\ell$ gives rise to a minimal set of generators of $C_\ell(K)$ allowing us to define $G_\ell\coloneqq\bigoplus_{g_\sigma^x\in\mathcal{S}_\ell}\proj[x]$ (see Proposition~\ref{prop:alpha_proj_cover}). Given this set of generators, we have to determine how they are related. A presentation of $C_\ell(K)$ based on $G_\ell$ can be built using the following three kinds of relations. The fact that these three types of relations are sufficient to form a complete presentation is a consequence of the proof of correctness of our algorithm in Section \ref{sec:p1correctness}, which constructs the presentation by adding exactly these relations.
\begin{enumerate}
    \item Two generators $g_\sigma^x$ and $g_\sigma^y$ of the same simplex $\sigma$ born at different grades can join at a grade $b>x,y$ where they must be identified by a relation $r^b\mapsto g_\sigma^x+g_\sigma^y$.
    \item An $\ell$-simplex $\sigma\in K(a)$ represented by a generator $g_\sigma^x$ can be collapsed into an $\ell$-simplex $\tau\in K(b)$ represented by a generator $g_\tau^y$ via a simplicial map such that $K(a<b)(\sigma)=\tau$. In this case, they also have to be identified by a relation $r^b\mapsto g_\sigma^x+g_\tau^y$.
    \item An $\ell$-simplex $\sigma\in K(a)$ represented by a generator $g_\sigma^x$ can be collapsed into a simplex $\tau\in K(b)$ of dimension $<\ell$ via a simplicial map such that $K(a<b)(\sigma)=\tau$. In this case, the generator has to be killed by a relation $r^b\mapsto g_\sigma^x$.
\end{enumerate}
To provide some intuition why generators in a poset tower are always related in pairs, we look at the example in \eqref{eq:pairwise_relation_example_1}, where three copies of the same simplex $\sigma$ join at $x_4$. 
\begin{equation} \label{eq:pairwise_relation_example_1}
\begin{tikzcd}
& x_4 & & & \sigma & & & \field\\
x_1\arrow[ur] & x_2 \arrow[u] & x_3 \arrow[ul]  & \sigma \arrow[ur] & \sigma \arrow[u] & \sigma \arrow[ul]  & \field \arrow[ur] & \field \arrow[u] & \field \arrow[ul] \\[-10pt]
& P & & & K & & & C_\ell(K)
\end{tikzcd}
\end{equation}
These three copies cannot be related by a single relation at $x_4$ because the presentation matrix in \eqref{eq:pairwise_relation_example_2} on the right would induce a two-dimensional cokernel at $x_4$. Instead, they have to be related in pairs as shown in \eqref{eq:pairwise_relation_example_2} on the left.
\begin{equation}\label{eq:pairwise_relation_example_2}
p_\ell^1=\begin{blockarray}{ccc}
& r^{x_4} & r^{x_4} \\
\begin{block}{c(cc)}
g^{x_1}_\sigma & 1 & 0 \\
g^{x_2}_\sigma & 1 & 1 \\
g^{x_3}_\sigma & 0 & 1 \\
\end{block}
\end{blockarray} 
\hspace{30pt} , \hspace{30pt}
p_\ell^1\neq\begin{blockarray}{cc}
& r^{x_4}  \\
\begin{block}{c(c)}
g^{x_1}_\sigma & 1 \\
g^{x_2}_\sigma & 1 \\
g^{x_3}_\sigma & 1 \\
\end{block}
\end{blockarray}
\end{equation}
The complexity of finding these relations arises because they depend on each other. Two generators $g_\sigma^x$ and $g_\sigma^y$ should not be identified at a common join if one of them is already killed by a relation at this point. In this case, they are different simplices. However, we also cannot just treat them as different simplices because they might have to be identified at some other join. Moreover, many generators can meet multiple times at multiple points along the poset tower. Therefore, we need an efficient method to keep track of which generators have to be related, which simplices are already related, and which simplices are already dead. 

\textbf{Graph structure.} Our key observation is that we can model the relations by a $\overline{P}$-filtered graph, where $\overline{P}=P\cup\{-\infty\}$ such that $-\infty<x$ for all $x\in P$. 
The generators in a poset tower always get related in pairs (Cases 1 and 2) except in the case of simplices getting totally collapsed (Case 3). But we can model the third case by a generator that gets related to a distinguished ``graveyard'' vertex $\Omega$ at grade $-\infty$. Together with this distinguished vertex $\Omega$, we obtain a $\overline{P}$-filtered graph $\mathcal{G}\colon \overline{P}\rightarrow\Delta\mathbf{Cpx}_{\leq 1}$ (Definition \ref{def:boundary_p}) whose other vertices are the generators $g_\sigma^x\in\mathcal{S}$, born at grade $x$, and whose edges are the relations $r^b\mapsto g_\sigma^x+g_\tau^y$ or $r^b\mapsto g_\sigma^x+\Omega$, born at grade $b$. In particular, each vertex and edge of the $\overline{P}$-filtered graph $\mathcal{G}$ is born at a unique grade in $\overline{P}$. Therefore, we can directly represent the boundary matrix of $\mathcal{G}$ by a $\overline{P}$-graded matrix. We formally define $\mathcal{G}$ as a one-dimensional $\overline{P}$-filtered $\Delta$-complex, because there can be multiple edges between two vertices. The following theorem is a direct consequence of the proof of correctness, in Section \ref{sec:correctness}, of the upcoming algorithm {\sc PiRep} (see Definition \ref{def:boundary_p} and Theorem \ref{thm:correct_minimal_presentation}).

\begin{theorem} \label{thm:graph_presentation_matrix}
Let $K\colon P\rightarrow \mathbf{SCpx}$ be a poset tower. Then there exists a $\overline{P}$-filtered graph $\mathcal{G}\colon \overline{P}\rightarrow\Delta\mathbf{Cpx}_{\leq 1}$, where each vertex and edge is born at a unique grade $x\in\overline{P}$, with $\overline{P}$-graded boundary matrix $\overline{p}^1_\ell$ such that the matrix $p^1_\ell$, obtained from $\overline{p}^1_\ell$ by removing the row corresponding to a distinguished vertex $\Omega$ at $-\infty$, is a minimal presentation of $C_\ell(K)$.
\end{theorem}

At every grade $x\in P$, the connected component of $\Omega$ in $\mathcal{G}(x)$ represents the simplices that already died. Each remaining connected component represents a simplex in $K(x)$. Moreover, the cycles in $\mathcal{G}$ represent the kernel of the presentation matrix $p^1_\ell$ which corresponds to the relations of the relations (see Definition \ref{def:projective_resolution}(1)). In the following algorithm {\sc PiRep}, we successively build the graph $\mathcal{G}$ by adding generators and relations between them in the form of vertices and edges. If a new relation is added that forms a cycle in $\mathcal{G}(x)$, then this new relation is a linear combination of other existing relations, which should not be added to the minimal presentation (see Definition \ref{def:projective_cover}). Nevertheless, a cycle that is generated by the inclusion of edges from predecessors implies a relation of relations.    

The special structure of $p_\ell^1$ is also crucial for the efficient computation of $f_\ell^1$ and $\vartheta_{\ell+1}$. These morphisms are determined by the equations $p^1_{\ell\minus 1}\circ f^1_{\ell}=f_\ell^0\circ p_\ell^1$ and $p_{\ell\minus 1}^1\circ \vartheta_{\ell+1}=f_\ell^0\circ f_{\ell+1}^0$, i.e., by $P$-graded linear systems as in Definition \ref{def:P_graded_system}. By Proposition \ref{prop:eliminate_constraints}, for a specific column of $f_\ell^1$, corresponding to a relation $r^x$, or a specific column of $\vartheta_{\ell+1}$, corresponding to a generator $g_\sigma^x$, these equations can be solved via a standard system of linear equations, where the coefficient matrix is the boundary matrix of $\mathcal{G}(x)$. Using this graph structure, we can solve these systems in linear time.   

\textbf{Matrices.} 
Our algorithm computes matrix representations of $p^1_\ell$, $p^2_\ell$, $f_\ell^0$, $f_\ell^1$, and $\vartheta_{\ell+1}$. All matrices are implemented with a sparse representation that contains only non-zero entries. As a result, the space and time complexities for computing and representing them may be less than their actual size given by the product of their row and column dimensions. While analyzing the time complexities of the algorithms, we assume that the matrices
are represented sparsely.

The rows of $p_\ell^1$ and the rows and columns of $f_\ell^0$ are associated to generators $g_\sigma^{x}$, i.e., we maintain a bijective map $I\rightarrow \{g_\sigma^{x}\}$ for the index set $I$ for rows. Similarly, the columns of $p_\ell^1$, the rows and columns of $f_\ell^1$, and the rows of $p_\ell^2$ are associated to relations $r^x$ and are indexed by a graded identifier. This identifier is the column index paired with the grade $x$ if the column is added at the grade $x$. In other words, denoting this identifier as $\rsf^x$, we maintain a bijective map $J\rightarrow \{\rsf^x\}$ for the column index set $J$. Therefore, we write $p_\ell^1[g_\sigma^x,\rsf^y]$ for $p_\ell^1[i,j]$ if $i\mapsto g_\sigma^x$ and $j\mapsto \rsf^y$. We use this notation to address entries of a matrix in our algorithm to simplify presentation even though the matrix is implemented with a sparse representation. The columns of the matrix $p^2_\ell$ are associated to relations of relations and are indexed by a graded global counter
$\rrsf$. A column gets the identifier $\rrsf^x$ if the column is generated at grade
$x$. We write $p^2_\ell[\rsf^x,\rrsf^y]$ in place of $p^2_\ell[i,j]$ if $i\mapsto \rsf^x$ and $j\mapsto \rrsf^y$.

In what follows, the algorithms handle all degrees $\ell\geq 0$ together and thus maintain single matrices $p^1$, $p^2$, $f^0$, $f^1$, and $\vartheta$ accounting for all degrees.

The main routine {\sc PiRep} computes a linear extension $x_0,\ldots, x_{t_0\minus 1}$
of the poset $P$. Then it iterates over each grade $x_i$ maintaining the invariant
that, for every degree $\ell\geq 0$, it has correctly computed matrix representations of the maps $p^1_\ell$, $p^2_\ell$, $f_\ell^0$, $f_\ell^1$, and $\vartheta_{\ell+1}$ restricted to the subposet induced by grades $x_0,\ldots, x_i$. We note that the choice of linear extension is arbitrary, and different choices may yield different output matrices. However, since the matrix representations of (even minimal) presentations and projective resolutions of modules are generally not unique, this behavior is expected. In Section \ref{sec:correctness}, we prove that any choice of linear extension produces a correct PiRep.

The algorithm maintains a set $L_{act}^{x_i}$ of active generators, i.e., a set of pairs $(\sigma,g_\tau^{y})$ such that the generator $g_\tau^y$ represents the simplex $\sigma$ at the grade $x_i$, and a forest $\F_{act}^{x_i}$
of active relations. The set $L_{act}^{x_i}$ helps to determine
the relations $R_\ell$ (columns of $p_\ell^1$) that may arise due to
vertex collapses handled by the subroutine {\sc Collapse} or by
identification of generators handled by the subroutine {\sc Generator}.
The forest $\F_{act}^{x_i}$, initialized with a special vertex $\Omega$, helps to determine superfluous relations
in $R_\ell$ and thus helps guarantee minimality of $p_\ell^1$ and also
to determine relations of relations $RR_\ell$ registered in the matrix $p_\ell^2$ by the subroutine {\sc RelRel}. The special vertex $\Omega$ 
plays the role of a ``graveyard'' vertex, which gets attached to any edge in the
relation graph generated by a ``vanishing'' simplex due to collapse.

The subroutines take the arguments in the parentheses as input and update the global matrices $p^1$, $p^2$,  $f^0$, $f^1$, and $\vartheta$ along with other global structures such as an active list $L_{act}^\bullet$ and
graphs $\mathcal{G}_{act}^\bullet$ and $\mathcal{F}_{act}^\bullet$ as part of their output.

\subsection{Algorithm}

\noindent
{\bf Algorithm {\sc PiRep}} ($P$, $\mathcal C$, $\mathcal S$)
\begin{itemize}
    \item Compute a linear extension $x_0,\ldots,x_{t_0\minus 1}$ of $P$;
    \item Initialize empty matrices $p^1$, $p^2$, $f^0$, $f^1$, $\vartheta$, and counters $\rsf=0$, $\rrsf=0$;
    \item For $i:=0$ to $t_0-1$ do 
    \begin{enumerate}
        \item Call {\sc Collapse}($P$, $\mathcal C$, $x_i$);
        \item Call {\sc Generator}($P$, $\mathcal S$, $x_i$);
        \item Call {\sc Lift}($x_i$);
        \item Call {\sc RelRel}($x_i$)
    \end{enumerate}
\end{itemize}

In the routine {\sc Collapse}, for every poset edge $y\prec x$, we handle collapsing a vertex $v$ to a vertex $w$. In step 1, we initialize an empty list of active generators $L_{act}^x$ which we keep associated with the grade $x$ for subsequent use. In step 2, we compute the union of forests
$\cup_{y\prec x}\F_{act}^{y}$ resulting in a graph $\G_{act}^{x}$
that may not be a forest. This graph (relation graph) is used and updated 
in the successive steps and routines (Figure~\ref{fig:relgraph}). In step 3(a), every simplex $\sigma\in K(y)$ 
gets its vertices updated either remaining the same ($v\mapsto v$)
or renamed to $w$ ($v\mapsto w$); e.g. in Figure~\ref{fig:posettower}, the edges
$uv, uw$ and $vw$ at $x_4$ become $uw, uw$ and $ww$ at $x_5$, respectively.

If a simplex $\sigma$ in $K(y)$ after vertex renaming contains none or
only one of $v$ and $w$, it is included in $L_{act}^{x}$ with the appropriate generator (step 3(b.i));
e.g. $uv$ from $x_4$ becoming $uw$ at $x_5$ is added to $L_{act}^{x_5}$ with generator $g_{uv}^{x_1}$. Similarly, the edge $uw$ from $x_4$ remaining $uw$ at $x_5$ is added
with the generator $g_{uw}^{x_2}$. Now, $L_{act}^{x_5}$ has two copies of
$uw$ with two different generators. They are identified by the routine {\sc Generator}.

Otherwise, in step 3(b.ii), the vanishing simplex $\sigma$, such as the edge $vw$ becoming $ww$, is not included in $L_{act}^{x}$; instead, a relation is generated for $p^1$ unless the addition of this relation to the graph $\G_{act}^{x_i}$ generates a cycle. For the vanishing edge $vw$, we add the edge $e^{x_5}_{(\Omega,g_{vw}^{x_3})}$ to the graph $\G^{x_5}_{act}$
which does not create a cycle (Figure~\ref{fig:relgraph} (bottom right)). Therefore, we create
a column indexed by $\rsf^{x_5}$ for the relation $r_1^{x_5}\mapsto g_{vw}^{x_3}$ setting
$p^1[g_{vw}^{x_3},\rsf^{x_5}]=1$.

\begin{figure}[htbp]
\centerline{\includegraphics[width=\textwidth]{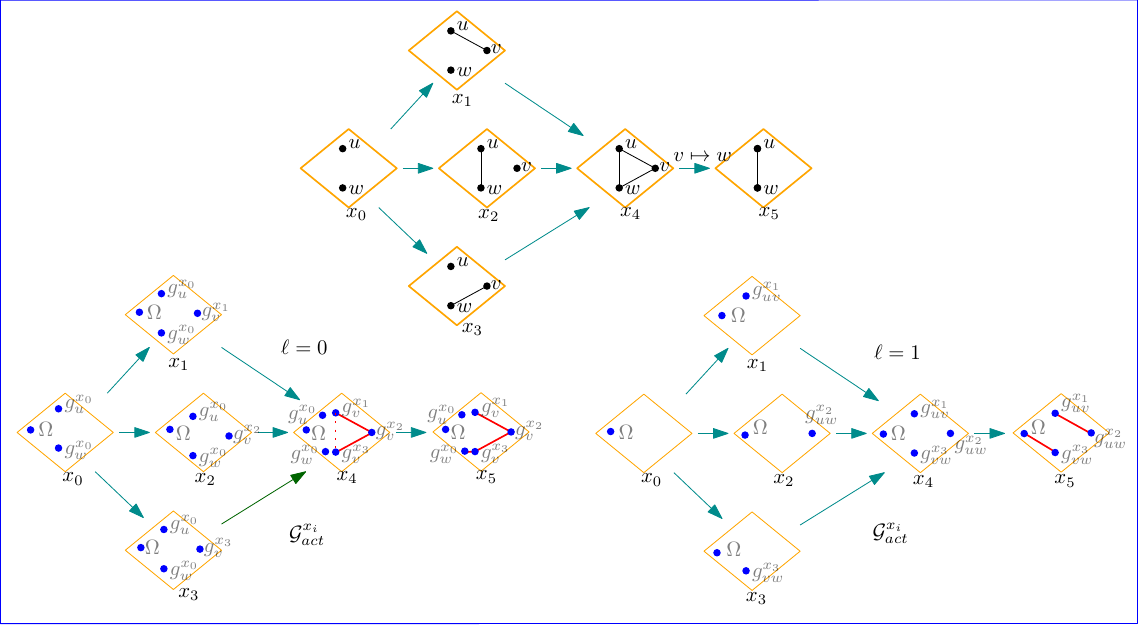}}
\caption{(Figure~\ref{fig:posettower} example) Relation graphs $\G^{x_i}_{act}$ for vertices ($\ell=0$) on left and for edges ($\ell=1$) on right.}
\label{fig:relgraph}
\end{figure}

\vspace{0.1in}
\noindent
{\bf Algorithm {\sc Collapse}} ($P$, $\mathcal C$, $x$)
\begin{enumerate}
\item  Initialize $L_{act}^x\coloneqq\emptyset$ and keep it associated with grade $x$; 
\item If $x$ is not minimal in $P$, compute $\G_{act}^{x}:= \cup_{y\prec x} \F_{act}^y$, else initialize $\G_{act}^{x}$ with vertex $\Omega$; 
\item For all $y\prec x$ do
        \begin{enumerate}
            \item Compute $K(y\prec x)$ as $v\mapsto w$ if $c_{v\mapsto w}^{y\prec x}\in \mathcal{C}$ and as $v\mapsto v$ else; 
            \item For all $(\sigma, g_\rho^z)\in L^y_{act}$ do
            \begin{enumerate}
                \item If $K(y\prec x)(\sigma)$ has no repeated vertices
                \begin{itemize}
                    \item Add $\big(K(y\prec x)(\sigma),g_\rho^z\big)$ to $L_{act}^{x}$;
                \end{itemize}
                \item Otherwise, if adding $e^{x}_{(\Omega,g_\rho^{z})}$ to $\G_{act}^{x}$ does not make a cycle
                 \begin{itemize}
                    \item Add edge $e^{x}_{(\Omega,g_\rho^{z})}$ to $\G_{act}^{x}$;
                    \item Add a column with index $\rsf^x$ in $p^1$ with $p^1[g_\rho^{z},\rsf^{x}]=1$; $\rsf\texttt{++}$
                 \end{itemize}
            \end{enumerate}
        \end{enumerate}
\end{enumerate}

In the routine {\sc Generator}, we add a pair $(\sigma, g_\sigma^x)$ to the active set
$L_{act}^{x}$ if $\sigma$ has a generator at the current grade $x$. 
After adding a new generator $g_\sigma^x$, we record its boundary map in $f^0$ by calling the routine {\sc Boundary} (step 1). In {\sc Boundary}, we find active generators for the boundary of $\sigma$ and map the new generator of $\sigma$ to these boundary generators. Because we add simplices in order of increasing simplex dimension and because each simplex is either generated at grade $x$ or is the image of a simplex in a predecessor, at this point, each boundary simplex has at least one active generator.
If the active set contains a simplex $\sigma$ with two different generators, they are
identified by a relation. If this relation does not
create a cycle in the relation graph $\G_{act}^{x}$, we add it to both
$\G_{act}^{x}$ and $p^1$. Otherwise, we ignore it as a superfluous relation.

Again, consider the example in Figure~\ref{fig:posettower}. First, consider
the case for edges ($\ell=1$). At $x_5$, we have $(uw,g_{uv}^{x_1})$ and
$(uw,g_{uw}^{x_2})$ in $L_{act}^{x_5}$ from the routine {\sc Collapse}.
We identify the two copies of $uw$, and add the edge $e^{x_5}_{(g_{uv}^{x_1}, g_{uw}^{x_2})}$
to the graph $\G^{x_5}_{act}$ whose addition does not create a cycle; see Figure~\ref{fig:relgraph} (bottom right). Therefore, we
generate a relation $r_2^{x_5}\mapsto g_{uv}^{x_1}+g_{uw}^{x_2}$ (step 2(a)), and delete, say, $(uw,g_{uw}^{x_2})$
(step 2(b)). A column for $r_2^{x_5}$ is added to $p^1$.

Next, consider the case for vertices ($\ell=0$). At $x_4$, three
generators $g_v^{x_1}$, $g_v^{x_2}$, and $g_v^{x_3}$ of $v$ generate the pairs
$(v,g_v^{x_1}), (v,g_v^{x_2})$, and $(v,g_v^{x_3})$ in $L_{act}^{x_4}$ after the
{\sc Collapse} call. Step 2 potentially creates three relations $r_1^{x_4}\mapsto g_v^{x_1}+g_v^{x_2}$, $r_2^{x_4}\mapsto g_v^{x_2}+g_v^{x_3}$, and
$r_3^{x_4}\mapsto g_v^{x_1}+g_v^{x_3}$. Assuming that edges for relations
$r_1^{x_4}$ and $r_2^{x_4}$ are added first in $\G_{act}^{x_4}$ (Figure~\ref{fig:relgraph} (bottom left)), the edge for $r_3^{x_4}$ (shown dotted) creates a cycle. So, it is not added to $\G^{x_4}_{act}$, and equivalently
only columns for relations $r_1^{x_4}$ and $r_2^{x_4}$ are added to $p^1$.

At the end (step 3), we compute a spanning forest $\F_{act}^{x}$ of $\G_{act}^{x}$
which we keep associated with the grade $x$ for subsequent use. 
In Figure~\ref{fig:relgraph} (bottom left), $\F^{x}_{act}$ happens to be the
same as $\G^{x}_{act}$ for every $x$. However, it is possible that, at a grade $x$, relation forests from predecessors come together to form a cycle, see e.g. grade $x_4$ in Figure~\ref{fig:relrelgraph}. In this case, the relation forest and the
relation graph would be different and the edges $(e_{(g_\sigma^{y_1},g_\sigma^{y_2})},
e_{(g_\sigma^{y_2},\Omega)})$ not included in the forest
would give rise to relations of relations, adding two columns to $p^2$.

\begin{figure}[htbp]
\centerline{\includegraphics[width=0.5\textwidth]{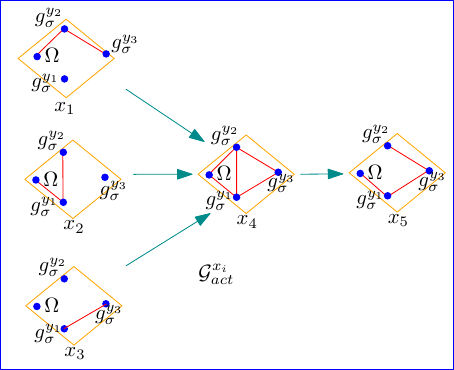}}
\caption{Relation graphs $\G^{x_i}_{act}$; poset tower not shown; the relation forest $\F_{act}^{x_4}$, which
is passed to $x_5$, is a proper subgraph of $\G^{x_4}_{act}$.}
\label{fig:relrelgraph}
\end{figure}

\vspace{0.1in}
\noindent
{\bf Algorithm {\sc Generator}} ($P$, $\mathcal S$, $x$)
\begin{enumerate}
    \item For all $g_\sigma^x\in \mathcal{S}$ (in increasing order of simplex dimension) do
    \begin{itemize}
        \item Add $(\sigma,g_{\sigma}^x)$ to $L_{act}^{x}$;
        \item Add vertex
    $g_\sigma^{x}$ to $\G_{act}^{x}$;
        \item Call {\sc Boundary}($x$, $\sigma$);
    \end{itemize} 
    \item  For every ($(\sigma,g_\tau^y), (\sigma,g_\rho^z)\in L_{act}^{x}$) and ($g_\tau^y \neq g_\rho^z$) do
            \begin{enumerate}
                \item If adding edge $e^{x}_{(g_\tau^y,g_\rho^z)}$ to $\G_{act}^{x}$
                does not create a cycle
                \begin{itemize}
                    \item Add edge $e^{x}_{(g_\tau^y,g_\rho^z)}$ to $\G_{act}^{x}$;
                    \item Add column $\rsf^{x}$ in $p^1$ with $p^1[g_\tau^{y},\rsf^{x}]=1$ and
                $p^1[g_\rho^{z},\rsf^{x}]=1$; $\rsf\texttt{++}$;
                \end{itemize}
                \item Delete $(\sigma,g_\rho^z)$ from $L_{act}^{x}$;
            \end{enumerate}
    \item Compute a spanning forest $\F_{act}^{x}$ of $\G_{act}^{x}$ and keep it associated with grade $x$
\end{enumerate}

\noindent
{\bf Algorithm {\sc Boundary}} ($x$, $\sigma$)
\begin{enumerate}
    \item Find $(\tau_1,g_{\kappa_1}^{y_1}),\ldots,(\tau_k,g_{\kappa_k}^{y_k})$ in $L_{act}^{x}$ s.t. $\sum_{i=1}^k \tau_i=\partial \sigma$;
    \item Add column $g_\sigma^x$ in $f^0$ with $f^0[g_{\kappa_j}^{y_j},g_\sigma^{x}]=1$ $\forall j\in \{1,\ldots, k\}$
\end{enumerate}

In the routine {\sc Lift}, we compute the matrices $f^1$ and $\vartheta$, by lifting the maps $f^0\circ p^1$ and $f^0\circ f^0$ along $p^1$. Let $A[:,\text{ind}]$ denote the column of a matrix $A$ with the column index $\text{ind}$. Then the lifting is carried out by solving the equations $p^1\circ f^1[:,\rsf^{x}]=f^0\circ p^1[:,\rsf^{x}]$ and $p^1\circ\vartheta[:,g_\sigma^x]=f^0\circ f^0[:,g_\sigma^x]$ at the current grade $x$, for all unknown columns $f^1[:,\rsf^{x}]$ and $\vartheta[:,g_\sigma^x]$, indexed by $\rsf^x$ and $g_\sigma^x$, respectively.

The columns of the coefficient matrix $p^1$ and the rows of the right-hand side $f^0\circ p^1$ and $f^0\circ f^0$ correspond to edges and vertices in the total relation graph $\mathcal{G}$ associated to $p^1$, respectively (see Theorem \ref{thm:graph_presentation_matrix}). By Proposition \ref{prop:eliminate_constraints}, to solve the equations, we only have to consider those columns and rows that are indexed by a grade less than or equal to $x$, which correspond to the edges and vertices of $\mathcal{G}(x)$. Moreover, it is enough to consider only the columns corresponding to the edges in a spanning forest of $\mathcal{G}(x)$, since these columns already yield the full column rank. 

The algorithm maintains such a spanning forest $\mathcal{F}_{act}^x$ (see Lemma \ref{prop:spanning_forest}). For each equation $p^1\circ f^1[:,\rsf^{x}]=f^0\circ p^1[:,\rsf^{x}]$, it first computes the right-hand side $b[\mathcal{F}^x_{act}\setminus\Omega]\coloneqq f^0\circ p^1[:,\rsf^{x}]$ and $b[\Omega]\coloneqq\overline{1}\circ f^0\circ p^1[:,\rsf^{x}]$, where $b[\mathcal{F}^x_{act}\setminus\Omega]$ is a vector indexed as the rows of $f^0\circ p^1$, $\overline{1}=\begin{pmatrix}1 & \cdots & 1\end{pmatrix}$, and $b[\Omega]$ is an additional row that accounts for the additional vertex $\Omega$ in the relation graph. For a connected component $C$ of $\mathcal{F}_{act}^x$, let $V(C)$ and $E(C)$ denote its vertices and edges. Then for each such component, the algorithm calls the routine {\sc TreeSolve} (see details in Appendix \ref{app:tree_solve}) to  solve the equation $\overline{p}^1[V(C),E(C)]\circ f^1[E(C),\rsf^{x}]=b[V(C)]$ in linear time. Here $\overline{p}^1[V(C),E(C)]$ is the boundary matrix of $C$, which is $p^1$ with an additional row accounting for the vertex $\Omega$ (see Theorem \ref{thm:graph_presentation_matrix}) restricted to the vertices and edges in the component $C$. The computation for $\vartheta$ is analogous.

\vspace{0.1in}
\noindent
{\bf Algorithm {\sc Lift}} ($x$)
\begin{enumerate}
    \item For each $\rsf^{x}$ added at grade $x$ do
    \begin{enumerate}
        \item Compute $b[\mathcal{F}^x_{act}\setminus\Omega]\coloneqq f^0\circ p^1[:,\rsf^{x}]$ and $b[\Omega]\coloneqq\overline{1}\circ f^0\circ p^1[:,\rsf^{x}]$;
        \item For each connected component $C$ of $\mathcal{F}^x_{act}$ do 
        \begin{itemize}
            \item Set $f^1[E(C),\rsf^{x}]\coloneqq${\sc TreeSolve}($C,b[V(C)]$);
        \end{itemize}

    \end{enumerate}
    \item For each $g^x_\sigma$ added at grade $x$ do
    \begin{enumerate}
        \item Compute $b[\mathcal{F}^x_{act}\setminus\Omega]\coloneqq f^0\circ f^0[:,g_\sigma^x]$ and $b[\Omega]\coloneqq\overline{1}\circ f^0\circ f^0[:,g_\sigma^x]$;
        \item For each connected component $C$ of $\mathcal{F}^x_{act}$ do 
        \begin{itemize}
            \item Set $\vartheta[E(C),g_\sigma^x]\coloneqq${\sc TreeSolve}($C,b[V(C)]$)
        \end{itemize}
    \end{enumerate}
\end{enumerate}

The routine {\sc RelRel} creates relations of relations adding columns to
the matrix $p^2$. It uses the spanning forest $\F_{act}^{x}$ of
the relation graph $\G_{act}^{x}$ available
at this stage. For every edge in $\G_{act}^{x}$ that is not in
$\F_{act}^{x}$, it computes a cycle, say $\gamma$, and
registers a relation of relations using the edges on $\gamma$. It is registered as a column
in $p^2$. For example, in Figure~\ref{fig:relrelgraph}, the two edges
$e_{(g_\sigma^{y_1},g_\sigma^{y_2})},
e_{(g_\sigma^{y_2},\Omega)}$ add two columns $\rrsf_1^{x_4}$ and $\rrsf_2^{x_4}$
to $p^2$.
Finally, there is an option to call {\sc Reduce} which minimizes $p^2$.

\vspace{0.1in}
\noindent
{\bf Algorithm {\sc RelRel}} ($x$)
\begin{itemize}
      \item For every edge $e_*^{z_0}\in  \G_{act}^{x}\setminus \F_{act}^{x}$ do
        \begin{enumerate}
            \item Find the path of edges $e_*^{z_1},\ldots, e_*^{z_s}$ in $\F_{act}^{x}$
            connecting the vertices of $e_*^{z_0}$;
            \item Add a relation of relations $\rrsf^{x}$ in $p^2$ where
            $p^2[\rsf^{z_j},\rrsf^{x}]=1$ $\forall j\in\{0,\ldots,s\}$; $\rrsf\texttt{++}$;
        \end{enumerate}
        \item (Optional) Call {\sc Reduce} ($x$)
\end{itemize}

The routine {\sc Reduce} removes redundant relations of relations by left-to-right column reduction of an appropriate submatrix of $p^2$. This makes $p^2$ minimal as we prove in Proposition~\ref{prop:proj_cover_p2}.

\vspace{0.1in}
\noindent
{\bf Algorithm {\sc Reduce}} ($x$)
\begin{enumerate}
    \item Compute the submatrix $A^{\leq x}$ of $p^2$ consisting of columns with indices
    $\rrsf^y$ such that $y\leq x$ and the column order 
    respects the chosen linear extension of grades;
    \item Reduce $A^{\leq x}$ (with left-to-right column additions as done in standard persistent homology)
    such that each column either has
    a unique row as a pivot or is completely zeroed out;
    \item Remove a column with index $\rrsf^{x}$ from $p^2$ if the column is reduced
    to a zero column in the reduced matrix $A^{\leq x}$
\end{enumerate}

\subsection{Time complexity analysis} \label{subsec:complexity_analysis}

We analyze the time complexity of the algorithm {\sc PiRep}. Let $t_0$ and $t_1$ be the total number of vertices (grades) and edges, respectively, in the directed acyclic graph (Hasse diagram) representing the poset $P$ and let $t=t_0+t_1$. Let $n$ denote the number $|\Scal|$ of generators plus the number $|\Ccal|$ of edge events input to the algorithm. 
Thus, the maximum number of simplices
in the complex $K(x)$ is $O(n)$ for any $x$. The maximum dimension
$d$ of a simplex is assumed to be constant for this analysis. For this analysis, we denote by $w_x$ the number of $y$ in $P$ such that $y\prec x$ and note that $\sum_{x\in P}w_x=t_1$.

We first analyze the maximal size of the data structures $L_{act}^x$, $\mathcal{F}_{act}^x$, and $G_{act}^x$ at a given grade $x$. At a minimal element in $P$, the size of $L_{act}^x$ is clearly $O(n)$ because we add at most $O(n)$ simplex generators to it in {\sc Generator}. For each $y\prec x$, $L_{act}^y$ contains at most one pair $(\sigma,g_\tau^y)$ for each $\sigma \in K(y)$ at the end of {\sc Generator} because we remove all but one such pair. Hence, at the end of {\sc Generator} $L_{act}^y$ is of size $O(n)$. This implies that {\sc Collapse} adds at most $O(w_xn)$ elements and {\sc Generator} adds at most $O(n)$ elements to $L_{act}^x$. Therefore, the maximal size of $L_{act}^x$ at any point in the algorithm is $O\big((w_x+1)n\big)$.

The set of vertices of $\mathcal{F}_{act}^x$ always corresponds to a subset of the set of simplex generators $\Scal$, so it has at most $O(n)$ vertices. Thus, the forest $\mathcal{F}_{act}^x$ is of size $O(n)$ at any point of the algorithm. The graph $\mathcal{G}_{act}^x$ is first initialized in {\sc Collapse} as the union of all $\mathcal{F}_{act}^y$ with $y\prec x$. Thus, it has size $O(w_x n)$ after initialization. Also, as there are at most $O(n)$ vertices overall, $\mathcal{G}_{act}^x$ contains at most $O(n)$ vertices throughout. In particular, $\mathcal{G}_{act}^x$ has at most $O(n)$ components. Because in {\sc Collapse} and {\sc Generator} we only add edges if they do not create a cycle, each added edge must connect two components. This implies that we can add at most $O(n)$ edges at each grade $x$ and thus the graph $\mathcal{G}_{act}^x$ has size $O\big((w_x+1)n\big)$ at any point of the algorithm.

\vspace{2pt}
{\bf Analysis of {\sc Collapse}.}
First assume that $w_x=0$, i.e., the grade $x$ has no predecessors. In this case, we only initialize $\mathcal{G}_{act}^x$ with the vertex $\Omega$, which can be done in constant time. Over all grades $x\in P$ such that $w_x=0$ this takes at most $O(t_0)$ time. Notice that the poset $\overline{P}=P\cup\{-\infty\}$ has a single minimum at $\{-\infty\}$ where the vertex $\Omega$ is born, but our algorithm directly operates on $P$. Thus, we have to initialize a vertex $\Omega$ at each minimal element of $P$, which are the only grades where $w_x=0$.

Now assume that $w_x>0$. Step 1 clearly takes constant time. In step 2, we construct $\mathcal{G}_{act}^x$ as a union of $w_x$ forests $\mathcal{F}_{act}^y$ of size $O(n)$. Since we have a fixed set of vertices, corresponding to the simplex generators, this can be done in $O(w_xn)$ time. 

For a given $y\prec x$, let $n_y\leq n$ denote the number of simplices in $K(y)$ and $n_{y\prec x}$ denote the number of edge events on the poset edge $y\prec x$. Step 3(a) can be done using a dictionary (implemented with a simple array of vertex ids) for $y\prec x$ initialized as $v\mapsto v$ for all vertices in $L^y_{act}$, which has length at most $n_y$. Then we go over at most $n_{y\prec x}$ edge events and set $v\mapsto w$ for $c_{v\mapsto w}^{y\prec x}$. In step 3(b), we update at most $n_y$ simplices using this dictionary, which can be done in constant time because we assume the maximum dimension $d$ of a simplex to be a constant. The same is true for checking whether there is a repeated vertex in step 3(b.i).

For step 3(b.ii), we can use a union-find data structure~\cite{CLRbook} to check efficiently
whether an added edge creates a cycle in the graph $\G^x_{act}$. Such a union-find data structure for the components of $\mathcal{G}_{act}^x$ can be initialized in $O\big(nw_x\alpha(n)\big)$ time (once per grade $x$), where $\alpha(\cdot)$ is the 
extremely slowly growing inverse Ackermann function. Over all grades this adds a total cost of $O\big(t_1n\alpha(n)\big)$. Note that if there are no predecessors we do not need the union-find structure. With this union-find data structure, we can check if an edge creates a cycle in $O\big(\alpha(n)\big)$ time. An edge creates a cycle if and only if the two endpoints are already in the same component. If they are in different components, the edge does not create a cycle and the two components are merged. Therefore, step (3) takes at most $O\big(n_y+n_{y\prec x}+n_y\alpha(n)\big)=O\big(n\alpha(n)\big)$ time for each $y\prec x$. Over all grades $x\in P$ such that $w_x>0$ and all $w_x$ predecessors $y\prec x$ this results in a worst case complexity of $\sum_x O\big(w_x n\alpha(n)\big)=O\big(nt_1\alpha(n)\big)$. Hence, overall {\sc Collapse} has a worst-case complexity of $O\big(nt_1\alpha(n)+t_0\big)$.   

\vspace{2pt}
{\bf Analysis of {\sc Generator}.}
Step 1 in the routine {\sc Generator} accesses the generators in $\Scal$ and their
boundaries. Finding the simplex generators born at grade $x$ and adding them to $L_{act}^x$ takes at most $O(n)$ time and we add at most $O(n)$ generators. When adding these new generators, we also update the union-find data structure for the components of $\mathcal{G}_{act}^x$, initialized in {\sc Collapse}, which can be done in constant time, since each generator added at this point forms a new component. For each added generator, finding active generators of its boundary can be done in $O(d)$ time, assuming that simplices and their boundaries are maintained
with an appropriate pointer data structure. Assuming $d$ to be a constant, we get a total
complexity of this step as $O(n)$ per grade and $O(nt_0)$ over all grades. 

In step 2, we identify all pairs $(\sigma,g^y_\tau)$ and $(\sigma,g_\rho^z)$ with the same simplex \texttt{id}. This can be implemented efficiently as follows: We maintain $L_{act}^{x}$ with an array of lists containing all pairs $(\sigma, g_\tau^y)$ indexed by the unique \texttt{id} of the simplex $\sigma$. Then for each simplex $\sigma$ that appears in that array, we keep the first list entry $(\sigma,g^y_\tau)$ indexed by $\sigma$ and remove all following $(\sigma,g_\rho^z)$. For each removed $(\sigma,g_\rho^z)$, we have to check if adding the edge $e^{x}_{(g_\tau^y,g_\rho^z)}$ creates a cycle in $\mathcal{G}_{act}^x$. As discussed in the analysis of {\sc Collapse}, this can be done in $O\big(\alpha(n)\big)$ time using the union-find data structure on the components of $\mathcal{G}_{act}^x$. We have already analyzed that $L_{act}^x$ is of size at most $O\big((w_x+1)n\big)$. Therefore, we have to check for at most $O\big((w_x+1)n\big)$ many $(\sigma,g_\rho^z)$. Hence, step (2) takes $O\big((w_x+1)n\alpha(n)\big)$ time per grade and $O\big(nt\alpha(n)\big)$ time after summing over all grades. 

In step (3), we compute a spanning forest $\mathcal{F}_{act}^x$ of $\mathcal{G}_{act}^x$. As we have analyzed above, the graph $\mathcal{G}_{act}^x$ is of size $O\big((w_x+1)n\big)$. Computing a spanning forest can be done in $O(\vert\text{Vertices}\vert+\vert\text{Edges}\vert)$ using breadth-first or depth-first search. Thus, step (3) costs $O\big( (w_x+1)n\big)$ per grade and $O(nt)$ over all grades. Therefore, overall the routine {\sc Generator} has a worst-case time complexity of $O\big(nt\alpha(n)\big)$. 

Note that the matrix $p^1$ gets at most $O(n)$ columns at every grade $x$
because they correspond to the edges of the forest $\F^{x}_{act}$ (we only add edges that connect two components). Therefore, overall we add at most $O(nt_0)$ columns over all grades and $p^1$ has at most $O(nt_0)$ columns. Each column of $p^1$ has at most two entries in sparse matrix representation. Thus, the overall time (and memory) cost of creating a representation of $p^1$ is $O(nt_0)$. 

\vspace{2pt}
{\bf Analysis of {\sc Lift}.} In step $1(a)$, we first compute $f^0\circ p^1[:,\rsf^x]$. Since each column of $p^1$ has at most two entries and each column of $f^0$ has at most $d+1$ entries, computing this product is equivalent to an XOR of two columns of $f^0$ in sparse matrix format and takes constant time ($d$ is assumed to be constant). Moreover,  $f^0\circ p^1[:,\rsf^x]$ has at most $2(d+1)$ entries, which implies that computing $\overline{1}\circ f^0\circ p^1[:,\rsf^x]$, i.e., summing up all entries, can also be done in constant time. Let $n_C$ denote the size of the connected component $C$ of $\mathcal{F}_{act}^x$. Since $\mathcal{F}_{act}^x$ is of size $O(n)$, we have $O(\sum_{C} n_C)=O(n)$. By Theorem \ref{prop:tree_solve} in Appendix \ref{app:tree_solve}, the call of {\sc TreeSolve} on $C$ costs $O(n_C)$ time. Thus, overall step $1(b)$ takes time $O(n)$. We execute step $1$ exactly once for each $\rsf^x$ added at grade $x$. Since there are at most $O(nt_0)$ relations, the total cost of step $1$ over all grades $x\in P$ is $O(n^2t_0)$. For step $2(a)$, we again note that $f^0$ has at most $d+1$ entries. Thus, step $2(a)$ can also be done in constant time. The complexity of step $2(b)$ is $O(n)$ by the same argument as for step $1(b)$. We execute step $2$ exactly once for every generator $g_\sigma^x$ added by the algorithm at grade $x$. Since there are at most $O(n)$ generators, the total cost of step $2$ over all grades $x\in P$ is $O(n^2)$. Thus, in total {\sc Lift} has a worst-case time complexity of $O(n^2t_0)$.

\vspace{2pt}
{\bf Analysis of {\sc RelRel}.} Steps 1 and 2 take at most $O(n)$ time per edge
$e_*^{z_0}$ since $\F_{act}^{x}$ has size $O(n)$. We have that $\G^{x}_{act}\setminus \mathcal{F}_{act}^x$ is of size $O(w_xn)$ because after initialization $\mathcal{G}_{act}^x$ is of size $O(nw_x)$ and every edge added afterwards to $\mathcal{G}_{act}^x$ is in $\mathcal{F}_{act}^x$. Counting over all edges of $\G^{x}_{act}\setminus \mathcal{F}_{act}^x$, we get a complexity of $O(n^2w_x)$. Then, over all grades, we get a bound of $O(\sum_x n^2w_x )=O(n^2t_1)$.
Notice that, at every grade $x$, we add at most $O(w_xn)$ columns to $p^2$. Therefore, $p^2$ has at most $O(nt_1)$ columns. Because each column of $p^2$ corresponds to the relations/edges on a path in the forest $\mathcal{F}_{act}^x$ plus one additional edge closing the cycle, it has at most $O(n)$ entries. Thus, the overall time (and memory) cost of creating a representation of $p^2$ is $O(n^2t_1)$.

\vspace{2pt}
{\bf Analysis of {\sc Reduce}.} The most costly step in {\sc Reduce} is the matrix
reduction. This can be done in $O(s^\omega)$ time 
if the matrix has $O(s)$ columns and rows, where $\omega< 2.373$ is the exponent for the
matrix multiplication time complexity \cite{theoretics11261}. By the analysis of {\sc RelRel}, the matrix $A^{\leq x}$ being a submatrix
of $p^2$, cannot have more than $O(nt_1)$ columns. Moreover, since the entries in the columns of $p^2$ correspond to relations (columns of $p^1$), in non-sparse representation $p^2$ has at most $O(nt_0)$ rows (see analysis of {\sc Generator}). Then, {\sc Reduce} takes $O\big(n^\omega t^\omega)$ time per grade and $O\big(n^\omega t^{\omega+1})$ time for all grades.

The proof of correctness in Section \ref{sec:correctness} together with Theorem \ref{thm:pirep_main_result} establishes that the algorithm {\sc PiRep} computes a projective implicit representation of $H_\ell(K)$. It also establishes that $p^1_\ell$ is a minimal presentation (Theorem \ref{thm:p1min}), that $p^2_\ell$ without the call of {\sc Reduce} is asymptotically minimal (Proposition \ref{prop:asymptotic_optimality}), and that $p^2_\ell$ with the call of {\sc Reduce} is minimal (Theorem \ref{thm:minimal_resolution}). We summarize this in our main theorem:

\begin{theorem} \label{thm:algorithm_complexity}
Let $K\colon P\rightarrow \mathbf{SCpx}$ be a poset tower represented by the Hasse diagram of $P$ with $t_0$ nodes and $t_1$ edges, together with the lists $\mathcal S$ and $\mathcal C$ of simplex generators and edge events. Set $t=t_0+t_1$ and $n=|\mathcal S|+|\mathcal C|$. Then the following statements hold for all $\ell$, where (1) uses only {\sc Collapse} and {\sc Generator}, (2) additionally uses {\sc RelRel}, (3) additionally uses  {\sc Reduce}, and (4) refers to the full algorithm without {\sc Reduce}.  \vspace{1pt}
\begin{enumerate}
\item {\sc PiRep} computes a minimal presentation matrix $p_\ell^1$ for $C_\ell(K)$ with $O(nt_0)$ columns in $O\big(nt\alpha(n)\big)$ time. \vspace{1pt}

\item {\sc PiRep} computes matrices $p_\ell^1$ and $p_\ell^2$, representing a projective resolution of $C_\ell(K)$ up to the second term, in $O(n^2t)$ time. The computed matrix $p_\ell^2$ has $O(nt_1)$ columns, which is asymptotically minimal because a minimal $p_\ell^2$ can have $\Theta(nt_1)$ columns. \vspace{1pt}

\item {\sc PiRep} computes a minimal $p_\ell^2$ in $O\big(n^\omega t^{\omega+1})$ time. \vspace{1pt}

\item {\sc PiRep} computes a projective implicit representation of $H_\ell(K)$ in $O(n^2t)$ time.
\end{enumerate}
\end{theorem}

\section{Correctness of the algorithm \textsc{PiRep}} \label{sec:correctness}

In this section, we prove the correctness of the {\sc PiRep} algorithm, hereafter referred to as ``the algorithm'' for brevity. The overall proof strategy is to show that the algorithm correctly computes matrix representations of the maps $p_\ell^1$, $p_\ell^2$, $f_\ell^0$, $f_\ell^1$, and $\vartheta_{\ell+1}$, shown in the diagram \eqref{eq:diagram_homology_from_resolutions}. This means that $p_\ell^1$ and $p_\ell^2$ form a partial resolution of $C_\ell(K)$ up to the second term, $f_\ell^0$ and $f_\ell^1$ form a partial lift of $\partial_\ell$ up to the first term, and $\vartheta_{\ell+1}$ satisfies $p_{\ell\minus 1}^1\circ \vartheta_{\ell+1}=f_\ell^0\circ f_{\ell+1}^0$. Assembling these maps into the block matrices in \eqref{eq:projective_complex_homology} then yields a PiRep of $H_\ell(K)$ by Theorem \ref{thm:pirep_main_result}. In Section \ref{sec:p1correctness}, we show that $p^1_\ell$, as constructed by {\sc PiRep}, is a valid presentation and that $f_\ell^0$ is a valid lift of $\partial_\ell$ to the generators. In Section \ref{sec:p1minimality}, we show that the constructed $p_\ell^1$ is a minimal presentation. In Section \ref{sec:p2correctness}, we show that $p_\ell^1$ and $p_\ell^2$ together form a partial resolution up to the second term. Moreover, we show that $p_\ell^2$ is asymptotically minimal and even minimal if {\sc Reduce} is called in {\sc PiRep}. In Section \ref{subsec:graph_equations}, we show that the maps $f_\ell^1$ and $\vartheta_{\ell+1}$ satisfy the required commutativity relations. 

\subsection{Correctness of computing $p_\ell^1$ and $f_\ell^0$}
\label{sec:p1correctness}

In this section, we prove that $p^1_\ell$, as constructed in the algorithm {\sc PiRep}, for a poset tower $K\colon P\rightarrow\mathbf{SCpx}$, is a presentation of $C_\ell(K)$. While the algorithm processes all degrees $\ell$ simultaneously, for the proof we will w.l.o.g. fix a degree $\ell$. Also, we abbreviate $C_\ell(K)$ by $C_\ell$. We note that the algorithm implicitly constructs a set of generators $g_\sigma^x$ and relations $r^y$ as the indices of the rows and columns of the matrices $p_\ell^1$, respectively. Hence, using the notation established in Section \ref{subsec:projective_modules}, the domain and codomain of the maps of projective modules $p_\ell^1\colon R_\ell\rightarrow G_\ell$, represented by the matrices, are given by $R_\ell= \langle r^y\vert r^y\in \text{ column indices of }p_\ell^1 \rangle$ and $G_\ell= \langle g_{\sigma}^x\vert g_\sigma^x\in \text{ row indices of }p_\ell^1 \rangle$. We start by showing that $G_\ell$ generates $C_\ell$.  

\begin{proposition} \label{prop:correct_generators}
For $K\colon P\rightarrow\mathbf{SCpx}$, let $G_\ell= \langle g_{\sigma}^x\vert g_\sigma^x\in \text{ row indices of }p_\ell^1 \rangle$ be the codomain of $p_\ell^1$ constructed by the algorithm. Then 
\begin{equation*}
G_\ell=\bigoplus_{g^x_\sigma\in\mathcal{S}_\ell}\proj[x] . 
\end{equation*}
\end{proposition}

\begin{proof}
The algorithm adds every generator $g^x_\sigma\in\mathcal{S}_\ell$ at the grade $x\in P$ in the routine {\sc Generator}.    
\end{proof}

At every grade $x\in P$, the vector space $G_\ell(x)$ has a basis $\{g_\sigma^y\in \mathcal{S}_\ell\vert y\leq x\}$ where $g^y_\sigma$ is the contribution of the summand $\proj[y]$. The vector space $C_\ell(x)$ has a basis generated by the $\ell$-simplices $K_\ell(x)$. We denote the basis elements of $C_\ell(x)$ simply by $\sigma$.

\begin{definition} \label{def:alpha}
Define the morphism $\alpha_\ell\colon G_\ell\rightarrow C_\ell$ in the following way:
\begin{equation*}
\alpha_\ell(x)(g^y_\sigma)=\begin{cases} K(y\leq x)(\sigma) & \mbox{ if } K(y\leq x)(\sigma)\in K_\ell(x) \\ 0 &   \text{ otherwise} \end{cases}.
\end{equation*}   
\end{definition}

In the following, we will use the convention $K(y\leq x)(\sigma)=0$, for $\sigma\in K_\ell(y)$, if $K(y\leq x)(\sigma)\notin K_\ell(x)$. Also note that $g^y_\sigma\in\mathcal{S}_\ell$ implies $\sigma\in K_\ell(y)$.

\begin{proposition} \label{prop:alpha_well_defined}
The morphism $\alpha_\ell$ is well-defined.
\end{proposition}

\begin{proof}
Let $x\leq y\in P$ and $g^z_\sigma\in \mathcal{S}_\ell$ such that $z\leq x$. Then 
\begin{equation*}
\begin{aligned}
C_\ell(x\leq y)\circ \alpha_\ell(x)(g^z_\sigma)&=C_\ell(x\leq y)\big(K(z\leq x)(\sigma)\big) \\
&=K(x\leq y)\circ K(z\leq x)(\sigma)\\&=K(z\leq y)(\sigma).    
\end{aligned}
\end{equation*}
On the other hand, 
\begin{equation*}
\alpha_\ell(y)\circ G_\ell(x\leq y)(g^z_\sigma)=\alpha_\ell(y)(g^z_\sigma)=K(z\leq y)(\sigma).    
\end{equation*}
Hence, the diagram
\begin{equation*}
\begin{tikzcd}
C_\ell(y) & G_\ell(y) \arrow[l,swap,"\alpha_\ell(y)"] \\
C_\ell(x) \arrow[u,"C_\ell(x\leq y)"] & G_\ell(x)  \arrow[l,swap,"\alpha_\ell(x)"] \arrow[u,swap,"G_\ell(x\leq y)"]
\end{tikzcd}
\end{equation*}
commutes and $\alpha_\ell$ is well-defined.    
\end{proof}

\begin{proposition} \label{prop:alpha_epi}
The morphism $\alpha_\ell$ is an epimorphism.
\end{proposition}

\begin{proof}
Let $\sigma\in K_\ell(x)$. If $\sigma\notin \bigcup_{y\prec x}\im\!\big( K(y\prec x)\big)$, then $g^x_\sigma\in\mathcal{S}_\ell$ and $\alpha_\ell(x)(g^x_\sigma)=\sigma$. Otherwise, $\sigma=K(y\prec x)(\tau)$ for some $\tau\in K_\ell(y)$. If $\tau$ is generated at $y$, there is $g^y_\tau$ in $\mathcal{S}_\ell$, otherwise we can iterate this argument until we find $\rho\in K_\ell(z)$ such that $g^z_\rho\in\mathcal{S}_\ell$ and $K(z\leq x)(\rho)=\sigma$. Note that this works because we reach a minimal element $a\in P$ after a finite number of steps and, by definition, every $\pi\in K_\ell(a)$ is a generator $g^a_\pi\in \mathcal{S}_\ell$. This implies
\begin{equation*}
\begin{aligned}
\alpha_\ell(x)(g^z_\rho)&=\alpha_\ell(x)\circ G_\ell(z\leq x)(g^z_\rho)\\&=C_\ell(z\leq x)\circ \alpha_\ell(z)(g^z_\rho) \\ &=C_\ell(z\leq x)(\rho)\\&=K(z\leq x)(\rho)=\sigma .
\end{aligned}    
\end{equation*}
Since $C_\ell(x)$ has a basis $K_\ell(x)$, $\alpha_\ell(x)$ is an epimorphism for all $x\in P$ and, thus, $\alpha_\ell$ is an epimorphism.    
\end{proof}

Proposition \ref{prop:alpha_epi} shows that $G_\ell$ does indeed generate $C_\ell$. The following propositions establish that $p_\ell^1$ is a presentation of $C_\ell$, a result we summarize in Theorem \ref{thm:correct_presentation}. To prove them, we make use of the following lemmas. 

\begin{lemma} \label{lem:act_gen_simp}
At every grade $x_i$ in the algorithm {\sc PiRep}, after the call of {\sc Generator}, $L^{x_i}_{act}\cong K(x_i)$ as sets. In particular, for each $\sigma\in K(x_i)$, there exists an active generator $(\sigma,g_\tau^y)\in L_{act}^{x_i}$.
\end{lemma}

\begin{proof}
We proceed by induction on the linear extension of the grades in $P$. The grade $x_0$ is minimal in $P$ as any predecessor would have to come before $x_0$ in the linear extension. At a minimal element $x_0\in P$, for every $\sigma\in K(x_0)$ there is a $g_\sigma^{x_0}\in\mathcal{S}$. Therefore, the identification $(\sigma,g^{x_0}_\sigma)\sim \sigma$ yields $L_{act}^{x_0}\cong K(x_0)$.

Assume that $L_{act}^{x_j}\cong K(x_j)$ for all $0\leq j < i$. For every $k$-simplex $\sigma\in \bigcup_{y\prec x_i}\im \!\big(K(y\prec x_i)\big)$, there exists a $y\prec x_i$ and an $l$-simplex $\tau\in K(y)$, with $l\geq k$, such that $K(y\prec x_i)(\tau)=\sigma$. By dropping superfluous vertices from $\tau$, we can always find a $k$-simplex $\tau'<\tau\in K(y)$ such that $K(y\prec x_i)(\tau')=\sigma$. Thus, w.l.o.g.\ we can assume $l=k$ and no vertex of $\tau$ gets collapsed by $K(y\prec x_i)$.  By assumption, there exists $(\tau,g^z_\rho)$ in $L^y_{act}$. By calling {\sc Collapse}, the algorithm adds $(K(y\prec x_i)(\tau),g^z_\rho\big)=(\sigma,g^z_\rho)$ to $L^{x_i}_{act}$. For every $\sigma\in K(x_i)\setminus \bigcup_{y\prec x_i}\im \!\big(K(y\prec x_i)\big)$ there exists $g^{x_i}_\sigma\in\mathcal{S}$ which is added as $(\sigma,g^{x_i}_{\sigma})$ to $L_{act}^{x_i}$ in step 1 of {\sc Generator}. Thus, every $\sigma\in K(x_i)$ has at least one representative in $L_{act}^{x_i}$. In step 2 of {\sc Generator}, $(\sigma,g^y_\tau)$ is deleted if there is another element $(\sigma,g^z_\rho)$ in $L_{act}^{x_i}$. Therefore, after this step, there is exactly one representative $(\sigma,g^z_\rho)$ for every $\sigma\in K(x_i)$ and $L_{act}^{x_i}\cong K(x_i)$.
\end{proof}

\begin{lemma} \label{lemma:algo_1}
Suppose that after the call of the routines {\sc Collapse} and {\sc Generator} in {\sc PiRep}, $(\tau,g^y_\sigma)\in L_{act}^x$ is active at $x\in P$. Then $K(y\leq x)(\sigma)=\tau$. 
\end{lemma}

\begin{proof}
Since $g^y_\sigma\in\mathcal{S}$, we have that $(\sigma,g^y_\sigma)$ is added to $L_{act}^y$ by {\sc Generator} at $y\in P$. Since active simplices are added from $\mathcal{S}$ by {\sc Generator} or inherited from predecessors by {\sc Collapse}, $(\tau,g^y_\sigma)$ can only be active at $x$ if $y\leq x$ and if there is a path $y=z_0\prec\cdots \prec z_k=x$ in $P$ along which $g^y_\sigma$ is active as $(\rho_i,g^y_\sigma)$ at $z_i$ with $\rho_0=\sigma$ and $\rho_k=\tau$. By the update procedure of the {\sc Collapse} routine in {\sc PiRep}, $\rho_i=K(z_{i\minus 1}\prec z_i)(\rho_{i\minus 1})$. Therefore, $K(y\leq x)(\sigma)=K(z_{k\minus 1}\prec z_k)\circ\cdots\circ K(z_0\prec z_1)(\rho_0)=\rho_k=\tau$.      
\end{proof}

Let $r^{y_1},\ldots,r^{y_n}$ be the relations added by the algorithm {\sc PiRep} for the $\ell$-simplices. Then we have $R_\ell=\bigoplus_{i=1}^n\proj[y_i]$. As for the generators $G_\ell$, at each grade $x\in P$ the vector space $R_\ell(x)$ has a basis $\{r^{y}\vert y\leq x\}$.   

\begin{proposition} \label{prop:p_exact_1}
Let $G_\ell\xleftarrow{p_\ell^1}R_\ell$ be the matrix for the $\ell$-simplices constructed by the {\sc PiRep} algorithm. Then $\alpha_\ell\circ p_\ell^1=0$.
\end{proposition}

\begin{proof}
First, let $r^y\in R_\ell$ be a relation, added by the {\sc Collapse} routine, such that $p_\ell^1(r^y)=g^z_\sigma$. The algorithm adds the relation at $y\in P$ because there exists $a\prec y\in P$ such that $(\tau,g^z_\sigma)$ is active at $a$ but $K(a\prec y)(\tau)=0$. For $y\leq x\in P$ we obtain:
\begin{equation*}
\alpha_\ell(x)\circ p_\ell^1(x)(r^y)=\alpha_\ell(x)(g^z_\sigma)=K(z\leq x)(\sigma).    
\end{equation*}
Since $(\tau,g^z_\sigma)$ is active at $a$, Lemma~\ref{lemma:algo_1} implies that $z\leq a$ and 
\begin{equation*}
\begin{aligned}
K(z\leq x)(\sigma)&=K(y\leq x)\circ K(a\leq y)\circ K(z\leq a)(\sigma)\\ &=K(y\leq x)\circ K(a\leq y)(\tau)=0 .
\end{aligned}
\end{equation*}
Thus, $\alpha_\ell(x)\circ p_\ell^1(x)(r^y)=0$. 

Now, let $r^y\in R_\ell$ be a relation, added by the algorithm, such that $p_\ell^1(r^y)=g^z_\sigma+g^w_{\tau}$. Since a simplex generator $g_\xi^y$ is added only when
$\xi$ is not in the image of any predecessor of $y$, no relation
$r^y$ created by {\sc Generator} at grade $y$ can involve a generator
added at grade $y$. Indeed, any other active representative of the same
simplex $\xi$ would have to be inherited from a predecessor of $y$,
contradicting the defining property of $g_\xi^y$. Hence, the algorithm adds the relation $r^y$ in the {\sc Generator} routine at $y\in P$ because there exist $a\prec y$ and $b\prec y$ such that $(\rho,g^z_\sigma)$ is active at $a$, $(\kappa,g^w_{\tau})$ is active at $b$ and $\pi\coloneqq K(a\prec y)(\rho)=K(b\prec y)(\kappa)$. For $y\leq x$, we get 
\begin{equation*}
\begin{aligned}
\alpha_\ell(x)\circ p_\ell^1(x)(r^y)&=\alpha_\ell(x)(g^z_\sigma+g^w_{\tau})\\ &=\alpha_\ell(x)(g^z_\sigma)+\alpha_\ell(x)(g^w_{\tau})\\ &=K(z\leq x)(\sigma)+K(w\leq x)(\tau).    
\end{aligned}
\end{equation*}
By Lemma~\ref{lemma:algo_1}, we further obtain 
\begin{equation*}
\begin{aligned}
K(z\leq x)(\sigma)&=K(y\leq x)\circ K(a\leq y)\circ K(z\leq a)(\sigma)\\ &=K(y\leq x)\circ K(a\leq y)(\rho)\\ &=K(y\leq x)(\pi)   
\end{aligned} 
\end{equation*}
and analogously 
\begin{equation*}
\begin{aligned}
K(w\leq x)(\tau)&=K(y\leq x)\circ K(b\leq y)\circ K(w\leq b)(\tau)\\ &=K(y\leq x)\circ K(b\leq y)(\kappa)\\ &=K(y\leq x)(\pi).    
\end{aligned}
\end{equation*}
Therefore, $\alpha_\ell(x)\circ p_\ell^1(x)(r^y)=0$. 

We have shown that for all $x\in P$ and all relations $r^y$ of $R_\ell$ such that $y\leq x$ we have $\alpha_\ell(x)\circ p_\ell^1(x)=0$. This implies $\alpha_\ell\circ p_\ell^1=0$ and $\im(p_\ell^1)\subseteq \ker(\alpha_\ell)$.    
\end{proof}

\begin{lemma}\label{lemma:algo_2}
Let $g^x_\sigma\in\mathcal{S}_\ell$ be such that $\tau:=K(x\leq y)(\sigma)\in K_\ell(y)$. Then after the call of {\sc Collapse} and {\sc Generator} in the algorithm, there exists $(\tau,g^z_\rho)\in L^y_{act}$ and $r\in R_\ell$ such that $g^x_\sigma+g^z_\rho=p_\ell^1(r)$.
\end{lemma}

\begin{proof}
Let $x=z_0\prec\cdots\prec z_m=y$ be a path from $x$ to $y$ in $P$. Moreover, let $\pi_i=K(z_0\leq z_i)(\sigma)$ for $0\leq i\leq m$. In particular, $\pi_0=\sigma$, $\pi_m=K(x\leq y)(\sigma)=\tau$ and, since $\tau\in K_\ell(y)$, we have $\pi_i\in K_\ell(z_i)$. By Lemma \ref{lem:act_gen_simp}, each $\pi_i$ has an active generator $(\pi_i,g_{\kappa_i}^{a_i})$ at $z_i$. At $x=z_0$, the generator $(\sigma,g^x_\sigma)=(\pi_0,g_{\kappa_0}^{a_0})$ is active. 

In the routine {\sc Collapse} at grade $z_i$, we add $\big(K(z_{i\minus 1}\prec z_i)(\pi_{i\minus 1}),g_{\kappa_{i\minus 1}}^{a_{i\minus 1}}\big)=\big(\pi_i,g_{\kappa_{i\minus 1}}^{a_{i\minus 1}}\big)$ to $L_{act}^{z_i}$. If $(\pi_i,g_{\kappa_{i\minus 1}}^{a_{i\minus 1}})\neq (\pi_i,g_{\kappa_i}^{a_i})$, in the routine {\sc Generator}, the algorithm adds the relation $r^{z_i}\mapsto g^{a_{i\minus 1}}_{\kappa_{i\minus 1}}+g^{a_i}_{\kappa_i}$ if the edge $e^{z_i}_{(g^{a_{i\minus 1}}_{\kappa_{i\minus 1}},g^{a_i}_{\kappa_i})}$ does not close a cycle. If the edge $e^{z_i}_{(g^{a_{i\minus 1}}_{\kappa_{i\minus 1}},g^{a_i}_{\kappa_i})}$ closes a cycle, then there is a path from $g^{a_{i\minus 1}}_{\kappa_{i\minus 1}}$ to $g^{a_i}_{\kappa_i}$ in $\mathcal{G}_{act}^{z_i}$. The edges on this path correspond to relations $r^{w_1},\ldots,r^{w_q}\in R_\ell$ with $w_j\leq z_i$ such that $p_\ell^1(r^{w_1}+\cdots+r^{w_q})=g^{a_{i\minus 1}}_{\kappa_{i\minus 1}}+g^{a_i}_{\kappa_i}$. Therefore, in any case, there exists $r_i\in R_\ell$ such that $p_\ell^1(r_i)=g^{a_{i\minus 1}}_{\kappa_{i\minus 1}}+g^{a_i}_{\kappa_i}$. If $(\pi_i,g_{\kappa_{i\minus 1}}^{a_{i\minus 1}})=(\pi_i,g_{\kappa_i}^{a_i})$ we set $r_i=0$. 

We now define $(\tau,g^z_\rho)\coloneqq (\pi_m,g^{a_m}_{\kappa_m})$ and $r=r_1+\cdots+r_m$ and obtain
$p_\ell^1(r)=p_\ell^1(r_1+\cdots+r_m)=g^{a_0}_{\kappa_0}+g^{a_m}_{\kappa_m}=g^x_\sigma+g^{z}_\rho$.
\end{proof}

\begin{lemma}\label{lemma:algo_3}
Let $g^x_\sigma\in \mathcal{S}_\ell$ and $x\leq y\in P$ such that $K(x\leq y)(\sigma)\notin K_\ell(y)$. Then after the call of {\sc Collapse} in the algorithm, there exists $r\in R_\ell$ such that $p_\ell^1(r)=g^x_\sigma$.
\end{lemma}

\begin{proof}
Let $x=z_0\prec\cdots\prec z_m=y$ be a path from $x$ to $y$ in $P$. Moreover, let $\pi_i=K(z_0\leq z_i)(\sigma)$ for $0\leq i\leq m$. Since $K(x\leq y)(\sigma)=K(z_{m\minus 1}\prec z_m)\circ\cdots\circ K(z_0\prec z_1)(\sigma)\notin K_\ell(y)$, there exists a minimal $0<i\leq m$ such that $\pi_i\notin K_\ell(z_i)$. By Lemma \ref{lemma:algo_2}, there exists an active generator $(\pi_{i\minus 1},g^z_\rho)$ at $z_{i\minus 1}$ representing $\pi_{i \minus 1}$ and $r'\in R_\ell$ such that $g^x_\sigma+g^z_\rho=p_\ell^1(r')$. Since $K(z_{i\minus 1}\prec z_i)(\pi_{i\minus 1})\notin K_\ell(z_i)$ the algorithm adds the relation $r^{z_i}\mapsto g^z_\rho$ in the {\sc Collapse} routine if the edge $e_{(\Omega,g^z_\rho)}^{z_i}$ does not close a cycle in $\mathcal{G}_{act}^{z_i}$. If the edge $e_{(\Omega,g^z_\rho)}^{z_i}$ does close a cycle, then there is a path from $\Omega$ to $g^z_\rho$ in $\mathcal{G}_{act}^{z_i}$. This path corresponds to relations $r^{w_1},\ldots,r^{w_q}$ such that $p_\ell^1(r^{w_1}+\cdots+r^{w_q})=g^z_\rho$. Hence, in any case, there exists a relation $r''\in R_\ell$ such that $p_\ell^1(r'')=g^z_\rho$. Therefore, $p_\ell^1(r'+r'')=g^x_\sigma+g^z_\rho+g^z_\rho=g^x_\sigma$. 
\end{proof}

\begin{proposition} \label{prop:p_exact_2}
Let $G_\ell\xleftarrow{p_\ell^1}R_\ell$ be the matrix for the $\ell$-simplices constructed by the {\sc PiRep} algorithm. Then $\im(p_\ell^1)=\ker(\alpha_\ell)$. 
\end{proposition}

\begin{proof}
Let $x\in P$ and $g^{y_1}_{\sigma_1},\ldots,g^{y_m}_{\sigma_m}\in \mathcal{S}_\ell$ such that $y_i\leq x$ and $g^{y_i}_{\sigma_i}\neq g^{y_j}_{\sigma_j}$ for all $1\leq i<j \leq m$ and assume $\alpha_\ell(x)(g^{y_1}_{\sigma_1}+\cdots+g^{y_m}_{\sigma_m})=\sum_{i=1}^m K(y_i\leq x)(\sigma_i)=0$. Suppose $\tau\in K_\ell(x)$ appears in this sum and $K(y_i\leq x)(\sigma_i)=\tau$ for $1\leq i\leq q$ and $K(y_i\leq x)(\sigma_i)\neq\tau$ for $q+1\leq i\leq m$. This implies that $q$ is even. Let $(\tau,g^z_\rho)\in L_{act}^x$ be the active generator at $x$ representing $\tau$. By Lemma \ref{lemma:algo_2}, there are relations $r_i\in R_\ell$ such that $p_\ell^1(x)(r_i)=g^z_\rho+g^{y_i}_{\sigma_i}$ for $1\leq i\leq q$. Therefore, $p_\ell^1(x)(r_1+\cdots+r_q)=q\cdot g^z_\rho+g^{y_1}_{\sigma_1}+\cdots+g^{y_q}_{\sigma_q}=g^{y_1}_{\sigma_1}+\cdots+g^{y_q}_{\sigma_q}$. For every other simplex $\tau'\in K_\ell(x)$ appearing in this sum we can argue analogously. Now assume $K(y_i\leq x)(\sigma_i)=0$, i.e.\ $K(y_i\leq x)(\sigma_i)\notin K_\ell(x)$. By Lemma \ref{lemma:algo_3}, there exists $r'_i\in R_\ell$ such that $p_\ell^1(x)(r'_i)=g^{y_i}_{\sigma_i}$. We conclude that $g^{y_1}_{\sigma_1}+\cdots+g^{y_m}_{\sigma_m}\in\im\!\big( p_\ell^1(x)\big)$. Therefore, $\ker\!\big(\alpha_\ell(x)\big)\subseteq \im\!\big(p_\ell^1(x)\big)$ for all $x\in P$ 
which together with Proposition~\ref{prop:p_exact_1} implies that $\ker(\alpha_\ell)=\im(p_\ell^1)$. 
\end{proof}

Finally, we show that the matrices $f^0_\ell$ represent the boundary maps $\partial_\ell$ on the generators. 

\begin{proposition} \label{prop:f_correct}
The morphism $f^0_\ell\colon G_\ell\rightarrow G_{\ell\minus 1}$, constructed by the {\sc PiRep} algorithm, lifts the morphism $\partial_\ell\colon C_\ell\rightarrow C_{\ell\minus 1}$, i.e.\ the following square commutes:
\begin{equation}\label{eq:lift_diagram_algo}
\begin{tikzcd}
C_{\ell} \arrow[d,swap,"\partial_\ell"] & G_\ell \arrow[l,swap,"\alpha_\ell"] \arrow[d,"f^0_\ell"] \\
C_{\ell\minus 1} & G_{\ell\minus 1} \arrow[l,swap,"\alpha_{\ell\minus 1}"]
\end{tikzcd}.
\end{equation}
\end{proposition}

\begin{proof}
Let $x\in P$ and $g_\rho^y\in \mathcal{S}_\ell$ such that $y\leq x$. When $(\rho,g^y_\rho)$ is added to $L^y_{act}$ by {\sc Generator} at $y\in P$, the {\sc Boundary} routine of the algorithm finds active generators $(\pi_1,g^{y_1}_{\kappa_1}),\ldots,(\pi_k,g^{y_k}_{\kappa_k})$, which exist by Lemma \ref{lem:act_gen_simp}, such that $y_i\leq y$ and $\partial\rho=\sum_{i=1}^k\pi_i$ and defines $f^0_\ell(g^y_\rho)=\sum_{i=1}^kg^{y_i}_{\kappa_i}$. By Lemma \ref{lemma:algo_1}, $K(y_i\leq y)(\kappa_i)=\pi_i$ for all $1\leq i\leq k$. First, assume $K(y\leq x)(\rho)=\sigma\in K_\ell(x)$. Then 
\begin{equation*}
\partial_\ell(x)\circ \alpha_\ell(x)(g_\rho^y)=\partial_\ell(x)\big(K(y\leq x)(\rho)\big)=\partial_\ell(x)(\sigma)   
\end{equation*}
and
\begin{equation*}
\begin{aligned}
\alpha_{\ell\minus 1}(x)\circ f^0_\ell(x)(g^y_\rho)&=\alpha_{\ell\minus 1}(x)(\sum_{i=1}^kg^{y_i}_{\kappa_i})=\sum_{i=1}^k \alpha_{\ell\minus 1}(x)(g^{y_i}_{\kappa_i})=\sum_{i=1}^k K(y_i\leq x)(\kappa_i)  \\ &=\sum_{i=1}^k K(y\leq x)\circ K(y_i\leq y)(\kappa_i)=\sum_{i=1}^k K(y\leq x)(\pi_i) \\ & =\sum_{i=1}^k C_{\ell\minus 1}(y\leq x)(\pi_i)=C_{\ell\minus 1}(y\leq x)(\sum_{i=1}^k \pi_i)\\ &=C_{\ell\minus 1}(y\leq x)\circ\partial_\ell(y)(\rho)=\partial_\ell(x)\circ C_{\ell}(y\leq x)(\rho) \\ &=\partial_\ell(x)\big(K(y\leq x)(\rho)\big)=\partial_\ell(x)(\sigma)
\end{aligned}
\end{equation*}
where we use that
\begin{equation*}
\begin{tikzcd}[column sep=large]
C_\ell(y) \arrow[r,"C_\ell(y\leq x)"] \arrow[d,swap,"\partial_\ell(y)"] & C_\ell(x) \arrow[d,"\partial_\ell(x)"] \\
C_{\ell\minus 1}(y) \arrow[r,"C_{\ell\minus 1}(y\leq x)"] & C_{\ell\minus 1}(x)
\end{tikzcd}
\end{equation*}
commutes. Now assume $K(y\leq x)(\rho)\notin K_\ell(x)$. Then $\partial_\ell(x)\circ \alpha_\ell(x)(g^y_\rho)=0$ and
\begin{equation*}
\begin{aligned}
\alpha_{\ell\minus 1}(x)\circ f^0_\ell(x)(g^y_\rho)=\partial_\ell(x)\big(K(y\leq x)(\rho)\big)=\partial_\ell(x)\big(0\big)=0.
\end{aligned}
\end{equation*}
Therefore, \eqref{eq:lift_diagram_algo} commutes and $f^0_\ell$ is a lift of $\partial_\ell$ to the generators $G_\ell$ and $G_{\ell\minus 1}$.
\end{proof}

\begin{theorem} \label{thm:correct_presentation}
The matrices $p^1_\ell$ and $f^0_\ell$ constructed by the algorithm {\sc PiRep}: 
\begin{equation*}
\begin{tikzcd}
& \vdots \arrow[d] & \vdots \arrow[d] &  \\
0 & C_2 \arrow[d,swap,"\partial_2"] \arrow[l] & G_2 \arrow[l,swap,"\alpha_2"] \arrow[d,"f^0_2"] & R_2 \arrow[l,swap,"p_2^1"] \\
0 & C_1 \arrow[l]  \arrow[d,swap,"\partial_1"] & G_1 \arrow[l,swap,"\alpha_1"] \arrow[d,"f^0_1"] & R_1 \arrow[l,swap,"p_1^1"] \\
0 & C_0 \arrow[l] & G_0 \arrow[l,swap,"\alpha_0"] & R_0 \arrow[l,swap,"p_0^1"]
\end{tikzcd}
\end{equation*}
form degreewise presentations of $C_\bullet(K)$ and lifts of $\partial_\bullet$ to the generators.
\end{theorem}

\begin{proof}
By Proposition \ref{prop:alpha_well_defined} and \ref{prop:alpha_epi}, $\alpha_\ell$ is a well-defined epimorphism. By Proposition \ref{prop:p_exact_2}, $\ker(\alpha_\ell)=\im(p_\ell^1)$. Hence, $p_\ell^1$ is a presentation of $C_\ell$. By Proposition \ref{prop:f_correct}, $f^0_\ell$ is a lift of $\partial_\ell$ to the generators.    
\end{proof}

\subsection{Minimality of the presentation $p_\ell^1$}
\label{sec:p1minimality}
In this section, we show that the matrix $p_\ell^1$, as constructed by the algorithm {\sc PiRep}, is a minimal presentation of $C_\ell$. We use the notion of projective cover in Definition \ref{def:projective_cover} together with Proposition \ref{prop:minimal_eqq_projective_cover}.

\begin{proposition} \label{prop:alpha_proj_cover}
The morphism $\alpha_\ell\colon G_\ell\rightarrow C_\ell$ is a projective cover.    
\end{proposition}

\begin{proof}
Let $\mathrm{Rad}(G_\ell)$ denote the radical of $G_\ell$ as defined in Definition \ref{def:radical}. Suppose $\alpha_\ell(x)(g^{y_1}_{\sigma_1}+\cdots+g^{y_m}_{\sigma_m})=0$ for $g^{y_i}_{\sigma_i}\in \mathcal{S}_\ell$ such that $y_i\leq x$ and $g^{y_i}_{\sigma_i}\neq g^{y_j}_{\sigma_j}$ for $i\neq j$. If $y_j=x$, then $\sigma_j\in K_\ell(x)$ and $\sigma_j\notin \bigcup_{z\prec x}\im \!\big(K(z\prec x)\big)$. For all $i\neq j$, if $y_i=x$, then $\sigma_j\neq\sigma_i$. For all $y_i<x$, $K(y_i\leq x)(\sigma_i)\neq \sigma_j$ since any path from $y_i$ to $x$ contains some predecessor $z\prec x$ which would imply $K(z\prec x)\circ K(y_i\leq z)(\sigma_i)=\sigma_j$, contradicting $\sigma_j\notin \bigcup_{z\prec x}\im \!\big(K(z\prec x)\big)$. Therefore,
\begin{equation*}
\alpha_\ell(x)(g^{y_1}_{\sigma_1}+\cdots+g^{y_m}_{\sigma_m})=\sum_{i=1}^m K(y_i\leq x)(\sigma_i)=\sigma_j+\sum_{i\neq j}K(y_i\leq x)(\sigma_i).
\end{equation*}
But this sum cannot be zero since $K(y_i\leq x)(\sigma_i)\neq \sigma_j$ for all $y_i<x$ and $\sigma_i\neq\sigma_j$ for all $y_i=x$ with $i\neq j$. This contradiction implies $y_i<x$ for all $1\leq i\leq m$. Therefore, $g^{y_i}_{\sigma_i}\in \im \!\big(G_\ell(z\prec x)\big)$ for some $z\prec x$ and $g^{y_1}_{\sigma_1}+\cdots+g^{y_m}_{\sigma_m}\in \mathrm{Rad}(G_\ell)(x)$. We conclude that $\ker(\alpha_\ell)\subseteq \mathrm{Rad}(G_\ell)$, which implies that $\alpha_\ell$ is a projective cover of $C_\ell$.    
\end{proof}

Next, we show that the presentation matrix $p_\ell^1$, constructed by the {\sc PiRep} algorithm, almost has the structure of a $\overline{P}$-filtered graph. We will use this structure to show that $p_\ell^1$ is a minimal presentation.

\begin{definition} \label{def:boundary_p}
Suppose $p^1_\ell$ is the presentation matrix generated by the algorithm. Recall that $\overline{P}\coloneqq P\cup\{-\infty\}$ with $-\infty<x$ for all $x\in P$. Define the morphism of $\overline{P}$-persistence modules $\overline{p}^1_\ell\colon R_\ell\rightarrow \overline{G}_\ell$, where $\overline{G}_\ell=G_\ell\oplus \langle\Omega\rangle$ with an additional generator $\Omega$ of grade $-\infty$. The matrix $\overline{p}_\ell^1$ is obtained from $p_\ell^1$ by adding one row labeled by $\Omega$ and replacing each column of the form $r^y\mapsto g^x_\sigma$ by the column $r^y\mapsto g^x_\sigma+\Omega$. Since every column of $\overline{p}_\ell^1$ has exactly two non-zero entries, it is the boundary matrix of a $\overline{P}$-filtered graph $\mathcal{G}\colon \overline{P}\rightarrow\Delta\mathbf{Cpx}_{\leq 1}$. The graph $\mathcal{G}$ is defined by the vertices $\Omega$ born at $-\infty$ and $g^x_\sigma\in G_\ell$ born at $x\in P$ and by the edges $r^y\in R_\ell$ born at $y\in P$. All morphisms $\mathcal{G}(x< y)$ are inclusions.
\end{definition}

\begin{remark}
The edges of the graph $\mathcal{G}(x)$ are exactly the edges added by the algorithm to the graphs $\mathcal{G}_{act}^{y}$ with $y\leq x$. The relations $r^y\mapsto g^w_\sigma+g^z_\tau$ correspond to the edges $e^{y}_{(g^w_\sigma,g^z_\tau)}$ added in step 2(a) of {\sc Generator}. The relations $r^y\mapsto g^w_\sigma$ that become $r^y\mapsto g^w_\sigma+\Omega$ in $\overline{p}_\ell^1 $ correspond to the edges $e^y_{(g^w_\sigma,\Omega)}$ added in step 3(b) of {\sc Collapse}. Therefore, $\mathcal{G}_{act}^x$ is a subgraph of $\mathcal{G}(x)$ for every $x\in P$.
\end{remark}

\begin{lemma}\label{prop:spanning_forest}
At any grade $x_i$ in the algorithm, after execution of {\sc Generator}, $\mathcal{F}_{act}^{x_i}$ is a spanning forest of $\mathcal{G}(x_i)$.
\end{lemma}

\begin{proof}
By definition, it is clear that $\mathcal{F}_{act}^{x_i}$ is always a subgraph of $\mathcal{G}(x_i)$. We proceed by induction on the linear extension of the grades in $P$. At a minimal grade $x\in P$, the statement is obviously true as there are no edges in $\mathcal{F}_{act}^{x}$ or $\mathcal{G}(x)$. Assume that the statement is true for all $0\leq j<i$. At the grade $x_i$, we have that $\mathcal{G}(x_i)$ is the union of all $\mathcal{G}(y)$ with $y\prec x_i$ plus the generators $g^{x_i}_\sigma$ and relations $r^{x_i}$ born at $x_i$. Define $\mathcal{G}'\coloneqq\bigcup_{y\prec x_i}\mathcal{G}(y)$. By inductive assumption $\mathcal{F}_{act}^y$ is a spanning forest of $\mathcal{G}(y)$. Hence, after initialization, $\mathcal{G}_{act}^{x_i}=\bigcup_{y\prec x_i}\mathcal{F}_{act}^y$ contains a spanning forest of $\mathcal{G}'$. In {\sc Collapse}, the algorithm adds the edge $e^{x_i}_{(\Omega,g^z_\rho)}$ to $\mathcal{G}_{act}^{x_i}$ exactly when it adds the relation $r^{x_i}\mapsto g^z_\rho$. Similarly, in {\sc Generator}, the algorithm adds the edge $e^{x_i}_{(g^y_\tau,g^z_\rho)}$ exactly when it adds the relation $r^{x_i}\mapsto g^y_\tau+g^z_\rho$. A relation is only added to $\mathcal{G}_{act}^{x_i}$ if it does not close a cycle. This implies that, at the grade $x_i$ where the edge is added, it is in the spanning forest $\mathcal{F}_{act}^{x_i}$. Hence, every edge added at $x_i$ is added to $\mathcal{G}'$ and also to $\mathcal{F}_{act}^{x_i}$. Before adding edges at $x_i$, a spanning forest $\mathcal{F}_{act}^{x_i}$ of $\mathcal{G}_{act}^{x_i}$ is also a spanning forest of $\mathcal{G'}$. Since all edges at $x_i$ are added to both $\mathcal{G}'$ and $\mathcal{F}_{act}^{x_i}$, this property is preserved. The remaining vertices $g_\sigma^{x_i}$ of $\mathcal{G}(x_i)$ are added to $\mathcal{G}_{act}^{x_i}$ by {\sc Generator}. Since the existence of $g_\sigma^{x_i}\in\mathcal{S}$ implies that $\sigma\notin\bigcup_{y\prec x_i}\im \!\big(K(y\prec x_i)\big)$, $g_\sigma^{x_i}$ is an isolated vertex in $\mathcal{G}_{act}^{x_i}$ and $\mathcal{G}(x_i)$ and hence also in $\mathcal{F}_{act}^{x_i}$.  Therefore, the spanning forest $\mathcal{F}_{act}^{x_i}$ of $\mathcal{G}_{act}^{x_i}$, constructed in {\sc Generator}, is also a spanning forest of $\mathcal{G}(x_i)$.  
\end{proof}

\begin{remark}
Since $\overline{p}^1_\ell(x)$ is the boundary matrix of $\mathcal{G}(x)$, $\coker\!\big(\overline{p}^1_\ell(x)\big)$ is the space of connected components and $\ker\!\big(\overline{p}^1_\ell(x)\big)$ is the space of cycles in $\mathcal{G}(x)$.
\end{remark}

\begin{lemma} \label{lem:graph_boundary_kernel}
For all $x\in P$, we have $\ker\!\big(\overline{p}^1_\ell(x)\big)= \ker\!\big(p^1_\ell(x)\big)$ and $\coker\!\big(\overline{p}^1_\ell(x)\big)\cong \coker\!\big( p^1_\ell(x)\big)\oplus [\Omega]$ where $[\Omega]$ denotes the connected component of $\Omega$.
\end{lemma}

\begin{proof}
The matrix $p^1_\ell(x)$ is the restriction of $p^1_\ell$ to all generators $g^y_\sigma$ and relations $r^z$ such that $y,z\leq x$. Suppose $p^1_\ell(x)(r^{y_1}+\cdots+r^{y_k})=(g^{a_1}_{\sigma_1}+g^{b_1}_{\tau_1})+\cdots+(g^{a_p}_{\sigma_p}+g^{b_p}_{\tau_p})+g^{c_1}_{\rho_1}+\cdots+g^{c_q}_{\rho_q}=0$. Because all generators $g^{*}_{*}$ are basis elements of $G_\ell(x)$ and a sum of basis elements over $\mathbb{F}_2$ can only be zero if every basis element appears an even number of times in the sum, the number $q$ of single-generator relations has to be even. Also, this implies $\overline{p}^1_\ell(x)(r^{y_1}+\cdots+r^{y_k})=(g^{a_1}_{\sigma_1}+g^{b_1}_{\tau_1})+\cdots+(g^{a_p}_{\sigma_p}+g^{b_p}_{\tau_p})+g^{c_1}_{\rho_1}+\cdots+g^{c_q}_{\rho_q}+q\Omega=0$. Conversely, if $\overline{p}^1_\ell(x)(r^{y_1}+\cdots+r^{y_k})=(g^{a_1}_{\sigma_1}+g^{b_1}_{\tau_1})+\cdots+(g^{a_p}_{\sigma_p}+g^{b_p}_{\tau_p})+(g^{c_1}_{\rho_1}+\Omega)+\cdots+(g^{c_q}_{\rho_q}+\Omega)=(g^{a_1}_{\sigma_1}+g^{b_1}_{\tau_1})+\cdots+(g^{a_p}_{\sigma_p}+g^{b_p}_{\tau_p})+g^{c_1}_{\rho_1}+\cdots+g^{c_q}_{\rho_q}+q\Omega=0$. We get that $q$ is even and therefore also $p^1_\ell(x)(r^{y_1}+\cdots+r^{y_k})=(g^{a_1}_{\sigma_1}+g^{b_1}_{\tau_1})+\cdots+(g^{a_p}_{\sigma_p}+g^{b_p}_{\tau_p})+g^{c_1}_{\rho_1}+\cdots+g^{c_q}_{\rho_q}=0$. Hence, $\ker\!\big( p^1_\ell(x)\big)=\ker\!\big(\overline{p}^1_\ell(x)\big)$ for all $x\in P$. 

Since $\overline{G}_\ell(x)=G_\ell(x)\oplus \Omega$ we can define a projection map $\pi\colon \overline{G}_\ell(x)\rightarrow G_\ell(x)$ that forgets $\Omega$. We obtain
\begin{equation*}
\begin{tikzcd}
\coker\!\big(\overline{p}^1_\ell(x)\big) \arrow[d,swap,"\coker \pi"] & \overline{G}_\ell(x) \arrow[l,swap,"\phi"] \arrow[d,"\pi"] & R_\ell(x) \arrow[l,swap,"\overline{p}(x)"] \arrow[d,"="] \\
\coker\!\big(p^1_\ell(x)\big) & G_\ell(x) \arrow[l,swap,"\psi"] & R_\ell(x) \arrow[l,swap,"p^1_\ell(x)"]
\end{tikzcd}
\end{equation*}
Since $\overline{p}^1_\ell(x)$ is the boundary matrix of a graph, the cokernel has a basis consisting of the connected components of this graph. If $[\Omega]$ denotes the equivalence class of $\Omega$ in $\coker \!\big(\overline{p}^1_\ell(x)\big)$, then  $(\coker \pi)([\Omega])=[\pi(\Omega)]=0$. Since $\psi\circ \pi=(\coker \pi)\circ \phi$ is an epimorphism, also $\coker \pi$ is an epimorphism and $\coker\!\big(\overline{p}^1_\ell(x)\big)\cong \im(\coker \pi)\oplus \ker(\coker \pi)\cong \coker\!\big(p^1_\ell(x)\big) \oplus \ker(\coker \pi)$. Since $\dim\ker\!\big(p^1_\ell(x)\big)=\dim\ker\!\big(\overline{p}^1_\ell(x)\big)$, we also get $\dim \im\!\big( p^1_\ell(x)\big)=\dim \im \!\big(\overline{p}^1_\ell(x)\big)$ and, moreover, $\dim \coker \!\big(p^1_\ell(x)\big)=\dim G_\ell(x)-\dim \im \!\big(p^1_\ell(x)\big)=\dim \overline{G}_\ell(x)-1-\dim \im \!\big(\overline{p}^1_\ell(x)\big)=\dim \coker\!\big( \overline{p}^1_\ell(x)\big)-1$. Therefore, $\ker(\coker \pi)$ is one-dimensional, generated by $[\Omega]$, and $\coker\!\big( \overline{p}^1_\ell(x)\big)\cong \coker\!\big( p^1_\ell(x)\big) \oplus [\Omega]$.
\end{proof}

\begin{proposition} \label{prop:p1_proj_cover}
The morphism $p_\ell^1\colon R_\ell\rightarrow \ker(\alpha_\ell)$ is a projective cover.
\end{proposition}

\begin{proof}
By construction $R_\ell$ is projective. By Proposition \ref{prop:p_exact_2}, $\ker(\alpha_\ell)=\im(p_\ell^1)$. Thus, $p_\ell^1\colon R_\ell\rightarrow \ker(\alpha_\ell)$ is an epimorphism. 

Let $x\in P$ and $r^{y_1},\ldots, r^{y_k}\in R_\ell$ such that $y_i\leq x$ and $p_\ell^1(x)(r^{y_1}+\cdots+r^{y_k})=0$. W.l.o.g.\ assume that the sum $r^{y_1}+\cdots+r^{y_k}$ is reduced, i.e., even numbers of copies of the same relation (same formal basis element of $R_\ell$) are canceled. Note that different formal basis elements $r^y$ of $R_\ell$ can have the same grade $y$. By Lemma \ref{lem:graph_boundary_kernel}, we obtain that $\overline{p}_\ell^1(x)(r^{y_1}+\cdots+r^{y_k}) =0$
and, thus, $e^{y_1},\ldots,e^{y_k}$ is a cycle in the graph $\mathcal{G}(x)$. For any edge $e^{y_i}$ such that $y_i=x$ the corresponding relation $r^x$ is added at the grade $x$ during the algorithm. Since this only happens when the edge $e^x$ is added to $\mathcal{G}_{act}^x$, any edge $e^x$ on the cycle is in $\mathcal{G}_{act}^x$. If $y_i<x$ then there exists a $z\prec x$ such that $e^{y_i}\in\mathcal{G}(z)$. By Lemma \ref{prop:spanning_forest}, $\mathcal{F}_{act}^z$ is a spanning forest of $\mathcal{G}(z)$. If $e^{y_i}\notin\mathcal{F}_{act}^z$, then there exists a path $\gamma_1^{i},\ldots,\gamma^{i}_{l_i}$ in $\mathcal{F}_{act}^z$ connecting the endpoints of $e^{y_i}$. So, by construction of $\mathcal{G}_{act}^x$ in {\sc Collapse}, in any case, there exists a path in $\mathcal{G}_{act}^x$ connecting the endpoints of the edge $e^{y_i}$. Hence, we can replace any edge on the cycle $e^{y_1},\ldots,e^{y_k}$ that is not in $\mathcal{G}_{act}^x$ by a path in $\mathcal{G}_{act}^x$. This yields another cycle that lies completely in $\mathcal{G}_{act}^x$. If there are edges $e^x$ on the initial cycle, they will not be replaced and are also on the cycle in $\mathcal{G}_{act}^x$. But this implies that there is an edge $e^x$, added at the grade $x$, closing the cycle. This is a contradiction, since the algorithm does not add edges that close cycles in $\mathcal{G}_{act}^x$. Therefore, $y_i<x$ for all $1\leq i\leq k$. Recalling Definition \ref{def:radical}, we get that $r^{y_1}+\cdots+r^{y_k}\in \mathrm{Rad}(R_\ell)(x)$ and, moreover, $\ker\!\big(p_\ell^1(x)\big)\subseteq \mathrm{Rad}(R_\ell)(x)$. This implies that $\ker(p_\ell^1)\subseteq \mathrm{Rad}(R_\ell)$ and, thus, $p^1_\ell$ is a projective cover of $\ker(\alpha_\ell)$.
\end{proof}

\begin{theorem} \label{thm:correct_minimal_presentation}
The matrix $p_\ell^1$, computed by the algorithm {\sc PiRep}, is a minimal presentation of $C_\ell$.
\label{thm:p1min}
\end{theorem}

\begin{proof}
By Theorem \ref{thm:correct_presentation}, $p_\ell^1$ is a presentation. By Proposition \ref{prop:alpha_proj_cover} and \ref{prop:p1_proj_cover}, $\alpha_\ell\colon G_\ell\rightarrow C_\ell$ is a projective cover of $C_\ell$ and $p_\ell^1\colon R_\ell\rightarrow \ker(\alpha_\ell)$ is a projective cover of $\ker(\alpha_\ell)$. Therefore, by Proposition \ref{prop:minimal_eqq_projective_cover}, $p_\ell^1$ is a minimal presentation.
\end{proof}

\subsection{Correctness of computing $p_\ell^2$ and minimality}
\label{sec:p2correctness}

In this section, we show that the matrix $p_\ell^2$, as constructed by the {\sc PiRep} algorithm, yields an asymptotically minimal partial projective resolution of $C_\ell$ up to the second term. With the additional call of {\sc Reduce}, we will show that it is a minimal partial projective resolution of $C_\ell$ up to the second term.

\begin{proposition} \label{prop:exact_at2_1}
    $\im(p_\ell^2) \subseteq \ker(p_\ell^1)$.
\end{proposition}

\begin{proof}
Suppose $rr^{x}$ is a relation of relations in $RR_\ell$ such that $p^2_\ell(rr^{x})= r^{z_0}+\cdots+r^{z_s}$. The relation of relations is added by the algorithm if and only if the corresponding edges $e^{z_0},\ldots,e^{z_s}$ form a cycle in the graph $\mathcal{G}_{act}^{x}$ at the grade $x$. Since $\mathcal{G}_{act}^{x}$ is a subgraph of $\mathcal{G}(x)$, the edges $e^{z_0},\ldots,e^{z_s}$ also form a cycle in $\mathcal{G}(x)$. By Definition \ref{def:boundary_p}, we get that $\overline{p}_\ell^1(y)(r^{z_0}+\cdots+r^{z_s})=0$ and, by Lemma \ref{lem:graph_boundary_kernel}, also $p_\ell^1(y)(r^{z_0}+\cdots+r^{z_s})=0$ for all $y\geq x$. Hence, $p^1_\ell\circ p_\ell^2(rr^{x})=0$ and $\im(p_\ell^2)\subseteq \ker(p_\ell^1)$.
\end{proof}

\begin{proposition} \label{prop:exact_at2_2}
     $\ker(p_\ell^1) \subseteq \im(p_\ell^2)$.
\end{proposition}

\begin{proof}
We proceed by induction on the grades in the linear extension of $P$. Assume that $\ker\!\big(p_\ell^1(z)\big)\subseteq \im\!\big(p_\ell^2(z)\big)$ for all $z<x$. This is obviously true for the minimal elements in $P$.
Let $r^{y_1},\ldots,r^{y_n}\in R_\ell$ so that $y_i\leq x$ for all $1\leq i\leq n$ and $p^1_\ell(x)(r^{y_1}+\cdots+r^{y_n})=0$. By Lemma \ref{lem:graph_boundary_kernel}, $e^{y_1},\ldots,e^{y_n}$ form a cycle in $\mathcal{G}(x)$. By the proof of Proposition \ref{prop:p1_proj_cover}, $\ker\!\big(p_\ell^1(x)\big)\subseteq \mathrm{Rad}(R_\ell)(x)$. Thus, for each $1\leq i\leq n$, we have $y_i<x$ and $e^{y_i}\in \mathcal{G}(y)$ for some $y\prec x$.  Suppose $e^{y_i}\notin \mathcal{G}_{act}^x$. Then $e^{y_i}\notin \mathcal{F}_{act}^y$. Since, by Lemma \ref{prop:spanning_forest}, $\mathcal{F}_{act}^y$ is a spanning forest of $\mathcal{G}(y)$, there exists a path $e^{z^i_1},\ldots,e^{z^i_{k_i}}$ in $\mathcal{F}_{act}^y$ that connects the endpoints of $e^{y_i}$. Hence, this path together with $e^{y_i}$ forms a cycle in $\mathcal{G}(y)$. By Lemma \ref{lem:graph_boundary_kernel}, $p_\ell^1(y)(r^{y_i}+r^{z^i_1}+\cdots+r^{z^i_{k_i}})=0$. Thus, by induction, there exists $rr_i\in RR_\ell(y)$ such that $p^2_\ell(y)(rr_i)=r^{y_i}+r^{z^i_1}+\cdots+r^{z^i_{k_i}}$. We now replace the edge $e^{y_i}$ in the cycle $e^{y_1},\ldots,e^{y_n}$ by the path $e^{z^i_1},\ldots,e^{z^i_{k_i}}$ which lies in $\mathcal{G}_{act}^{x}$. If we replace each edge $e^{y_i}\notin \mathcal{G}_{act}^x$ by such a path, we obtain a cycle $e^{z_1},\ldots,e^{z_l}$ in $\mathcal{G}_{act}^x$ since 
\begin{equation*}
\begin{aligned}
p_\ell^1(x)\big(r^{z_1}+\cdots+r^{z_l}\big)=&p_\ell^1(x)\big(\sum_{r^{y_i}\in\mathcal{G}_{act}^x}r^{y_i}+\sum_{r^{y_i}\notin \mathcal{G}_{act}^x}r^{z^i_1}+\cdots+r^{z^i_{k_i}}\big)\\&=p_\ell^1(x)\big(\sum_{r^{y_i}\in\mathcal{G}_{act}^x}r^{y_i}\big)+\sum_{r^{y_i}\notin \mathcal{G}_{act}^x}p_\ell^1(x)\big(r^{z^i_1}+\cdots+r^{z^i_{k_i}}\big)\\&=p_\ell^1(x)\big(\sum_{r^{y_i}\in\mathcal{G}_{act}^x}r^{y_i}\big)+p_\ell^1(x)\big(\sum_{r^{y_i}\notin\mathcal{G}_{act}^x}r^{y_i}\big)\\&=p_\ell^1(x)(r^{y_1}+\cdots+r^{y_n})=0 .
\end{aligned}
\end{equation*}
The algorithm {\sc RelRel} computes a cycle basis for $\mathcal{G}_{act}^x$ and adds the computed cycles to $p_\ell^2$. Hence, there exists $rr\in RR_\ell(x)$ such that $p^2_\ell(x)(rr)=r^{z_1}+\cdots+r^{z_l}$. Combining this cycle with the cycles for the replaced edges yields
\begin{equation*}
\begin{aligned}
p_\ell^2(x)(rr+\sum_{r^{y_i}\notin \mathcal{G}_{act}^x}rr_i)&=p_\ell^2(x)(rr)+\sum_{r^{y_i}\notin \mathcal{G}_{act}^x}p_\ell^2(x)(rr_i)\\&=r^{z_1}+\cdots+r^{z_l}+\sum_{r^{y_i}\notin \mathcal{G}_{act}^x}r^{y_i}+r^{z^i_1}+\cdots+r^{z^i_{k_i}}\\&=\sum_{r^{y_i}\in\mathcal{G}_{act}^x}r^{y_i}+\sum_{r^{y_i}\notin \mathcal{G}_{act}^x}r^{z^i_1}+\cdots+r^{z^i_{k_i}}\\&+\sum_{r^{y_i}\notin \mathcal{G}_{act}^x}r^{y_i}+\sum_{r^{y_i}\notin \mathcal{G}_{act}^x}r^{z^i_1}+\cdots+r^{z^i_{k_i}}\\&=r^{y_1}+\cdots+r^{y_n} .
\end{aligned}
\end{equation*}
Therefore, $r^{y_1}+\cdots+r^{y_n}\in\im \!\big(p_\ell^2(x))$ and, moreover, $\ker(p_\ell^1)\subseteq \im(p_\ell^2)$.
\end{proof}

\begin{proposition} \label{prop:asymptotic_optimality}
There exists a family of poset towers $K^m\colon P^m\rightarrow \mathbf{SCpx}$, $m\geq 2$, with $n=\Theta(m)$ simplex generators and $t_1=\Theta(m^2)$ edges in the Hasse diagram of $P^m$, such that the second term $RR_0$ in a minimal projective resolution of $C_0(K^m)$ has $\Theta(nt_1)$ elementary projective summands. In particular, the bound $O(nt_1)$ on the number of columns of $p_\ell^2$ computed by {\sc PiRep} without {\sc Reduce} is asymptotically optimal.
\end{proposition}

\begin{proof}
Let $A=\{a_1,\ldots,a_m\}$, $B=\{b_1,\ldots,b_m\}$, $C=\{c_1,\ldots,c_m\}$ and let $P^m=A\cup B\cup C$ be the poset with relations $a_i\prec b_j\prec c_k$ for all $1\leq i,j,k\leq m$. Thus $P^m$ has $t_0=3m$ grades and $t_1=2m^2$ edges in its Hasse diagram. We define a poset tower $K^m\colon P^m\rightarrow \mathbf{SCpx}$ consisting of a single vertex $v$ at every grade and identity maps on this vertex along every edge. Equivalently, $K^m(x)=\{v\}$ for all $x\in P^m$, and $K^m(x\leq y)(v)=v$ whenever $x\leq y$. Then $v$ has exactly one simplex generator $g_v^{a_i}$ at each grade $a_i\in A$ and no simplex generators at the grades in $B\cup C$. Hence $\mathcal{S}_0=\{g_v^{a_1},\ldots,g_v^{a_m}\}$, $\mathcal{C}=\emptyset$ and the input size is $n=|\mathcal{S}|+|\mathcal{C}|=m$.

At a grade $b_j$, the vector space $G_0(b_j)$ has basis $g_v^{a_1},\ldots,g_v^{a_m}$ whereas $C_0(K_m)(b_j)$ is one-dimensional. The map $\alpha_0(b_j)\colon G_0(b_j)\rightarrow C_0(K_m)(b_j)$ sends every basis element
$g_v^{a_i}$ to $v$. Therefore, $\ker\!\big(\alpha_0(b_j)\big)$ has dimension $m-1$. Since $\ker\!\big(\alpha_0(a_i)\big)=0$ for all $1\leq i\leq m$, there must be $m-1$ relations born at $b_j$. Since we proved in Theorem \ref{thm:p1min} that {\sc PiRep} computes a minimal presentation this can also be seen by applying the algorithm, which has to identify the $n$ copies $(v,g_v^{a_i})$ at each grade $b_j$. The algorithm adds edges until we obtain a spanning tree $T_j$ on the vertices $g_v^{a_1},\ldots,g_v^{a_m}$, which requires exactly $n-1$ edges.

Now consider a grade $c_k$. The relation graph at $c_k$ contains the union of the $m$ spanning trees $T_1\cup \cdots \cup T_m$ on the same set of $m$ vertices. The edges are indexed by the relations born at the different grades $b_j$, so they are distinct columns of $p_0^1$, even if some of them have the same endpoints. The graph $T_1\cup\cdots\cup T_m$ is connected and has $m$ vertices and $m(m-1)$ edges. Therefore, its cycle space has dimension $m(m-1)-m+1=(m-1)^2$. By Lemma~\ref{lem:graph_boundary_kernel}, this cycle space is equal to
$\ker\!\big(p_0^1(c_k)\big)$. Hence $\dim \ker\!\big(p_0^1(c_k)\big)=(m-1)^2$.

Moreover, for every $x<c_k$ we have $\ker\!\big(p_0^1(x)\big)=0$: this is clear at the grades $a_i$, where no relations are present, and at the grades $b_j$, where the relation graph is a tree. It follows that any projective cover of $\ker(p_0^1)$ must contain at least $(m-1)^2$ elementary projective summands born at $c_k$. Since the grades $c_1,\ldots,c_m$ are pairwise incomparable, summands born at one of them cannot contribute to any other one. Thus, the second term $RR_0$ in a minimal projective resolution has $m(m-1)^2$ elementary projective summands.

Finally, since $n=m$ and $t_1=2m^2$, we obtain $m(m-1)^2=\Theta(m^3)=\Theta(nt_1)$. This proves that the $O(nt_1)$ bound on the number of columns of $p_\ell^2$ is asymptotically optimal.
\end{proof}

\begin{remark}
In the following Proposition \ref{prop:proj_cover_p2}, we show that, with the additional call of the routine {\sc Reduce}, the algorithm computes a minimal resolution up to the second term. In other words, the algorithm computes a projective cover $p_\ell^2$ of $\ker(p_\ell^1)$. This step can also be interpreted geometrically in terms of the $\overline{P}$-filtered graph $\mathcal{G}$ represented by $\overline{p}_\ell^1$. The kernel of $\overline{p}_\ell^1$ is the cycle space of $\mathcal{G}$ and equal to the kernel of $p_\ell^1$ by Lemma \ref{lem:graph_boundary_kernel}. Thus, the task of computing a projective cover of $\ker(p_\ell^1)\colon P\rightarrow \mathbf{Vec}$ is equivalent to computing a minimal set of generators of $H_1(\mathcal{G})\colon P\rightarrow \mathbf{Vec}$. Each column of $p_\ell^2$, added by the algorithm, can be viewed as a $2$-cell filling a cycle in $\mathcal{G}$. The set of $2$-cells generated in this way is minimal if a $2$-cell added at grade $x$ never closes a void, i.e., never creates a class in $H_2$. This is exactly a reformulation of the projective cover condition. The kernel of $p_\ell^2$ corresponds to $H_2$ classes or voids, and if we close a void at grade $x$ by adding a $2$-cell, then this $2$-cell is a linear combination of already existing $2$-cells and thus redundant. Reducing the matrix $A^{\leq x}$ in the {\sc Reduce} routine and checking if columns added at grade $x$ can be reduced to zero is exactly checking if the $2$-cell closes a void. Because this is essentially a two-dimensional problem, it cannot be solved as efficiently as the other tasks. It is thus remarkable that we can still compute an asymptotically optimal $p_\ell^2$ without checking this condition.
\end{remark}

\begin{proposition}\label{prop:proj_cover_p2}
If {\sc Reduce} is called, $p_\ell^2$ is a projective cover of $\ker(p_\ell^1)$.  
\end{proposition}

\begin{proof}
Let $x\in P$ and $rr^{y_1},\ldots,rr^{y_m}\in RR_\ell$ such that $y_i\leq x$ and $p^2_\ell(x)(rr^{y_1}+\cdots+rr^{y_m})=0$. At the grade $x$ in the algorithm, at the start of {\sc Reduce}, all $rr^{y_i}$ are already in $RR_\ell$. The matrix $A^{\leq x}$ representing $p_\ell^2(x)$ consists of all columns $rr^{z}$ of $p_\ell^2$, such that $z\leq x$, in an order respecting the linear extension of $P$. Thus, the columns $rr^{z}$ such that $z=x$ come last in $A^{\leq x}$. At the end of {\sc Reduce}, the matrix $A^{\leq x}$ is column reduced from left to right (with respect to the order of the columns in the matrix). 

We now prove by contradiction that $y_i<x$ for all $i$. Suppose $rr^{y_i}$ is the rightmost column among $rr^{y_1},\ldots,rr^{y_m}$ in $A^{\leq x}$ such that $y_i=x$. The assumption $p^2_\ell(x)(rr^{y_1}+\cdots+rr^{y_m})=0$ implies that the column $rr^{y_i}$ is a linear combination of the columns $\{rr^{y_1},\ldots,rr^{y_m}\}\setminus rr^{y_i}$. But, since all columns $\{rr^{y_1},\ldots,rr^{y_m}\}\setminus rr^{y_i}$ come before $rr^{y_i}$ in $A^{\leq x}$, the column $rr^{y_i}$ must have been reduced to zero in {\sc Reduce}. Therefore, at the end of {\sc Reduce} it is deleted from $p_\ell^2$ contradicting $rr^{y_i}\in RR_\ell$.      

We conclude that $y_i<x$ for all $1\leq i\leq m$ and $rr^{y_1}+\cdots+rr^{y_m}\in \mathrm{Rad}(RR_\ell)(x)$. Therefore, $\ker(p_\ell^2)\subseteq \mathrm{Rad}(RR_\ell)$ and $p_\ell^2$ is a projective cover of $\ker(p_\ell^1)$.
\end{proof}

\begin{theorem} \label{thm:minimal_resolution}
Let $p_\ell^1$ and $p_\ell^2$ be the matrices computed by the algorithm {\sc PiRep}. Then the sequence
\begin{equation*}
\begin{tikzcd}
0 & C_\ell(K) \arrow[l] & G_\ell \arrow[l,swap,"\alpha_\ell"] & R_\ell \arrow[l,swap,"p_\ell^1"] & RR_\ell \arrow[l,swap,"p_\ell^2"]
\end{tikzcd}
\end{equation*}
is an asymptotically minimal projective resolution of $C_\ell(K)$ up to the second term. If {\sc Reduce} is called in {\sc RelRel}, then it is a minimal projective resolution of $C_\ell(K)$ up to the second term.
\end{theorem}

\begin{proof}
By construction $G_\ell$, $R_\ell$ and $RR_\ell$ are projective.
By Proposition \ref{prop:alpha_epi}, \ref{prop:p_exact_2}, \ref{prop:exact_at2_1} and \ref{prop:exact_at2_2} the sequence is exact. Hence, it is a projective resolution up to the second term. By Theorem \ref{thm:p1min}, $p_\ell^1$ is a minimal presentation of $C_\ell(K)$ and, by Proposition \ref{prop:asymptotic_optimality}, $p_\ell^2$ is asymptotically minimal. If {\sc Reduce} is called in {\sc RelRel}, then, by Proposition \ref{prop:proj_cover_p2}, $p_\ell^2$ is a projective cover of $\ker(p_\ell^1)$. Therefore, by Proposition \ref{prop:minimal_eqq_projective_cover}, together with $p_\ell^1$ it forms a minimal projective resolution up to the second term.
\end{proof}

\subsection{Correctness of computing $f_\ell^1$ and $\vartheta_{\ell+1}$} \label{subsec:graph_equations}

After proving the correctness of the computation of $p_\ell^1$, $p_\ell^2$, and $f^0_\ell$ by the {\sc PiRep} algorithm, the only missing ingredients for a PiRep of $H_\ell(K)$ are the matrices for $f^1_{\ell}$ and $\vartheta_{\ell+1}$ in the block matrices in \eqref{eq:projective_complex_homology}. The morphisms $f^1_{\ell}$ and $\vartheta_{\ell+1}$ need to satisfy $p_{\ell\minus 1}^1\circ f^1_{\ell}=f^0_\ell\circ p_\ell^1$ and $p_{\ell\minus 1}^1\circ \vartheta_{\ell+1}=f^0_\ell\circ f^0_{\ell+1}$, respectively. The routine {\sc Lift} efficiently solves for the unknown maps $f_\ell^1$ and $\vartheta_{\ell+1}$ in these equations by exploiting the special structure of $p_{\ell\minus 1}^1$. It solves a set of linear systems of equations with coefficient matrices induced by the edges of certain spanning trees. In this section, we prove that we can reduce the problem of solving these equations of morphisms of projective modules to such linear systems, which will then establish the correctness of the routine {\sc Lift}.  

Both the construction of $f^1_{\ell}$ and $\vartheta_{\ell+1}$ can be described as finding a solution of a $P$-graded linear system $A\circ X=B$, as defined in Definition \ref{def:P_graded_system}. For the unknown $X=f^1_{\ell}$, we set $A=p_{\ell\minus 1}^1$ and $B=f^0_\ell\circ p_\ell^1$. For the unknown $X=\vartheta_{\ell+1}$, we set $A=p_{\ell\minus 1}^1$ and $B=f^0_\ell\circ f^0_{\ell+1}$. As discussed in Section \ref{sec:pirep_theory}, Proposition \ref{prop:existence_resolutions}(2) and (3) imply that these systems have solutions. By Proposition \ref{prop:eliminate_constraints}, we can solve a $P$-graded linear system by solving $m$ standard linear systems $A^{\leq y_j}X^{\leq y_j}_j=B_j$ for submatrices $A^{\leq y_j}$ of $A$. In general, solving the $m$ linear systems in Proposition \ref{prop:eliminate_constraints} would cost cubic time per system, but, in our case, we can exploit the special structure of the coefficient matrix $A=p_{\ell\minus 1}^1$ to solve them in linear time. By Theorem \ref{thm:graph_presentation_matrix}, we obtain $p_{\ell\minus 1}^1$ from a matrix $\overline{p}_{\ell\minus 1}^1$ by removing a distinguished row $\Omega$ where $\overline{p}_{\ell\minus 1}^1$ is the boundary matrix of a $\overline{P}$-filtered graph $\mathcal{G}\colon \overline{P}\rightarrow \Delta\mathbf{Cpx}_{\leq 1}$. This means that our coefficient matrix $A$ has either one or two non-zero entries per column and its columns are labeled by the edges in the $\overline{P}$-filtered graph. First, we show that we can directly work with the coefficient matrix induced by $\overline{p}_{\ell\minus 1}^1$.  

\begin{proposition} \label{prop:linear_system_to_graph}
Let $A=\begin{pmatrix} T & S \end{pmatrix}$ be a block matrix, where $T$ is a block of columns with two non-zero entries and $S$ is a block of columns with one non-zero entry. Let $\overline{A}\coloneqq \begin{pmatrix}T & S \\ \overline{0} & \overline{1} \end{pmatrix}$ be a block matrix, where $\overline{0}=\begin{pmatrix} 0 & \cdots & 0 \end{pmatrix}$ and $\overline{1}=\begin{pmatrix} 1 & \cdots & 1 \end{pmatrix}$. Moreover, let $\overline{B}\coloneqq \begin{pmatrix}B \\ \overline{1}B\end{pmatrix}$. Then $AX=B$ if and only if $\overline{A}X=\overline{B}$.
\end{proposition}

\begin{proof}
Since $A$ and $B$ are obtained by removing the last row of $\overline{A}$ and $\overline{B}$, respectively, it is clear that a solution $X$ of $\overline{A}X=\overline{B}$ is also a solution of $AX=B$. Conversely, let $X=\begin{pmatrix}X_T \\ X_S\end{pmatrix}$ be a solution of $AX=\begin{pmatrix} T & S \end{pmatrix}\begin{pmatrix}X_T \\ X_S\end{pmatrix}=T X_T + S X_S=B$. Because $T$ has exactly two non-zero entries per column and $S$ has exactly one non-zero entry per column, we obtain $\overline{1}B=\overline{1}T X_T + \overline{1}S X_S=\overline{0}X_T + \overline{1}X_S=\overline{1}X_S$. Therefore, $\overline{A}X=\begin{pmatrix}T & S \\ \overline{0} & \overline{1} \end{pmatrix}\begin{pmatrix}X_T \\ X_S\end{pmatrix}=\begin{pmatrix}T X_T + S X_S \\ \overline{1}X_S\end{pmatrix}=\begin{pmatrix}B \\ \overline{1}B\end{pmatrix}=\overline{B}$. 
\end{proof}

By Definition \ref{def:boundary_p}, if $A=p_{\ell\minus 1}^1$, we obtain $\overline{A}=\overline{p}_{\ell\minus 1}^1$ for the matrix in Proposition \ref{prop:linear_system_to_graph}. So we can consider the system $\overline{A}X=\overline{B}$, where $\overline{A}=\overline{p}_{\ell\minus 1}^1$ and $\overline{B}$ is constructed from $B=f^0_\ell\circ p_\ell^1$ or $B=f^0_\ell\circ f^0_{\ell+1}$ as in Proposition \ref{prop:linear_system_to_graph}. For $\overline{A}=\overline{p}_{\ell\minus 1}^1$, the submatrices $\overline{A}^{\leq y_j}$ of Proposition \ref{prop:eliminate_constraints} are exactly the boundary matrices of $\mathcal{G}(y_j)$.  The following proposition shows that it is enough to solve the system $\overline{A}^{\leq y_j}X^{\leq y_j}_j=\overline{B}_j$ for a submatrix of $\overline{A}^{\leq y_j}$ and a subvector of $X^{\leq y_j}_j$ corresponding to the edges of a spanning forest.

\begin{proposition} \label{prop:graph_equation}
Let $Ax=b$ be a system of linear equations such that $A$ is the boundary matrix of a graph $G$ (without self-loops) and $b\in \im(A)$. Let $F=\bigcup_{i=1}^n F^i$ be a spanning forest of $G$ with connected components $F^1,\ldots,F^n$. Moreover, let $A^i$ be the submatrix of $A$ corresponding to the edges and vertices in $F^i$, $x^i$ the subvector of variables $x$ corresponding to the edges in $F^i$, $x^r$ the subvector of $x$ corresponding to edges in $G\setminus F$, and $b^i$ the subvector of $b$ corresponding to the vertices in $F^i$. Then $b^i\in\im(A^i)$ and $A^ix^i=b^i$ and $x^r=0$ implies $Ax=b$.  
\end{proposition}

\begin{proof}
Let $G^1,\ldots,G^n$ be the connected components of the graph $G$. Let $\hat{A}^i$ be the submatrix of $A$ corresponding to the vertices and edges in $G^i$, $\hat{x}^i$ the subvector of variables $x$ corresponding to the edges in $G^i$, and $\hat{b}^i$ the subvector of $b$ corresponding to the vertices in $G^i$. After reordering the columns and rows of $A$, $x$ and $b$ we can write the system as
\begin{equation*}
Ax=
\begin{pmatrix}
\hat{A}^1 & \cdots & 0 \\
\vdots & \ddots & \vdots \\
0 & \cdots & \hat{A}^n
\end{pmatrix}
\begin{pmatrix}
\hat{x}^1 \\
\vdots \\
\hat{x}^n \\
\end{pmatrix}=
\begin{pmatrix}
\hat{A}^1\hat{x}^1 \\
\vdots \\
\hat{A}^n\hat{x}^n \\
\end{pmatrix}=
\begin{pmatrix}
\hat{b}^1 \\
\vdots \\
\hat{b}^n 
\end{pmatrix}=b ,
\end{equation*}
where the coefficient matrix is in block-diagonal form because the connected components $G^1,\ldots,G^n$ of $G$ partition the vertices and edges into a disjoint union. Therefore, $Ax=b$ has a solution if and only if $\hat{A}^i\hat{x}^i=\hat{b}^i$ has a solution for all $1\leq i\leq n$. In other words, $b\in\im(A)$ if and only if $\hat{b}^i\in\im(\hat{A}^i)$. 

Since $F$ is a spanning forest of $G$, each connected component $F^i$ of $F$ must be a spanning tree of some connected component of $G$. W.l.o.g.\ assume $F^i$ is a spanning tree of the component $G^i$. Hence, $A^i$ is the submatrix of $\hat{A}^i$ corresponding to the edges in $F^i$. If $C^i$ is the submatrix of $\hat{A}^i$ corresponding to the edges in $G^i\setminus F^i$, then we can write $\hat{A}^i=\begin{pmatrix}A^i & C^i \end{pmatrix}$. Every edge in $G^i\setminus F^i$ closes a cycle in $G^i$ with the unique path connecting its endpoints in the spanning tree $F^i$. Thus, every column of $C^i$ is a linear combination of columns in $A^i$, given by the columns corresponding to the edges on this path. This implies that $\im(A^i)=\im(\hat{A}^i)$ and, since $G^i$ and $F^i$ have the same vertices, $b^i=\hat{b}^i\in \im(\hat{A}^i)=\im(A^i)$.  

After reordering the columns and rows of $A$, $x$ and $b$ again, we can write the system as
\begin{equation*}
Ax=
\begin{pmatrix}
A^1 & \cdots & 0 & C^1 & \cdots & 0 \\
\vdots & \ddots & \vdots & \vdots & \ddots & \vdots \\
0 & \cdots & A^n & 0 & \cdots & C^n
\end{pmatrix}
\begin{pmatrix}
x^1 \\
\vdots \\
x^n \\
x^r
\end{pmatrix}=
\begin{pmatrix}
b^1 \\
\vdots \\
b^n 
\end{pmatrix}=b ,
\end{equation*}
where $C^1,\ldots, C^n$ are the submatrices and $x^r$ is the subvector corresponding to all edges in $G\setminus F$. Since the columns of $A^1,\ldots,A^n$ correspond to edges in the connected components of the spanning forest $F$, the block of $A$ containing them is block-diagonal. Therefore, $A^1x^1=b^1,\ldots, A^nx^n=b^n$ and $x^r=0$ implies $Ax=b$.  
\end{proof}

By Proposition \ref{prop:graph_equation}, we can reduce the problem of solving $\overline{A}^{\leq y_j}X_j^{\leq y_j}=\overline{B}_j$ to the problem of solving equations $F^kx=b^k$ where the $F^k$ correspond to the connected components of a spanning forest of $\mathcal{G}(y_j)$. In Appendix \ref{app:tree_solve}, we discuss the routine {\sc TreeSolve}, which can solve such a linear system $F^kx=b^k$ in linear time. The routine {\sc Lift} computes solutions of $p_{\ell\minus 1}^1\circ f^1_{\ell}=f^0_\ell\circ p_\ell^1$ and $p_{\ell\minus 1}^1\circ \vartheta_{\ell+1}=f^0_\ell\circ f^0 _{\ell+1}$ by solving the systems $F^kx=b^k$ using the spanning forests $\mathcal{F}^{y_j}_{act}$ of $\mathcal{G}(y_j)$ maintained by the {\sc PiRep} algorithm and the routine {\sc TreeSolve}.

\begin{theorem} \label{thm:GraphSolve_correct}
The routine {\sc Lift} computes morphisms $f^1_\ell$ and $\vartheta_{\ell+1}$ such that $p_{\ell\minus 1}^1\circ f^1_\ell=f^0_\ell\circ p_\ell^1$ and $p_{\ell\minus 1}^1\circ\vartheta_{\ell+1}=f^0_\ell\circ f^0_{\ell+1}$.   
\end{theorem}

\begin{proof}
By Lemma \ref{prop:spanning_forest}, for each $x\in P$, after the call of {\sc Generator}, $\mathcal{F}^x_{act}$ is a spanning forest of $\mathcal{G}(x)$ and all updates of $p^1$ and $f^0$ at grade $x$ have already been computed. By Proposition \ref{prop:linear_system_to_graph}, we can replace the coefficient matrix $p_{\ell\minus 1}^1$ in $p_{\ell\minus 1}^1\circ f^1_\ell=f^0_\ell\circ p_\ell^1$ and $p_{\ell\minus 1}^1\circ\vartheta_{\ell+1}=f^0_\ell\circ f^0_{\ell+1}$ by $\overline{p}_{\ell\minus 1}^1$ if we also add an additional row $\overline{1}\circ f^0_\ell\circ p_\ell^1$ and $\overline{1}\circ f^0_\ell\circ f^0_{\ell+1}$, corresponding to the row $\Omega$ in $\overline{p}_{\ell\minus 1}^1$, to the right-hand sides. 

A matrix representation of the morphism $f^1_\ell\colon R_\ell\rightarrow R_{\ell\minus 1}$ has columns indexed by the relations $r^y$ spanning $R_\ell$ and rows indexed by the relations spanning $R_{\ell\minus 1}$. Let $r^{y_1},\ldots,r^{y_s}$ be the relations spanning $R_\ell$. Then we can rewrite the equation $\overline{p}_{\ell\minus 1}^1\circ f^1_\ell=\begin{pmatrix}f^0_\ell\circ p_\ell^1 \\ \overline{1}\circ f^0_\ell\circ p_\ell^1\end{pmatrix}$ as $\overline{p}_{\ell\minus 1}^1\circ f^1_\ell[:,r^{y_j}]=\begin{pmatrix}f^0_\ell\circ p_\ell^1[:,r^{y_j}] \\ \overline{1}\circ f^0_\ell\circ p_\ell^1[:,r^{y_j}]\end{pmatrix}$ for all $1\leq j\leq s$. By Proposition \ref{prop:eliminate_constraints}, we can drop all columns of $\overline{p}_{\ell\minus 1}^1$ and rows of $f^1_\ell[:,r^{y_j}]$ which are indexed by $r^z\in R_{\ell\minus 1}$ such that $z\nleq y_j$. The remaining columns of $\overline{p}_{\ell\minus 1}^1$ correspond exactly to the edges in $\mathcal{G}(y_j)$. The equation becomes an ordinary system of linear equations where the coefficient matrix is the boundary matrix of $\mathcal{G}(y_j)$, the variables correspond to the edges in $\mathcal{G}(y_j)$ and the right-hand side corresponds to the vertices of $\mathcal{G}(y_j)$. At the grade $x=y_j$ in the algorithm, where we add $r^{y_j}$, all $r^z$ in $R_{\ell\minus 1}$ with $z\leq y_j$ are already present. Therefore, all ingredients of the linear system at $y_j$ are available. By Proposition \ref{prop:graph_equation}, we can solve such a system by restricting the coefficient matrix to the submatrix induced by a spanning forest and instead solve the system for each connected component of the forest. 

In step $1(a)$, we compute the right-hand side with the additional row corresponding to $\Omega$. Note that all non-zero entries in the vector $f^0_\ell\circ p_\ell^1[:,r^{y_j}]$ correspond to $g_\sigma^z$ with $z\leq y_j$, i.e., to vertices in $\mathcal{F}_{act}^{y_j}$. In step $1(b)$, we call the routine {\sc TreeSolve} (see Appendix \ref{app:tree_solve}) for each connected component $C$ of the spanning forest $\mathcal{F}_{act}^{y_j}$. By Theorem \ref{prop:tree_solve}, this yields a solution of $\overline{p}_{\ell\minus 1}^1\circ f^1_\ell[:,r^{y_j}]=\begin{pmatrix}f^0_\ell\circ p_\ell^1[:,r^{y_j}] \\ \overline{1}\circ f^0_\ell\circ p_\ell^1[:,r^{y_j}]\end{pmatrix}$ and, thus, also of $p_{\ell\minus 1}^1\circ f^1_\ell[:,r^{y_j}]=f^0_\ell\circ p_\ell^1[:,r^{y_j}]$. After processing all grades of $P$, we obtain a solution of $p_{\ell\minus 1}^1\circ f^1_\ell=f^0_\ell\circ p_\ell^1$. The proof for $p_{\ell\minus 1}^1\circ\vartheta_{\ell+1}=f^0_\ell\circ f^0_{\ell+1}$ is analogous. 
\end{proof}

\section{Conclusion} \label{sec:conclusion}

We have developed an algorithm that converts the simplicial data of a poset
tower into a projective implicit representation of its persistent homology.
The main difficulty is that the chain modules of a poset tower are not
necessarily projective, so the boundary maps cannot, in general, be
read directly as graded matrices. Our construction overcomes this by first
computing partial projective resolutions of the chain modules, then lifting the
boundary maps to these resolutions, and finally assembling the resulting
data into a PiRep using an additional correction term. In particular, 
the construction provides efficient algorithms for
computing minimal presentations of the chain modules $C_\ell(K)$ arising
from poset towers, as well as asymptotically minimal projective resolutions
of these modules up to the second term.

The central observation behind the algorithm is that the relations in the
chain modules are not arbitrary. They arise from identifications of simplex
generators and can be organized in relation graphs. This graph structure is
used repeatedly: to compute minimal presentations by avoiding redundant
cycle-closing relations, to compute relations among relations from cycles,
and to solve the linear systems needed for the lifted boundary maps and the
correction term. Thus, rather than applying general-purpose algebraic reduction to the chain
modules, the algorithm exploits the combinatorial structure inherited from
the poset tower, leading to a significant improvement in efficiency.

Several directions remain open. Although the algorithm is general for finite posets, indexing
posets arising in applications might have additional structure. We expect that 
the general algorithm developed here can serve as a starting point for more specialized algorithms
and for practical implementations of persistent homology over more general
indexing posets. 

Moreover, one may ask whether the local construction carried out here for a
fixed homological degree can be extended to a more global replacement of the
persistent chain complex $C_\bullet(K)$ by a chain complex of projective
modules. Such a construction is unlikely to admit an efficient
general-purpose solution for arbitrary finite posets. Nevertheless, for
special classes of indexing posets, it would be interesting to investigate
whether the combinatorial structure of poset towers can be exploited to
construct projective resolutions, or partial projective resolutions, of the
persistent chain complex in a systematic way.

\bibliography{lib}

\appendix

\section{Missing proofs from Section \ref{sec:background}} \label{app:proofs_sec_background}

\begin{proof}[Proof of Proposition \ref{prop:elementary_projective}]
For the first claim, let $\proj[x]\xrightarrow{\phi}\proj[y]$ be a morphism of $P$-persistence modules. 
\begin{equation} \label{eq:proj_proof}
\begin{tikzcd}
\proj[x](x) \arrow[r,"\phi(x)"] \arrow[d,swap,"\proj\!\text{[}x\text{]}(x\leq z)"] & \proj[y](x) \arrow[d,"\proj\!\text{[}y\text{]}(x\leq z)"] \\
\proj[x](z) \arrow[r,"\phi(z)"] & \proj[y](z)
\end{tikzcd}
\end{equation}
By Definition \ref{def:elementary_module}, if $\phi(x)$ in \eqref{eq:proj_proof} is non-zero, then $\proj[y](x)\cong \mathbb{F}_2$ and, thus, $y\leq x$. Moreover, by commutativity, we get that $\phi(z)$ is non-zero for all $z\geq x$. Conversely, if $\phi(x)$ is zero, then, again by commutativity, $\phi(z)$ has to be zero for all $z\geq x$. Hence, the morphism $\phi$ is completely determined by $\phi(x)$ and $\phi(x)$ can only be non-zero if $y\leq x$.

The second claim is a direct consequence of the first claim and the fact that in abelian categories $\Hom$ commutes with finite direct sums in both arguments (\cite{maclane}).
\end{proof}

\begin{proof}[Proof of Proposition \ref{prop:eliminate_constraints}]
Writing 
\begin{equation*}
AX=A(X_1,\ldots,X_m)=(AX_1,\ldots,AX_m)=(B_1,\ldots,B_m)=B,
\end{equation*}
the problem reduces to solving $m$ constrained linear systems of equations of the form $AX_j=B_j$ and $X_{kj}=0$ if $z_k\nleq y_j$. For each $1\leq j\leq m$, after reordering, we can write $A=\begin{pmatrix}A^{\leq y_j} & A^{\nleq y_j}\end{pmatrix}$ and $X_j=\begin{pmatrix}X_j^{\leq y_j} \\ X_j^{\nleq y_j} \end{pmatrix}$. Then the system becomes 
\begin{equation*}
AX_j=\begin{pmatrix}A^{\leq y_j} & A^{\nleq y_j}\end{pmatrix}\begin{pmatrix}X^{\leq y_j}_j \\ X_j^{\nleq y_j} \end{pmatrix}=A^{\leq y_j}X^{\leq y_j}_j+A^{\nleq y_j}X_j^{\nleq y_j}=B_j \hspace{3pt}\text{ and }\hspace{3pt} X_j^{\nleq y_j}=0.
\end{equation*}
Therefore, we can equivalently solve the unconstrained system $A^{\leq y_j}X^{\leq y_j}_j=B_j$ for all $1\leq j\leq m$.   
\end{proof}

\begin{proof}[Proof of Proposition \ref{prop:minimal_eqq_projective_cover}]
Let $Q_\bullet \rightarrow M$ be a projective resolution such that $q_i \colon Q_i \rightarrow \ker(q_{i\minus 1})$ is a projective cover for all $i\geq 0$. By Corollary 5.10 in \cite{assem2006elements}, for a $P$-persistence module, which corresponds to a module over the finite-dimensional incidence algebra of $P$, such a projective resolution always exists and is unique up to isomorphism. Moreover, by Proposition 2.1 in \cite{nickel2010noncommutative}, it is a direct summand of every projective resolution of $M$. Thus, if $U_\bullet \rightarrow M$ is any projective resolution, then $U_\bullet \cong Q_\bullet \oplus W_\bullet$ for some complex $W_\bullet$. In particular, for each $i \geq 0$, we have $U_i \cong Q_i \oplus W_i$. Since projective modules decompose into finite direct sums of indecomposable projectives, it follows that the multiplicity of each indecomposable projective in $Q_i$ is bounded above by its multiplicity in $U_i$. Hence, $Q_i$ has minimal indecomposable-projective multiplicities among all projective resolutions of $M$. Conversely, if $U_\bullet$ has minimal indecomposable-projective multiplicities in each degree, then from $U_i \cong Q_i \oplus W_i$ we get $W_i = 0$ for all $i \geq 0$. Therefore, $U_\bullet \cong Q_\bullet$.

Any projective presentation can be extended to a projective resolution and, by Corollary 5.10 in \cite{assem2006elements}, a projective presentation satisfying the projective cover condition can be extended to a projective resolution satisfying the projective cover condition. Therefore, the result for presentations directly follows from the above argument.
\end{proof}

\section{Missing proofs from Section \ref{sec:pirep_theory}} \label{app:proofs}

The following proofs are based on diagram chasing. Note that this is legitimate as persistence modules can be viewed as actual modules over the incidence algebra of the poset.

\begin{proof}[Proof of Proposition \ref{prop:kernel_quotient}]
Suppose $x\in \ker(\partial_\ell)$. Since $\alpha_\ell$ is an epimorphism, there exists $y\in G_\ell$ such that $\alpha_\ell(y)=x$. By commutativity, we get $\alpha_{\ell\minus 1}\circ f^0_\ell(y)=\partial_\ell\circ \alpha_\ell(y)=0$ and by exactness of the rows $f^0_\ell(y)\in \ker(\alpha_{\ell\minus 1})=\im(p_{\ell\minus 1}^1)$. Hence, there is a $z\in R_{\ell\minus 1}$ such that $f^0_\ell(y)=p^1_{\ell\minus 1}(z)$. Phrased differently, $(y,z)\in \ker\begin{pmatrix}f^0_\ell & p_{\ell\minus 1}^1\end{pmatrix}$ and $\begin{pmatrix}\alpha_\ell & 0\end{pmatrix}(y,z)=\alpha_\ell(y)=x$. Hence, $\begin{pmatrix}\alpha_\ell & 0\end{pmatrix}$ is an epimorphism. 

We have $\begin{pmatrix}\alpha_\ell & 0\end{pmatrix}\circ\begin{pmatrix}p_\ell^1 & 0 \\ f_\ell^1 & p_{\ell\minus 1}^2\end{pmatrix}=\begin{pmatrix}\alpha_\ell\circ p_\ell^1 & 0\end{pmatrix}=0$. So $\im\begin{pmatrix}p_\ell^1 & 0 \\ f^1_\ell & p_{\ell\minus 1}^2\end{pmatrix}\subseteq \ker\begin{pmatrix}\alpha_\ell & 0\end{pmatrix}$.

Let $(x,y)\in \ker\begin{pmatrix}f^0_\ell & p_{\ell\minus 1}^1\end{pmatrix}$ such that $\begin{pmatrix}\alpha_\ell & 0\end{pmatrix}(x,y)=\alpha_\ell(x)=0$. Then $x\in\ker(\alpha_\ell)=\im(p^1_\ell)$ and there exists $a\in R_\ell$ such that $p^1_\ell(a)=x$. Moreover, we have 
\begin{equation*}
\begin{aligned}
\begin{pmatrix}f^0_\ell & p_{\ell\minus 1}^1\end{pmatrix}(x,y)&=f^0_\ell(x)+p_{\ell\minus 1}^1(y)=f^0_\ell\circ p_\ell^1(a)+p^1_{\ell\minus 1}(y)\\&=p_{\ell\minus 1}^1\circ f^1_\ell(a)+p_{\ell\minus 1}^1(y)=p_{\ell\minus 1}^1\big(f^1_\ell(a)+y\big)=0
\end{aligned}
\end{equation*}
Hence, $f^1_\ell(a)+y\in\ker(p^1_{\ell\minus 1})=\im(p^2_{\ell\minus 1})$ and there exists $b\in RR_{\ell\minus 1}$ such that $y=f^1_\ell(a)+p^2_{\ell\minus 1}(b)$. We conclude that $\begin{pmatrix}p_\ell^1 & 0 \\ f^1_\ell & p_{\ell\minus 1}^2\end{pmatrix}(a,b)=(x,y)$ and $\im\begin{pmatrix}p_\ell^1 & 0 \\ f^1_\ell & p_{\ell\minus 1}^2\end{pmatrix}=\ker\begin{pmatrix}\alpha_\ell & 0\end{pmatrix}$.
\end{proof}

\vspace{5pt}

\begin{proof}[Proof of Proposition \ref{prop:homology_quotient}]
In Proposition \ref{prop:kernel_quotient}, we proved that $\begin{pmatrix}\alpha_\ell & 0\end{pmatrix}\colon \ker\begin{pmatrix}f^0_\ell & p_{\ell\minus 1}^1\end{pmatrix}\rightarrow \ker(\partial_\ell)$ is an epimorphism. Hence, $\pi\circ\begin{pmatrix}\alpha_\ell & 0\end{pmatrix}$ is an epimorphism. Since
\begin{equation*}
\begin{aligned}
\pi\circ\begin{pmatrix}\alpha_\ell & 0\end{pmatrix}\circ \begin{pmatrix}f^0_{\ell+1} & p^1_{\ell} & 0 \\ \vartheta_{\ell+1} & f^1_\ell & p^2_{\ell\minus 1}\end{pmatrix}&=\pi\circ \begin{pmatrix}\alpha_\ell\circ f^0_{\ell+1} & 0 & 0 \end{pmatrix}\\&=\begin{pmatrix}\pi\circ\partial_{\ell+1}\circ\alpha_{\ell+1} & 0 & 0 \end{pmatrix}=0
\end{aligned}
\end{equation*}
we have $\im\begin{pmatrix}f^0_{\ell+1} & p^1_{\ell} & 0 \\ \vartheta_{\ell+1} & f^1_\ell & p^2_{\ell\minus 1}\end{pmatrix}\subseteq\ker\!\big(\pi\circ\begin{pmatrix}\alpha_\ell & 0\end{pmatrix}\big)$.

Suppose $(x,y)\in \ker\begin{pmatrix}f^0_\ell & p_{\ell\minus 1}^1\end{pmatrix}$ such that $\pi\circ \begin{pmatrix}\alpha_\ell & 0\end{pmatrix}(x,y)=\pi\big(\alpha_\ell(x)\big)=0$. Then there exists $z\in C_{\ell+1}$ such that $\partial_{\ell+1}(z)=\alpha_\ell(x)$. Moreover, since $\alpha_{\ell+1}$ is an epimorphism, there exists $a\in G_{\ell+1}$ such that $\alpha_{\ell+1}(a)=z$. Using commutativity, we obtain $\alpha_\ell\circ f^0_{\ell+1}(a)=\partial_{\ell+1}\circ \alpha_{\ell+1}(a)=\alpha_\ell(x)$. This implies $\alpha_\ell\big(f^0_{\ell+1}(a)+x\big)=0$, so $f^0_{\ell+1}(a)+x\in\ker\alpha_\ell=\im(p^1_\ell)$ and there exists $b\in R_\ell$ such that $x=f^0_{\ell+1}(a)+p^1_\ell(b)$. Since 
\begin{equation*}
\begin{aligned}
\begin{pmatrix}f^0_\ell & p_{\ell\minus 1}^1\end{pmatrix}(x,y)&=f^0_\ell(x)+p^1_{\ell\minus 1}(y)\\&=f^0_\ell\big(f^0_{\ell+1}(a)+p^1_\ell(b)\big)+p^1_{\ell\minus 1}(y)\\&=f^0_\ell\circ f^0_{\ell+1}(a)+f^0_\ell\circ p^1_\ell(b)+p^1_{\ell\minus 1}(y)\\&=p^1_{\ell\minus 1}\circ \vartheta_{\ell+1}(a)+p^1_{\ell\minus 1}\circ f^1_\ell(b)+p^1_{\ell\minus 1}(y)\\&=p^1_{\ell\minus 1}\big(\vartheta_{\ell+1}(a)+f^1_\ell(b)+y\big)=0
\end{aligned}
\end{equation*}
we obtain $\vartheta_{\ell+1}(a)+f^1_\ell(b)+y\in\ker(p^1_{\ell\minus 1})=\im(p^2_{\ell\minus 1})$ and there exists $c\in RR_{\ell\minus 1}$ such that $y=\vartheta_{\ell+1}(a)+f^1_{\ell}(b)+p^2_{\ell\minus 1}(c)$. We conclude that $\begin{pmatrix}f^0_{\ell+1} & p^1_{\ell} & 0 \\ \vartheta_{\ell+1} & f^1_\ell & p^2_{\ell\minus 1}\end{pmatrix}(a,b,c)=(x,y)$ and $\im\begin{pmatrix}f^0_{\ell+1} & p^1_{\ell} & 0 \\ \vartheta_{\ell+1} & f^1_\ell & p^2_{\ell\minus 1}\end{pmatrix}=\ker\!\big(\pi\circ\begin{pmatrix}\alpha_\ell & 0\end{pmatrix}\big)$.
\end{proof}

\section{TreeSolve algorithm} \label{app:tree_solve}

In this section, we discuss an algorithm {\sc TreeSolve} which efficiently solves linear systems of equations where the coefficient matrix has the structure of the boundary matrix of a tree. We will show that {\sc TreeSolve} solves such a linear system in linear time. It is based on the following definition of a leaf-order, which allows for an iterative substitution of the solution for edges connected to leaves in subgraphs. We call a vertex $v$ in a tree $T$ a leaf if it is incident to at most one edge $e$, denoted by $v<e$. Note that every leaf in a tree with more than one vertex has degree one.

\begin{definition}[Leaf-order] \label{def:leaf_order}
Let $T$ be a tree graph on $n$ vertices. We call an order $v_1,\ldots,v_n$ of the vertices of $T$ a leaf-order if $v_i$ is a leaf of $T\setminus\{v_1,\ldots,v_{i\minus 1}\}$ for all $1\leq i\leq n$.  
\end{definition}

A leaf-order can be found by the following routine. 

\vspace{0.1in}
\noindent
{\bf Algorithm {\sc LeafOrder}} ($T=(V,E)$)
\begin{enumerate}
    \item Initialize: $\text{deg}\colon V\rightarrow \mathbb{N}_0$, with $\text{deg}(v)=0$ for all $v\in V$, a queue $Q=\emptyset$, and a list $L=\emptyset$;
    \item For all $e=(u,v)\in E$, set $\text{deg}(u)=\text{deg}(u)+1$ and $\text{deg}(v)=\text{deg}(v)+1$;
    \item For all $v\in V$, if $\text{deg}(v)\leq 1$ add $v$ to $Q$;
    \item While $Q\neq\emptyset$ do 
    \begin{enumerate}
        \item Take $v$ from $Q$ and set $\text{deg}(v)=0$;
        \item Add $v$ to $L$;
        \item Set $\text{deg}(u)=\max\big(\text{deg}(u)-1,0\big)$ for all $u$ adjacent to $v$ in $T$;
        \item If $\text{deg}(u)=1$, for $u$ adjacent to $v$, add $u$ to $Q$
    \end{enumerate}
\end{enumerate}

\begin{proposition}\label{prop:leaf_order}
For a tree $T=(V,E)$ on $n$ vertices, the algorithm {\sc LeafOrder} computes a leaf-order $L$ in $O(n)$ time.
\end{proposition}

\begin{proof}
Step $2$ of the algorithm obviously computes the degree $\text{deg}(v)$ of $v$ in the graph $T$ for all $v\in V$.
After Step 3 of the algorithm $Q$ contains all leaves of $T$. Since every non-empty tree has a leaf, $Q\neq\emptyset$. At this point, either $T$ was a single vertex and $Q$ contains this single vertex of degree zero, or $\text{deg}(v)=1$ for all $v\in Q$ and, since $T$ is connected, $\text{deg}(v)>1$ for all $v\in V\setminus Q$. The case where $T$ is a single vertex is obvious. From now on assume that $T$ has at least two vertices. Let $v_1,\ldots,v_n$ be the order stored in the list $L$. We now show by induction on the iterations of step $4$ that, before each iteration of step $4$, $Q$ contains all leaves of $T\setminus \{v_1,\ldots,v_{i\minus 1}\}$ and $\text{deg}(u)$ equals the degree of the vertex $u$ in $T\setminus \{v_1,\ldots,v_{i\minus 1}\}$. Before the first iteration of step $4$ this is true as argued above. Now suppose this is true before the $i$-th iteration of step $4$. In iteration $i$, we remove the leaf $v_i$ of $T\setminus \{v_1,\ldots,v_{i\minus 1}\}$ from $Q$, add it to $L$, set $\text{deg}(v_i)=0$, and decrease the degree of every neighbor of $v_i$ in $T\setminus \{v_1,\ldots,v_{i\minus 1}\}$ by one. Thus, $\text{deg}(u)$ now equals the degree of every vertex $u$ in $T\setminus \{v_1,\ldots,v_{i}\}$. Let $u$ be a leaf of $T\setminus \{v_1,\ldots,v_{i}\}$. If $u$ was not a leaf in $T\setminus \{v_1,\ldots,v_{i\minus 1}\}$, then its degree has been reduced to one in iteration $i$ and $u$ was added to $Q$. Otherwise, $u$ was already a leaf in $T\setminus \{v_1,\ldots,v_{i\minus 1}\}$ and, thus, by assumption, it was already in $Q$. This proves the claim. Since we add a leaf of $T\setminus \{v_1,\ldots,v_{i\minus 1}\}$ to $L$ in each iteration $L$ is a leaf-order of $T$.

To initialize and compute $\text{deg}$ of $T$, we need to process all vertices and edges of $T$ once. Since a tree has $n-1$ edges, this is clearly $O(n)$. In the third step, we process all vertices which is again $O(n)$. Since each vertex is added and removed from $Q$ exactly once there are $n$ iterations in step $4$. If $\Gamma(v)$ denotes the set of neighbors of $v$, then the update of neighbor degrees is $O(\sum_{v\in V}\vert\Gamma(v)\vert)=O(\vert E\vert)=O(n)$. Hence, the overall worst case complexity is $O(n)$. 
\end{proof}

The following algorithm {\sc TreeSolve} solves a linear system of equations $Ax=b$, where $A$ is the boundary matrix of a tree $T$, in linear time by simply substituting the solution using a leaf-order. Each column of $A$ and row of $x$ corresponds to an edge in $T$ and each row of $A$ and $b$ corresponds to a vertex in $T$. The algorithm takes as input the tree $T=(V,E)$ and the right-hand side $b$ as a map $b\colon V\rightarrow \mathbb{F}_2$ representing the entries of the rows of $b$.

\vspace{0.1in}
\noindent
{\bf Algorithm {\sc TreeSolve}} ($T=(V,E),b\colon V\rightarrow\mathbb{F}_2$)
\begin{enumerate}
    \item Call {\sc LeafOrder}($T$) to compute a leaf-order $v_1,\ldots,v_n$ of $T$; 
    \item Initialize $x\colon E\rightarrow \{\bot\}\cup\mathbb{F}_2$ by $x(e)=\bot$ for all $e\in E$;
    \item For $i\coloneqq 1$ to $n-1$ do
    \begin{enumerate}
        \item Find $e\in E$ such that $v_i<e$ and $x(e)=\bot$;
        \item Set $x(e)=b(v_i)+\underset{v_i<e'\wedge x(e')\neq \bot}{\sum}x(e')$ (in $\mathbb{F}_2$)
    \end{enumerate}
\end{enumerate}

\begin{theorem} \label{prop:tree_solve}
Let $Ax=b$ be a system of linear equations such that $b\in\im(A)$ and $A$ is the boundary matrix of a tree $T$ on $n$ vertices. The algorithm {\sc TreeSolve} computes a solution of $Ax=b$ in $O(n)$ time.
\end{theorem}

\begin{proof}
Let $v_1,\ldots,v_n$ be the leaf-order computed in step $1$. By Definition \ref{def:leaf_order} and Proposition \ref{prop:leaf_order}, $v_i$ is a leaf of the tree $T_i\coloneqq T\setminus \{v_1,\ldots,v_{i\minus 1}\}$, obtained by iteratively removing the leaves $v_1,\ldots,v_{i\minus 1}$ from $T$. For every $1\leq i\leq n-1$, the leaf $v_i$ is incident to a unique edge $e_i$ in $T_i$ and removing the leaf $v_i$ from $T_i$ removes exactly the one edge $e_i$. In step $2$, we set $x(e)\coloneqq\bot$ for all $e\in T$ meaning that the values on all edges are undetermined. Let $A[v,:]$ denote the row of $A$ corresponding to the vertex $v\in T$. Then, viewing $x$ as a column vector, we get $A[v,:]x=\sum_{v<e\in E}x(e)$.

We now show by induction on $i$ that after each iteration of step $3$, the set of edges $E_{i+1}\coloneqq E(T_{i+1})=\{e\in E(T)\vert x(e)=\bot\}$ and $A[v_i,:]x=b(v_i)$. For $i=1$, the vertex $v_1$ is a leaf in $T$, there is exactly one edge $e_1$ incident to $v_1$ in $T$ and $x(e)=\bot$ for all $e\in E$. The algorithm picks $e_1$ and sets $x(e_1)=b(v_1)$. Thus, we get $A[v_1,:]x=b(v_1)$ and $x(e_1)\neq\bot$. Since $T_2$ is obtained from $T=T_1$ by removing the vertex $v_1$ and the edge $e_1$, we also get $E_2=\{e\in E(T)\vert x(e)=\bot\}$. Suppose the statement is true for all $i=1,\ldots,k-1$. By assumption, $E_k=\{e\in E(T)\vert x(e)=\bot\}$ before iteration $k$. Therefore, in iteration $k$, the algorithm picks the unique edge $e_k$ incident to the leaf $v_k$ in $T_k$. Every other edge $e'$ incident to $v_k$ in $T$ is not in $T_k$ anymore and thus, $x(e')\neq\bot$. In step $3b$, we set $x(e_k)=b(v_k)+\underset{v_k<e'\wedge x(e')\neq \bot}{\sum}x(e')$. Thus, after step $3b$, we get $A[v_k,:]x=\underset{v_k<e'\in E}{\sum}x(e')=\underset{v_k<e'\wedge x(e')\neq \bot}{\sum}x(e')=b(v_k)$. Moreover, $x(e_k)\neq\bot$ after step $3b$ implies $E_{k+1}=\{e\in E(T)\vert x(e)=\bot\}$.  

We conclude that $A[v_i,:]x=b(v_i)$ for all $1\leq i\leq n-1$. We now show that the final equation $A[v_n,:]x=b(v_n)$ is also satisfied. Let $\overline{1}=\begin{pmatrix}1&\cdots&1\end{pmatrix}$ and $\overline{0}=\begin{pmatrix}0&\cdots&0\end{pmatrix}$. Since $A$ has exactly two non-zero entries per column, we get $\overline{1}A=\sum_{i=1}^n A[v_i,:]=\overline{0}$. By assumption $b\in\im(A)$, thus, there exists a $y$ such that $Ay=b$, which implies $\overline{1}Ay=\overline{0}y=0=\overline{1}b$. In particular, $\sum_{i=1}^{n-1}b(v_i)=b(v_n)$. Therefore, we obtain $\overline{1}Ax=\sum_{i=1}^{n-1}A[v_i,:]x+A[v_n,:]x=\sum_{i=1}^{n-1}b(v_i)+A[v_n,:]x=b(v_n)+A[v_n,:]x=0$ and also the last equation is satisfied.

By Proposition \ref{prop:leaf_order}, step $1$ takes $O(n)$ time. Step $2$ obviously takes $O(n)$ time. Using an incidence list, checking the incident edges of all vertices and summing over all incident edges of all vertices is $O(\sum_{v\in V}\vert\Gamma(v)\vert)=O(\vert E\vert)=O(n)$. Therefore, in total, the worst-case complexity is $O(n)$.  
\end{proof}

\section{Computing presentations of homology from PiReps} \label{app:homology_presentation}

In this section, we discuss how to compute a presentation of the homology of a PiRep as in Definition \ref{def:pirep}. We consider a chain complex segment 
\begin{equation}\label{eq:segment_Q}
\begin{tikzcd}Q_{\minus 1} & Q_0 \arrow[l,swap,"q_0"] & Q_1 \arrow[l,swap,"q_1"] 
\end{tikzcd}
\end{equation}
where the $Q_i$ are projective and the maps $q_i$ are represented by $P$-graded matrices. We start by extending $Q_{\minus 1}\xleftarrow{q_0} Q_0$ to an exact sequence of projective modules:
\begin{equation}\label{eq:res_ker}
\begin{tikzcd}
Q_{\minus 1} & Q_0 \arrow[l,swap,"q_0"] & U_0 \arrow[l,swap,"u_0"] & U_1 \arrow[l,swap,"u_1"] .
\end{tikzcd}
\end{equation}
This can be done, for example, by using the algorithm\footnote{The algorithm {\sc MakeExact} in \cite{brown2024discretemicrolocalmorsetheory} is formulated in terms of injective resolutions but in our setting the projective case is completely dual.} {\sc MakeExact} in \cite{brown2024discretemicrolocalmorsetheory}.  Since, by exactness $\ker(q_0)=\im(u_0)$, $u_1$ is a presentation of $\ker(q_0)$. Since $\im(q_1)\subseteq \ker(q_0)$ we obtain the following diagram:
\begin{equation}\label{eq:ker_res_lift}
\begin{tikzcd}
& Q_1 \arrow[d,swap,"q_1"] \arrow[dr,dashed,"s"] \\
0 & \ker(q_0) \arrow[l] & U_0 \arrow[l,swap,"u_0"] & U_1 \arrow[l,swap,"u_1"]
\end{tikzcd}
\end{equation}
where $u_0$ is an epimorphism. Therefore, by the lifting property of projective modules, we get a lift $s$ of $q_1$ along $u_0$. This lift can be computed by solving a $P$-graded linear system, as in Definition \ref{def:P_graded_system}, using Proposition \ref{prop:eliminate_constraints}. We can interpret $U_0$ and $U_1$ as the generators and relations of the kernel and $Q_1$ as the initial relations on $Q_0$. We now have all the ingredients for a presentation of the homology. Let $\pi\colon \ker(q_0)\rightarrow \ker(q_0)/\im(q_1)$ denote the projection to the cokernel.

\begin{theorem} \label{prop:ker_res_lift}
The following exact sequence:
\begin{equation} \label{eq:final_presentation_from_pirep}
\begin{tikzcd}[ampersand replacement=\&,every label/.append style = {font = \small}]
0 \& \ker(q_0)/\im(q_1) \arrow[l] \&[15pt] U_0 \arrow{l}[swap]{\pi\circ u_0} \&[25pt] U_1\oplus Q_1 \arrow{l}[swap]{\begin{pmatrix}u_1 & s\end{pmatrix}}
\end{tikzcd}
\end{equation}
is a presentation of the homology of \eqref{eq:segment_Q}.
\end{theorem}

\begin{proof}
Since $u_0$ and $\pi$ are epimorphisms, $\pi\circ u_0$ is also an epimorphism. 

Since $\pi\circ u_0\circ \begin{pmatrix}u_1 & s\end{pmatrix}=\begin{pmatrix}\pi\circ u_0\circ u_1 & \pi\circ u_0\circ s\end{pmatrix}=\begin{pmatrix}0 & \pi\circ q_1\end{pmatrix}=0$, we have $\im\begin{pmatrix}u_1 & s\end{pmatrix}\subseteq\ker(\pi\circ u_0)$.

Let $x\in \ker(\pi\circ u_0)$. Then $\pi\circ u_0(x)=0$ implies that there exists $y\in Q_1$ such that $u_0(x)=q_1(y)=u_0\circ s(y)$. Since $u_0\big(x+s(y)\big)=0$ and $x+s(y)\in\ker(u_0)=\im(u_1)$, there exists $z\in U_1$ such that $x=u_1(z)+s(y)$. We conclude that $\begin{pmatrix}u_1 & s\end{pmatrix}(z,y)=x$ and $\im\begin{pmatrix}u_1 & s\end{pmatrix}=\ker(\pi\circ u_0)$.
\end{proof}

The presentation of Theorem \ref{prop:ker_res_lift} is not necessarily minimal. In \cite{brown2024discretemicrolocalmorsetheory}, the authors also discuss methods to minimize resolutions that can be used to minimize the presentation \eqref{eq:final_presentation_from_pirep}.

\end{document}